\documentclass[a4paper,twoside]{amsart}

\usepackage{amscd,amsthm,amsfonts,amssymb,amsmath,esint}
\usepackage{mathrsfs}
\usepackage[colorlinks=true,linkcolor=black]{hyperref}
\usepackage[margin=1in, headheight=14.5pt]{geometry} 
\usepackage{fancyhdr}
\usepackage[all]{xy}
\usepackage{CJK}
\usepackage{tikz-cd}

\newcommand{\un}{\mathrm{un}}
\newcommand{\tot}{\mathrm{tot}}
\newcommand{\rig}{\mathrm{rig}}
\newcommand{\HT}{\mathrm{HT}}
\newcommand{\wt}{\mathrm{wt}}

    \newcommand{\BA}{{\mathbb {A}}} 
    \newcommand{\BC}{{\mathbb {C}}} 
     \newcommand{\BF}{{\mathbb {F}}}
    \newcommand{\BG}{{\mathbb {G}}} 
     
     \newcommand{\BL}{{\mathbb {L}}}
     \newcommand{\BN}{{\mathbb {N}}}
     
    \newcommand{\BQ}{{\mathbb {Q}}} \newcommand{\BR}{{\mathbb {R}}}
     \newcommand{\BT}{{\mathbb {T}}}

     \newcommand{\BZ}{{\mathbb {Z}}}

     \newcommand{\CB}{{\mathcal {B}}}
    \newcommand{\CC}{{\mathcal {C}}} \renewcommand{\CD}{{\mathcal {D}}}
    \newcommand{\CE}{{\mathcal {E}}} \newcommand{\CF}{{\mathcal {F}}}
     \newcommand{\CH}{{\mathcal {H}}}
    \newcommand{\CI}{{\mathcal {I}}} 
     
    \newcommand{\CM}{{\mathcal {M}}} 
    \newcommand{\CO}{{\mathcal {O}}}

    \newcommand{\CU}{{\mathcal {U}}} \newcommand{\CV}{{\mathcal {V}}}
    \newcommand{\CW}{{\mathcal {W}}} 
     \newcommand{\CZ}{{\mathcal {Z}}}

    \newcommand{\fc}{{\mathfrak{c}}} 
    \newcommand{\fe}{{\mathfrak{e}}}

     \newcommand{\fp}{{\mathfrak{p}}}
    \newcommand{\fq}{{\mathfrak{q}}}

    \newcommand{\fI}{{\mathfrak{I}}} 
     
    \newcommand{\fM}{{\mathfrak{M}}} 
     \newcommand{\fP}{{\mathfrak{P}}}

    \newcommand{\fU}{{\mathfrak{U}}} 
     \newcommand{\fX}{{\mathfrak{X}}}
    \newcommand{\fY}{{\mathfrak{Y}}}

      \newcommand{\bB}{{\mathbf {B}}}

\newcommand{\Cl}{{\mathrm{Cl}}}                     
\newcommand{\End}{{\mathrm{End}}}                   
\newcommand{\Ext}{\mathrm{Ext}}                     
\newcommand{\Frac}{{\mathrm{Frac}}}                 
\newcommand{\Frob}{{\mathrm{Frob}}}                 
\newcommand{\Gal}{{\mathrm{Gal}}}                   
\newcommand{\GL}{{\mathrm{GL}}}                     
\newcommand{\Hom}{{\mathrm{Hom}}}                   
\renewcommand{\Im}{{\mathrm{Im}}}                   
\newcommand{\Norm}{{\mathrm{Norm}}}                 
\newcommand{\pr}{\mathrm{pr}}                       

\newcommand{\Spec}{{\mathrm{Spec}\,}}               
\newcommand{\Sp}{{\mathrm{Sp}}}
\newcommand{\Spf}{{\mathrm{Spf}}}                   
\newcommand{\Res}{{\mathrm{Res}}}                   
\newcommand{\tors}{\mathrm{tors}}                   
\newcommand{\tr}{{\mathrm{tr}}}                     
\newcommand{\Tr}{\mathrm{Tr}}                       

\renewcommand{\mod}{\, \mathrm{mod}\, }

\newcommand{\Qp}{\BQ_p}

\newcommand{\bQp}{\bar{\BQ}_p}

\newtheorem{thm}{Theorem}[section]

\newtheorem{theorem}[thm]{Theorem}

\newtheorem{corollary}[thm]{Corollary}

\newtheorem{lemma}[thm]{Lemma}

\newtheorem{proposition}[thm]{Proposition}

\theoremstyle{definition}

\newtheorem{definition}[thm]{Definition}

\newtheorem{remark}[thm]{Remark}

\numberwithin{equation}{subsection}

\renewcommand{\ss}{\mathrm{ss}}

\newcommand{\Drig}{D_{\mathrm{rig}}^\dagger}
\newcommand{\Drigs}{D_{\mathrm{rig}}^{\dagger, s}}
\newcommand{\Ddif}{D_{\mathrm{dif}}}
\newcommand{\Dcrysplus}{D_{\mathrm{crys}}^+}
\newcommand{\Dcrys}{D_{\mathrm{crys}}}
\newcommand{\DdRplus}{D_{\mathrm{dR}}^+}
\newcommand{\Dsen}{D_{\mathrm{Sen}}}

\newcommand{\bdR}{\bB_{\mathrm{dR}}}
\newcommand{\bcrys}{\bB_{\mathrm{crys}}}

\newcommand{\Fib}{\operatorname{Fib}}

\usepackage{enumitem}

\newcommand{\coh}{\mathrm{coh}}

\title{Properness of the Cuspidal Hilbert Eigenvariety}
\author{Yi Zhu}
\address{Department of Mathematical Sciences, Tsinghua University, Beijing 100084, China}
\email{zhuyi22@mails.tsinghua.edu.cn}
\date{}

\begin{document}
\begin{abstract}
We prove that for any prime $p$ and tame level $N\geq4$, the weight map from the cuspidal Hilbert eigenvariety constructed by Andreatta--Iovita--Pilloni for a totally real field $F$ to the weight space is proper in the sense of Buzzard--Calegari, under a non-vanishing hypothesis on certain Sen polynomials of Galois representations at the primes above $p$. 
In particular, this yields the properness of the eigenvariety away from non-positive integral Sen weights. 
Our approach extends the $p$-adic Hodge-theoretic argument of Diao--Liu. 
When $p$ has a unique prime of $F$ above it, we remove this hypothesis and obtain unconditional properness, 
by establishing base-change compatibility for spaces of overconvergent Hilbert modular forms
and analyzing the associated spectral variety over certain closed subspaces of the weight space.
\end{abstract}

\keywords{Hilbert eigenvarieties, properness, overconvergent Hilbert modular forms, $p$-adic Galois representations, finite slope subspaces}

\subjclass[2020]{Primary 11F33; Secondary 11F41, 11F80, 14G22}

\maketitle
\setcounter{tocdepth}{1}
\tableofcontents

\section*{Introduction}

\subsection{Background}
Eigenvarieties provide a geometric framework for studying $p$-adic families of automorphic forms and their systems of Hecke eigenvalues. 
Since the construction of the eigencurve by Coleman and Mazur \cite{ColemanMazur}, they have played an important role in the study of $p$-adic automorphic forms and the Galois representations attached to them. 
A basic geometric question concerning an eigenvariety is whether it has ``holes''. 
More precisely, one asks whether a family of finite slope eigensystems parametrized by a punctured disc, whose weights extend across the puncture, necessarily extends to the whole disc. 
This is the rigid-analytic valuative criterion for properness considered by Buzzard and Calegari \cite{BuzzardCalegari}. 
We use the word \emph{properness} in this sense throughout the paper.

For the Coleman--Mazur eigencurve, properness was first established by Buzzard and Calegari for $p=2$ and tame level $1$, and subsequently by Calegari at integral weights \cite{BuzzardCalegari,Calegari}. 
Diao and Liu eventually proved properness for arbitrary $p$ and tame level by a different method based on $p$-adic Hodge theory and the theory of trianguline representations in families \cite{DiaoLiu}. 
A central input in their argument is Liu's theory of finite slope subspaces for families of $p$-adic Galois representations \cite{Liu}. 
In the Hilbert modular setting, Andreatta, Iovita, and Pilloni constructed a cuspidal Hilbert eigenvariety parametrizing finite slope systems of Hecke eigenvalues occurring in spaces of overconvergent Hilbert modular cusp forms \cite{AIP}. 
Hattori proved a properness result for this eigenvariety at integral weights under certain restrictions on the primes above $p$, using the geometric method of Buzzard and Calegari \cite{Hattori}. 
In this paper, we establish properness results for the cuspidal Hilbert eigenvariety over arbitrary totally real fields $F$ and primes $p$, without restricting to integral weights.
When $p$ has a unique prime of $F$ above it, we prove unconditional properness.

\subsection{Main result}

Let $F$ be a totally real field and let $p$ be a rational prime. 
We denote by $\wt\colon \CE\rightarrow \CW$ the cuspidal Hilbert eigenvariety associated with $F$, together with its weight map. 
They are rigid analytic spaces over some finite extension $E$ of $\BQ_p$. 
The eigenvariety $\CE$ carries a two-dimensional pseudo-character of $G_F$. 
For every prime $\fp\mid p$, we restrict this pseudo-character to $G_{F_{\fp}}$ and make a suitable twist arising from the universal character on the weight space. 
We show that the Sen polynomials of the resulting semisimple representations interpolate over $\CE$ to a polynomial of the form $TQ_{\fp}(T)\in\bigl(F_{\fp}\otimes_{\Qp}\CO(\CE)\bigr)[T]$.
For an embedding $\tau\colon F_{\fp}\hookrightarrow \bQp$, we denote by $Q_{\fp}(T)_\tau$ the corresponding $\tau$-component. 
This notation allows us to state the main result of the paper.

\begin{theorem}[Theorem~\ref{thm:main}]\label{thm:intro-main}
Let $\CD$ be the closed unit disc over a finite extension $E'/E$, and let
$\CD^*=\CD\setminus\{0\}$ be the punctured disc. 
Suppose that $h\colon \CD^*\rightarrow \CE$ is a morphism of rigid analytic spaces such that $\wt\circ h$ extends to $h'\colon\CD\to\CW$.
Assume that, for every prime $\fp\mid p$ of $\CO_F$, there exists an embedding  $\tau_{0,\fp}\colon F_{\fp}\hookrightarrow\bQp$ such that, for every $j\in\BZ_{\leq 0}$, the pullback 
$h^*\bigl(Q_{\fp}(j)_{\tau_{0,\fp}}\bigr)$ is not identically zero on $\CD^*$. 
Then $h$ extends to a morphism $\widetilde h\colon \CD\longrightarrow\CE$ compatible with $h'$. 
If $\CO_F$ has only one prime ideal above
$p$, the conclusion holds without the
non-vanishing assumption on the
Sen polynomials. 
\end{theorem}

Therefore, the weight map of the eigenvariety $\CE\to \CW$ is proper under the above non-vanishing hypothesis on the Sen polynomials at the primes above $p$. 
If $F$ satisfies the last condition above, then there are no holes in the eigenvariety. 
See Section~\ref{sec:Gal rep} and Theorem~\ref{thm:main} for the precise formulation and the notation used in the statement. 

\begin{remark}
Hattori \cite[Theorem~1.1]{Hattori} proved properness at certain integral weights, under restrictions on the primes above $p$, by extending the geometric method of Buzzard--Calegari.
A key ingredient in his argument is to analytically continue overconvergent Hilbert eigenforms to a sufficiently large region and evaluate $U_p f$ at points associated with critical Hilbert--Blumenthal abelian varieties.
As Hattori explains, for weights that are not locally algebraic, such an evaluation is not available, since $U_p f$ is defined only on the canonical-subgroup locus.
Moreover, his analysis of the $U_p$-correspondence near the critical locus relies on \cite[Proposition~6.1]{Hattori2}, which is established under the restriction that the residue degrees at the primes above $p$ are at most two.

For general $F$, our approach instead follows the $p$-adic Hodge-theoretic method of Diao--Liu \cite{DiaoLiu}, avoiding the need to evaluate $U_p f$ at critical points.
However, a different difficulty arises in the Hilbert setting: the eigenvariety carries a Galois pseudo-character rather than a globally defined family of Galois representations.
We overcome this difficulty by using the oddness of the pseudo-character to realize its pullback on the disc as a family of Galois representations, and by developing the theory of $\tau_0$-finite slope subspaces together with a perfectness result for a variant of the Herr complex.
This allows arbitrary primes $p$ and totally real fields $F$, at the cost of a non-vanishing condition on the Sen polynomials along the punctured disc. 

When $p$ has a unique prime of $F$ above it, we remove the non-vanishing hypothesis by a second, geometric argument.
In this case, failure of the hypothesis on Sen polynomials forces the weight family to lie in a certain closed subspace of $\CW$. 
The key observation is that, over each connected affinoid subdomain of this locus, the induced weight for the auxiliary space of overconvergent Hilbert modular forms
is constant. 
We establish the required base-change compatibility for spaces of overconvergent modular forms and use it to identify the corresponding spectral variety with the base change of a spectral variety over a field.
This yields unconditional properness.

Thus the two results impose conditions of different kinds, and neither subsumes the other, although their ranges of applicability overlap; see Remark~\ref{rmk:on main thm}(3).
Our non-vanishing condition is automatically satisfied for a broad class of limiting weights. 
For fields in which $p$ has more than one prime above it, we hope to remove the remaining Sen-polynomial hypothesis in future work.
\end{remark}

\subsection{\texorpdfstring{$\tau_0$}{tau0}-finite slope subspaces}

The non-vanishing hypothesis in Theorem~\ref{thm:intro-main} is natural from the point of view of $\tau_{0,\fp}$-finite slope subspaces, which we will explain in a moment. 

Firstly, let $\CZ\subset \CW\times\BA^1_{\rig}$ be the spectral variety of $\CE$. 
Then the natural morphism $\pi\colon \CE\to \CZ$ is finite.
By \cite[Lemma~6.5]{Bellovin}, we can reduce to find a morphism $\CD\to \CZ$ compatible with $\pi\circ h$. 
Let $\alpha\in\CO(\CE)^\times$ be the $U_p$-eigenvalue function. 
After restricted to $\CD^*$, this desired morphism becomes 
\[
\pi\circ h=(\wt\circ h, h^*\alpha^{-1})\colon \CD^*\to \CZ\subset \CW\times \BA^1_\rig
\]
By construction, for any $x\in\CE$, $|\alpha(x)|\leq1$, so $h^*\alpha$ always extends to an element in $\CO(\CD)$ by \cite[Lemma~2.1]{DiaoLiu}. 
Therefore, it suffices to show that the extension of $h^*\alpha$ lies in $\CO(\CD)^\times$. 

To achieve this, the key step in Diao--Liu's proof on the properness of the eigencurve is to show that the $p$-adic family of Galois representations $V_{\CD^*}$ of $\Gal(\overline{\BQ_p}/\BQ_p)$ on $\CD^*$, obtained by pulling back the one on the normalization of the eigencurve, together with the pullback of the $U_p$-eigenvalue function $\alpha$, has the whole space $\CD^*$ as the finite slope subspace. 
Moreover, $V_{\CD^*}$ can be extended to a family $V_{\CD}$ on $\CD$. 
This allows them to establish comparisons between the De Rham periods and the crystalline periods
\[
\Dcrysplus(V_{R})^{\varphi=\alpha}=\DdRplus(V_{R}), \quad \Dcrysplus(V_{\CD})=\DdRplus(V_{\CD}),
\]
where $\Sp(R)$ is any affinoid subdomain of $\CD^*$.
It would imply $\alpha(0)\neq0$ and give the desired properness result. 
In our case, there is an odd pseudo-character
\[
T: G_F\to \CO(\CE). 
\]
Using the oddness, a result of Coleman--Mazur then allows one to realize the pullback of $T$ to $\CD^*$, which can be extended to $\CD$, as an actual family of Galois representations. We denote it by $V_{\CD}$. 
To imitate Diao--Liu's argument, it turns out that it suffices to have a comparison of the $\tau_0$-component of the various periods. 
More precisely, let $K=F_\fp$, $E$ be a finite extension of $\BQ_p$ in $\overline{\BQ_p}$ containing the Galois closure of $K$, $G_K=\Gal(\overline{\BQ_p}/K)$ and $H_K$ be the set of $\BQ_p$-embeddings of $K$ into $\overline{\BQ_p}$ (hence into $E$). 
For a $K\otimes_{\BQ_p}E$-module $M$, set 
\[M_\tau=M\otimes_{K\otimes_{\BQ_p}E}(K\otimes_{K, \tau}E)\]
for any $\tau\in H_K$, then we may identify $M$ with $\oplus_{\tau\in H_K}M_\tau$. 
Given a $p$-adic family of Galois representations $V_{\CD}$ of $G_K$ on $\CD=\Sp(E\langle T\rangle)$ and $\alpha\in\CO(\CD)\cap \CO(\CD^*)^\times$, we will see that if there are comparisons
\[
(K\otimes_{K_0} \Dcrysplus(V_{R})^{\varphi^f=\alpha})_{\tau_0}=\DdRplus(V_{R})_{\tau_0}, \quad
(K\otimes_{K_0} \Dcrysplus(V_{\CD}))_{\tau_0}=\DdRplus(V_{\CD})_{\tau_0}, 
\]
for some $\tau_0\in H_K$, where $K_0$ is the maximal unramified extension of $\BQ_p$ contained in $K$ and $f=[K_0:\BQ_p]$, then $\alpha(0)\neq0$. 
Under the assumptions in Theorem~\ref{thm:intro-main}, for each prime $\fp$ above $p$, we will take $\alpha$ to be a certain twist of the $U_{\fp}$-eigenvalue function on $\CE$ and prove the comparisons for a certain twist (to make the Sen polynomial have the form $TQ(T)$) of $V_{\CD}$.
These would imply that the pullback of $U_{\fp}$-eigenvalue function, which can be extended to $\CD$, is non-zero at the origin, hence the properness follows. 

To achieve the above comparison result, we make the following notion of $\tau_0$-finite slope subspace. 
Assume that $X$ is a rigid analytic space over $E$, and let $V_X$ be a family of $p$-adic representations of $G_K$ of dimension $d$ over $X$. 
Suppose the Sen polynomial for $V_X$ is $TQ(T)$ for some $Q(T)\in (K\otimes_{\BQ_p}\CO(X))[T]$. 
Let $\alpha\in\CO(X)^\times$. Fix $\tau_0\in H_K$. 
We define the $\tau_0$-finite slope subspace with respect to the triple $(X, \alpha, V_X)$ as follows. 
\begin{definition}\label{intro-def:X tau0-fs}
For such a quadruple $(X,\alpha,V_X, \tau_0)$, we call a Zariski closed $E$-subspace $X_{\tau_0-fs}\subset X$ a \emph{$\tau_0$-finite slope subspace} of $X$ with respect to the triple $(\alpha, V_X, \tau_0)$ if the following conditions hold.
\begin{enumerate}
\item For every integer $j\le0$, the subspace $(X_{\tau_0-fs})_{Q(j)_{\tau_0}}$, which denotes the complement of the vanishing locus of $Q(j)_{\tau_0}$, is Zariski-open and dense in $X_{\tau_0-fs}$.
\item For any affinoid $E$-algebra $R$ and morphism $g\colon \Sp(R)\longrightarrow X$ which factors through $X_{Q(j)_{\tau_0}}$ for every $j\le0$, 
the morphism $g$ factors through
$X_{\tau_0-fs}$ if and only if the
natural map
\begin{equation}\label{intro-eq:tau0-fs isom}
\iota_{n,K, \tau_0}\colon \left(K\otimes_{K_0}(D_{\mathrm{rig}}^\dagger(V_R))^{\varphi^f=g^*\alpha,\Gamma_K=1}\right)_{\tau_0} 
\to (D_{\mathrm{dif}}^{+,fn}(V_R))^{\Gamma_K}_{\tau_0}
\end{equation}
is an isomorphism for all sufficiently large $n$.
\end{enumerate}
\end{definition}
Then for each prime $\fp$ above $p$, we will see the non-vanishing hypothesis in Theorem~\ref{thm:intro-main} is an equivalent condition for $\CD^*$ to be the $\tau_{0,\fp}$-finite slope subspace of a corresponding family of Galois representations of $G_{F_\fp}$. 
It is slightly weaker than Liu's definition of finite slope subspace, and helps us to prove the desired comparison of various periods. 

\subsection{Ideas of the proof: general case with the Sen polynomial assumption}

Given the above partial strategy, it remains to show that for each prime $\fp$ above $p$, there is a suitable twist of $V_{\CD}$ and the $U_\fp$-eigenvalue function, such that under the assumptions in Theorem~\ref{thm:intro-main}, the $\tau_0$-finite slope subspace of the family over $\CD^*$ is the whole $\CD^*$. 

In the case of the eigencurve, Liu proves that for a weakly refined family of $p$-adic Galois representations, the finite slope subspace is the whole space \cite{Liu}, which finishes the crucial step in the case of the eigencurve if the morphism $h$ is non-constant. 

There are several obstacles to carrying out this argument for the Hilbert eigenvariety. 
The first is that $\CE$ naturally carries a family of pseudo-characters, rather than an actual family of Galois representations. 
Consequently, Liu's notion of a weakly refined family cannot be applied directly on $\CE$. 
This is particularly relevant here since the Hilbert eigenvariety has dimension greater than one. 

This leads to the first general result of the paper. 
We establish a new criterion for the $\tau_0$-finite slope subspace of a family of $p$-adic Galois representations to coincide with the whole base space. 
The criterion isolates the two features of weakly refined families that are relevant to Liu's argument.
The first condition is precisely the non-vanishing condition on the Sen polynomial appearing in the definition of the $\tau_0$-finite slope subspace. 
The second one asks that for any $x\in X$, 
\[
\Drig(V_x)^{\varphi^f=\alpha(x), \Gamma_K=1}\neq0. 
\]
It is closely related to condition (d) in \cite[Definition~1.5]{Liu} of a weakly refined family and may be viewed as a trianguline condition. 
In this way, the criterion can be applied directly to a family over $\CD^*$, without first exhibiting a weakly refined family on the ambient eigenvariety.

The two conditions in this criterion play rather different roles in the proof of the main theorem. 
The first one is exactly what is needed for the first part of our $\tau_{0,\fp}$-finite slope criterion, so it is responsible for the hypothesis in Theorem~\ref{thm:intro-main}.
The second condition, by contrast, is not imposed as an additional hypothesis: it is established by studying the local Galois representations occurring on the Hilbert eigenvariety. 

To check it, we need some results on the cohomology of $(\varphi,\Gamma)$-modules. 
They are closely related to the cohomological theory developed by Kedlaya--Pottharst--Xiao \cite{KPX} and, in the pseudorigid setting, by Bellovin \cite{Bellovin}. 
More precisely, let $S$ be a $\BQ_p$-affinoid algebra, $\alpha\in S^\times$ and let $D$ be a finite projective $(\varphi,\Gamma_K)$-module over
$\mathbf B_{\mathrm{rig},K,S}^{\dagger}$. 
We will show that the function
\[
\Sp(S)\to \BZ, \quad x\mapsto \dim_{k(x)} (D\widehat{\otimes}_{S}k(x))^{\varphi^f=\alpha, \Gamma_K=1}
\]
is upper semi-continuous on $\Sp(S)$ with respect to the Zariski topology. 
When the residue degree $f=1$, this follows from the perfectness of the Herr complex proved in \cite{KPX}. 
A closely related statement appears in \cite[Corollary~3.17]{Bellovin}. 
Although the proof follows the general techniques of the cohomology of families of $(\varphi,\Gamma)$-modules, we include the details since the precise form required for our application does not seem to be available in the literature.

We then turn to the Galois-theoretic properties of the Hilbert eigenvariety. 
Using the density of classical points and local--global compatibility at such points, together with an argument inspired by the proof of \cite[Proposition~6.6]{McDonald}, we establish the trianguline property required by the second condition of our criterion.
We may therefore apply our criterion to the actual Galois family over $\CD^*$ and conclude, for every $\fp\mid p$, that the $\tau_{0,\fp}$-finite slope subspace of the family over $\CD^*$ is all of $\CD^*$. 

There is one further foundational point concerning the Hilbert eigenvariety that we make explicit.
The construction of the Galois pseudo-character relies on the density of classical points.
For the applications in this paper we need the stronger statement that classical points with cohomological weights are very Zariski dense.
We prove in Appendix~\ref{app:coh wts} the corresponding very Zariski density statement for cohomological weights in the Hilbert weight space, and deduce from it, together with the classicality results entering the eigenvariety construction, that the classical points of $\CE$ with cohomological weights are very Zariski dense. 
This provides the density needed to interpolate Galois-theoretic information from classical points to the eigenvariety.

\subsection{The case of a unique prime above $p$}

We now explain the second approach, which removes the Sen-polynomial hypothesis when $p$ has a unique prime $\fp$ in $F$.

Suppose that the hypothesis on $Q_\fp(j)_{\tau_{0, \fp}}$ of Theorem~\ref{thm:intro-main} fails for a morphism $h\colon\CD^*\to\CE$.
Then, for every $\tau\in I:=\Hom_{\BQ_p}(F_\fp,\overline{\BQ_p})$, there exists $k_\tau\in\BZ_{\leq0}$ such that
\[
h^*\bigl(Q_\fp(k_\tau)_\tau\bigr)=0.
\]
By Proposition~\ref{prop:Sen poly}, there are functions $K_\tau\in\CO(\CW)$ such that
\[
Q_\fp(T)_\tau=T-\wt^*K_\tau.
\]
Hence the extended weight map $h'$ factors through the closed subspace
\[
\CW_\kappa:=V(K_\tau-k_\tau:\tau\in I)\subset\CW,
\qquad \kappa=(k_\tau)_{\tau\in I}.
\]
After shrinking $\CD$, we may assume that $h'$ factors through a connected affinoid subdomain $\CV\subset\CW_\kappa$.

The key observation concerns an auxiliary weight space. 
Recall that $\CW$ parametrizes continuous characters $(\nu,w)$ of $(\CO_F\otimes_\BZ\BZ_p)^\times
\times\BZ_p^\times$.
Let $\CW^*$ be the weight space parametrizing continuous characters of $(\CO_F\otimes_\BZ\BZ_p)^\times$.
There is a natural morphism
\[
k\colon\CW\longrightarrow\CW^*,
\qquad
(\nu,w)\longmapsto
\bigl(x\mapsto
\nu(x)^2w(\Norm_{F/\BQ}(x))\bigr),
\]
as recalled in Section~\ref{sec:eigenvariety}.
By Lemma~\ref{lem:loc alg wt}, the restriction $k|_\CV$ factors through a closed point $x_\CV\in\CW^*$.
Thus the auxiliary weight is constant over $\CV$, even though the original weight may vary.

To exploit this observation, we establish base-change compatibility for spaces of overconvergent Hilbert modular forms, both along certain closed immersions of weight spaces and along morphisms from a weight affinoid to a fixed weight.
These results allow us to identify
\[
\CZ_\CV:=\CZ\times_\CW\CV
\]
with the spectral variety obtained by base change from a compact operator over a finite extension of
$k(x_\CV)$.
Consequently, $\CZ_\CV$ admits a disjoint-union decomposition into subspaces that are finite and flat over
$\CV$ (Theorem~\ref{thm:spectral}).
Together with the finite morphism $\CE\to\CZ$, this gives the required extension across the puncture.

\subsection{Structure of the paper}

In Section~\ref{sec:families of p-adic rep}, we recall the basic facts about families of $p$-adic Galois representations and $p$-adic Hodge theory that will be used throughout the paper, including Sen polynomials, relative $(\varphi,\Gamma)$-modules, and finite slope subspaces. 
In Section~\ref{sec:fs subspace}, we introduce the notion of a $\tau_0$-finite slope subspace and establish our criterion for it to coincide with the whole base space. 
Section~\ref{sec:Diao-Liu} extends the local $p$-adic Hodge-theoretic argument of Diao--Liu from $G_{\Qp}$ to $G_K$ for a general finite extension $K/\Qp$, and proves the non-vanishing of the limiting Frobenius eigenvalue under the $\tau_0$-finite slope hypothesis.
In Section~\ref{sec:phi-f-herr}, we prove the $(\varphi,\Gamma)$-module cohomological results needed later. 

In Sections~\ref{sec:oc MFs for G*} and~\ref{sec:eigenvariety}, we review the construction of spaces of overconvergent Hilbert modular forms by Andreatta--Iovita--Pilloni and establish their compatibility with base change along certain morphisms of weight spaces, which helps us to study the base change of the spectral variety. 
We also establish the density properties of its classical points with cohomological weights in the cuspidal Hilbert eigenvariety, using the very Zariski density result proved in Appendix~\ref{app:coh wts}. 
Section~\ref{sec:Gal rep} studies the family of Galois pseudo-characters on the eigenvariety.
We establish the trianguline property needed for the $\tau_0$-finite slope criterion and show that, for every $\fp\mid p$, the Sen polynomials of the associated semisimple $G_{F_{\fp}}$-representations interpolate over $\CE$. 
Finally, in Section~\ref{sec:properness}, we prove Theorem~\ref{thm:main} by combining the $p$-adic Hodge-theoretic argument under the Sen-polynomial hypothesis with the spectral-geometric argument for the case of a unique prime above $p$.

The appendices collect two auxiliary results used in the above arguments.
In Appendix~\ref{app:coh wts}, we prove the very Zariski density of cohomological weights in the Hilbert weight space. 
In Appendix~\ref{app:global realization}, we clarify the relation between odd pseudo-characters and the notion of pseudo--representations used by Coleman--Mazur, and deduce that an odd pseudo-character on a smooth curve can be realized as an actual family of Galois representations.


\medskip
\noindent\textbf{Acknowledgements.}
I would like to express my deepest gratitude to my advisor, Hansheng Diao, for suggesting this problem and for their constant guidance and encouragement throughout this work. 
I am very grateful to Lynnelle Ye and Shin Hattori for patiently answering my questions. 
I also thank Koji Shimizu for organizing the seminar on $p$-adic modular forms, and Haoran Wang for organizing the seminar on deformations of Galois representations; I benefited greatly from the discussions with the participants in these seminars. 
I am indebted to Ruishen Zhao, Jiawei An, Yu Xiao, Hao Liu, Ningjun Jiang, Zida Wang, and others for helpful discussions. 
Finally, I thank my family and friends for their continuous support.

\medskip
\noindent\textbf{Use of AI tools.}
The author used ChatGPT (OpenAI; GPT-5.6 in August 2026 and GPT-6 in October and November 2026), accessible via chatgpt.com, for exploratory discussions concerning the proof of the density result in Appendix~\ref{app:coh wts}, and for assistance with drafting and revising the English text of parts of the Introduction.
In November 2026, GPT-6 was also used to help identify potential mathematical errors, gaps in arguments, inconsistencies in normalizations, and issues concerning the applicability of cited results in the revised manuscript.
The author reviewed the AI-generated suggestions and is solely responsible for the mathematical content and the final manuscript.

\medskip
\noindent\textbf{Notation and conventions.}
Let $p$ be a prime number.
Fix an algebraic closure $\overline{\BQ_p}$ of $\BQ_p$ with $p$-adic completion $\BC_p$. 
We normalize the $p$-adic valuation $v_p$ on $\BC_p$ so that $v_p(p)=1$. 
Choose a compatible system of roots of unity $(\mu_{p^n})_{n\geq0}$. 
For a finite extension $K$ of $\BQ_p$, we write $K_0$ to be the maximal unramified extension of $\BQ_p$ in $K$, $G_K=\Gal(\overline{K}/K)$. 
Let $\CO_K$ be the ring of integers of $K$, $\pi_K\in \CO_K$ be a uniformizer. 
For any $n\geq1$, let $K_n=K(\mu_{p^n})$, then we define $K_{\infty}=\cup_{n\geq1}K_n, \Gamma_K=\Gal(K_{\infty}/K)$.
For a number field $F/\BQ$, we write $G_F=\Gal(\overline{F}/F)$. 

For a topological group $G$ and a rigid analytic space $X$ over $\BQ_p$, by a \emph{family of $p$-adic representations of $G$} of dimension $d$ on $X$, we mean a locally free coherent $\CO_X$-module $V_X$ of rank $d$ equipped with a continuous $\CO_X$-linear $G$-action. 
If $\Sp(R)\to X$ is a morphism of rigid analytic spaces and if $V_X$ is a family of representations on $X$, we write $V_R$ for the pullback of $V_X$ to $\Sp(R)$. 
For every $x\in X$, we write $V_x$ to denote the specialization of $V_X$ at $x$. 
When $X=\Sp(S)$ is an affinoid space over $\BQ_p$, we also call a family of $p$-adic representations of $G$ on $X$ an \emph{$S$-linear representation of $G$. }
For a closed point $x$ of a rigid space, we use $k(x)$ to denote its residue field. 

\section{Families of \texorpdfstring{$p$}{p}-adic representations}
\label{sec:families of p-adic rep}

Let $K$ be a finite extension of $\BQ_p$ in $\overline{\BQ_p}, S$ be an affinoid algebra over $\BQ_p$ and let $V_S$ be a finite locally free $S$-linear representation of $G_K$. 
In this section, we briefly review various $p$-adic Hodge theoretic functors for families of $p$-adic representations of $G_K$. 

\subsection{The modules
  \texorpdfstring{$D_{\mathrm{dR}}^+(V_S)$}{D\_dR\textasciicircum+(V\_S)}
  and
  \texorpdfstring{$D_{\mathrm{crys}}^+(V_S)$}{D\_crys\textasciicircum+(V\_S)}}
\label{sec:modules-DdR-Dcrys}

Let $\bdR^+$ and $\bcrys^+$ be the de Rham and crystalline period rings used in $p$-adic Hodge theory (see \cite{Fontaine} for more details). 
For each $k>0, \bdR^+/t^k$ is naturally a $\BQ_p$-Banach space. 
This gives a Fr\'echet topology on 
\[\bdR^+=\varprojlim_{k}\bdR^+/t^k. \]
Thus, for any $\BQ_p$-affinoid algebra $S$, it makes sense to define $S\widehat{\otimes}_{\BQ_p}\bdR^+=\varprojlim_{k}S\widehat{\otimes}_{\BQ_p}\bdR^+/t^k$. 
On the other hand, the definition of $S\widehat{\otimes}_{\BQ_p}\bcrys^+$ is the natural one since $\bcrys^+$ has a $\BQ_p$-Banach space structure. 
Moreover, for an $S$-linear representation $V_S$ of $G_K$, we set
\[
\DdRplus(V_S)=\left((S\widehat{\otimes}_{\BQ_p}\bdR^+)\otimes_S V_S\right)^{G_K}
\]
and 
\[
\Dcrysplus(V_S)=\left((S\widehat{\otimes}_{\BQ_p}\bcrys^+)\otimes_S V_S\right)^{G_K}. 
\]

\begin{proposition}\label{prop:flat bc of DdRplus}
If $f\colon S\to S'$ is a flat morphism of $\BQ_p$-affinoid algebras, then
\[
\DdRplus(V_S)\otimes_{S}S'\xrightarrow{\sim}\DdRplus(V_S\otimes_SS'). 
\]
\end{proposition}

\begin{proof}
See the proof of \cite[Proposition~4.3.7]{Bellovin15}. 
\end{proof}

\subsection{Robba rings and \texorpdfstring{$(\varphi,\Gamma_K)$}{(phi,Gamma\_K)}-modules}

Recall that $\bB_{K}^{\dagger, [s, r]}, \bB_{\rig, K}^{\dagger, s}$ and $\bB_{\rig, K}^{\dagger}$ are the \emph{Robba-type} period rings in the theory of $(\varphi, \Gamma_K)$-modules introduced in \cite{Berger}. 
More precisely, for any $r\geq s>0$, there are natural norms $\mathrm{val}^{(0, r]}$ on $\bB_K^{\dagger, s}$, and $B_K^{[s, r]}$ is the completion of $\bB_K^{\dagger, s}$ with respect to $\max\{\mathrm{val}^{(0, r]}, \mathrm{val}^{(0, s]}\}$. 
It follows that 
\[
\bB_{\rig, K}^{\dagger, s}=\varprojlim_{r}\bB_K^{[s, r]}.
\]
Put
\[
\bB_{\rig, K}^{\dagger}=\bigcup_{s>0}\bB_{\rig, K}^{\dagger, s}. 
\]

Let $S$ be a $\BQ_p$-affinoid algebra, put
\begin{align*}
\mathbf B_{K,S}^{[s_1,s_2]}&=\mathbf B_K^{[s_1,s_2]}\widehat{\otimes}_{\Qp}S, \\ 
\mathbf B_{\mathrm{rig},K,S}^{\dagger,s}&=\mathbf B_{\mathrm{rig},K}^{\dagger,s}\widehat{\otimes}_{\Qp} S\simeq\varprojlim_{r}\mathbf B_{K,S}^{[s, r]},\\
\mathbf B_{\mathrm{rig},K,S}^{\dagger}&=\mathbf B_{\mathrm{rig},K}^{\dagger}\widehat{\otimes}_{\Qp} S=\bigcup_{s>0}\mathbf B_{\mathrm{rig},K,S}^{\dagger,s}.
\end{align*}
There are commuting $S$-linear actions on $\mathbf B_{K,S}^{[s,r]}$ of $\Gamma_K$ and of an operator 
\[
\varphi\colon\mathbf B_{K,S}^{[s,r]}\longrightarrow\mathbf B_{K,S}^{[ps,pr]}, 
\]
which induce
\[
\varphi\colon \mathbf B_{\mathrm{rig},K,S}^{\dagger,s}\longrightarrow\mathbf B_{\mathrm{rig},K,S}^{\dagger,ps}
\]
and 
\[
\varphi\colon \mathbf B_{\mathrm{rig},K,S}^{\dagger}\longrightarrow\mathbf B_{\mathrm{rig},K,S}^{\dagger}. 
\]
Moreover, $\varphi$ makes $\mathbf B_{K,S}^{[ps,pr]}$ into a free $\mathbf B_{K,S}^{[s,r]}$-module of rank $p$, and we obtain a $\Gamma_K$-equivariant left inverse $\psi\colon  \mathbf B_{K,S}^{[ps,pr]}\to \mathbf B_{K,S}^{[s,r]}$ by the formula $\psi:=p^{-1}\varphi^{-1}\circ\Tr$. 
The map $\psi$ extends to $B_{\mathrm{rig},K,S}^{\dagger,ps}\to B_{\mathrm{rig},K,S}^{\dagger,s}$ and $B_{\mathrm{rig},K,S}^{\dagger}\to B_{\mathrm{rig},K,S}^{\dagger}$. 
See \cite{Berger}, \cite{KedlayaLiu} and \cite{KPX}. 

\begin{definition}
\begin{itemize}
\item 
A \emph{$\varphi$-module over
$\mathbf B_{\mathrm{rig},K,S}^{\dagger,s}$} is a finite projective module $D_S^s$ together with an isomorphism
$\varphi^*D_S^s\xrightarrow{\sim}\mathbf B_{\mathrm{rig},K,S}^{\dagger,ps}\otimes_{\mathbf B_{\mathrm{rig},K,S}^{\dagger,s}}D_S^s$.
\item 
A \emph{$\varphi$-module $D_S$ over $\mathbf B_{\mathrm{rig},K,S}^{\dagger}$} is the base change to $\mathbf B_{\mathrm{rig},K,S}^{\dagger}$ of a $\varphi$-module $D_S^s$ over $\mathbf B_{\mathrm{rig},K,S}^{\dagger,s}$ for some $s>0$. 
\item 
A \emph{$(\varphi, \Gamma_K)$-module over $\mathbf B_{\mathrm{rig},K,S}^{\dagger,s}$} is a $\varphi$-module $D_S^s$ over $\mathbf B_{\mathrm{rig},K,S}^{\dagger,s}$ equipped with a commuting $\mathbf B_{\mathrm{rig},K}^{\dagger,s}$-semilinear and $S$-linear continuous action of $\Gamma_K$. 
\item 
A \emph{$(\varphi, \Gamma_K)$-module $D_S$ over $\mathbf B_{\mathrm{rig},K,S}^{\dagger}$} is the base change to $\mathbf B_{\mathrm{rig},K,S}^{\dagger}$ of a $(\varphi, \Gamma_K)$-module $D_S^s$ over $\mathbf B_{\mathrm{rig},K,S}^{\dagger,s}$ for some $s>0$. 
\end{itemize}
\end{definition}

\subsection{The modules
  \texorpdfstring{$D_{\mathrm{rig}}^{\dagger}(V_S)$}{D\_rig\textasciicircum{}dagger(V\_S)}
  and
  \texorpdfstring{$D_{\mathrm{dif}}^+(V_S)$}{D\_dif\textasciicircum+(V\_S)}}
\label{sec:Drig-Ddif}

Let $V_S$ be an $S$-linear $G_K$-representation of rank $d$.
When $V_S$ is free over $S$, there is an integer $n(V_S)$ introduced in \cite[\S~2.1]{Liu} to construct the associated $(\varphi, \Gamma_K)$-modules, and we set $s(V_S):=r_{n(V_S)}$, where $r_n:=p^{n-1}(p-1)$. 
For general $V_S$, which is locally free, choose a  finite covering of $\Sp(S)$ by affinoid subdomains $\Sp(S_i)$ such that $V_{S_i}$ is free over $S_i$. 
Let $s_0:=\max_{i}s(V_{S_i})$ and $n(V_S):=\max_{i}n(V_{S_i})$. 
Then for $s\geq s_0$, one can construct a $(\varphi, \Gamma_K)$-module $\Drigs(V_S)$ of rank $d$ over $\mathbf B_{\mathrm{rig},K,S}^{\dagger,s}$ such that for any $x\in\Sp(S), \Drigs(V_S)\otimes_Sk(x)$ is naturally isomorphic to $\Drigs(V_x)$. 
We set
\[\Drig(V_S)=\mathbf B_{\mathrm{rig},K,S}^{\dagger}\otimes_{\mathbf B_{\mathrm{rig},K,S}^{\dagger, s}}\Drigs(V_S)
=\bigcup_{s}\Drigs(V_S), \]
which is a $(\varphi, \Gamma_K)$-module of rank $d$ over $\mathbf B_{\mathrm{rig},K,S}^{\dagger}$, and specializes to $\Drig(V_x)$ for any $x\in\Sp(S)$. 
For details, see \cite{BergerColmez}, \cite{KedlayaLiu} and \cite[\S~2.2]{Liu}. 

For $0<s\leq r_n$, there is a continuous $\Gamma_K$-equivariant injective map
\[\iota_n\colon  \bB_K^{\dagger, s}\to K_n[[t]]\]
which extends to a continuous $\Gamma_K$-equivariant injective map 
\[\iota_n\colon  \mathbf B_{\mathrm{rig},K}^{\dagger,s}\to K_n[[t]]. \]
This induces a continuous map $\mathbf B_{\mathrm{rig},K,S}^{\dagger,s}\to K_n[[t]]\widehat{\otimes}_{\BQ_p}S$. 
Let $n\geq n(V_S), s_0\leq s\leq r_n$. 
We define
\[\Ddif^{+, n}(V_S)=\Drigs(V_S)\otimes_{\mathbf B_{\mathrm{rig},K,S}^{\dagger,s}, \iota_n}\left(K_n[[t]]\widehat{\otimes}_{\BQ_p}S\right). \]
It is clear that $\Ddif^{+, n}(V_S)$ is a locally free $K_n[[t]]\widehat{\otimes}_{\BQ_p}S$-module of rank $d$ equipped with a semilinear $\Gamma_K$-action. 
Abusing the notation, we still denote by $\iota_n$ the natural map $\iota_n\colon  \Drigs(V_S)\to \Ddif^{+, n}(V_S)$. 
We define 
\[\Ddif^{+}(V_S)=\bigcup_{n}\Ddif^{+, n}(V_S). \]

\begin{theorem}\cite[Theorem~4.2.8, 4.2.9]{Bellovin15}\label{thm:comparison of various D}
For an $S$-linear representation $V_S$ of $G_K$, we have $\DdRplus(V_S)=\Ddif^+(V_S)^{\Gamma_K}$ and $\Dcrysplus(V_S)=\Drig(V_S)^{\Gamma_K}$. 
\end{theorem}

\subsection{The module \texorpdfstring{$D_{\mathrm{Sen}}(V_S)$}{D\_Sen(V\_S)} and the Sen operator}\label{sec:Sen operator}

Keep notation as in the previous section. 
Let 
\[
\Dsen^{n}(V_S)=\Ddif^{+, n}(V_S)/t, \quad \Dsen(V_S)=\bigcup_n \Dsen^n(V_S).
\]
Then there is a natural operator 
\[
\Theta\in \End_{K_{n}\otimes_{\BQ_p}S}(\Dsen^n(V_S)), 
\]
called the \emph{Sen operator} for $V_S$. 
Moreover, $\Theta$ commutes with $\Gamma_K$, hence its characteristic polynomial lies in $K\otimes_{\BQ_p} S$, called the \emph{Sen polynomial} for $V_S$. 

For the rest of this section, we fix a finite extension $E$ of $\BQ_p$ that contains the images of all embeddings of $K$ into $\overline{\BQ_p}$. 
Let $H_K$ be the set of $\BQ_p$-embeddings of $K$ into $\overline{\BQ_p}$ (hence into $E$). 
For a $K\otimes_{\BQ_p}E$-module $M$, set 
\[M_\tau=M\otimes_{K\otimes_{\BQ_p}E}(K\otimes_{K, \tau}E)\]
for any $\tau\in H_K$, then we may identify $M$ with $\oplus_{\tau\in H_K}M_\tau$. 
For any $m\in M$, let $m_\tau$ denote the projection of $m$ onto $M_\tau$. 
We further assume $S$ is an $E$-affinoid algebra. 
In particular, for each $\tau\in H_K$, we have 
\[
\Theta_\tau\colon  \Dsen^n(V_S)_\tau\to \Dsen^n(V_S)_\tau, 
\]
so that $\Theta=\oplus_{\tau\in H_K}\Theta_\tau$. 
Fix $\tau_0\in H_K$. 

\begin{lemma}
For any $k\geq1, n\geq n(V_S)$, the natural map
\[
(\Ddif^{+, n}(V_S)/t^k)^{\Gamma_K}_{\tau_0}\to (\Dsen^n(V_S))^{\Gamma_K}_{\tau_0}
\]
has kernel and cokernel killed by $\prod_{i=1}^{k-1}\det(\Theta_{\tau_0}+iI)$. 
\end{lemma}

\begin{proof}\label{lem:tau0 Liu's cor2.27}
This is the same as the proof of \cite[Corollary~2.27]{Liu}, which follows from the analog of \cite[Proposition~2.26]{Liu} with $\Dsen^n(V_S)$ and $\Theta$ replaced by $\Dsen^n(V_S)_{\tau_0}$ and $\Theta_{\tau_0}$ respectively. 
\end{proof}

We further suppose $\det(\Theta_{\tau_0})=0$ and write $\det(TI-\Theta_{\tau_0})=TQ(T)_{\tau_0}\in S[T]$. 
Put $P(i)_{\tau_0}=\prod_{j=0}^{i-1}Q(-j)_{\tau_0}$ for every integer $i\geq1$. 

\begin{lemma}\label{lem:tau0 Liu's cor2.30}
If $f\colon S\to R$ is a morphism of $E$-affinoid algebras, for each $n\geq n(V_S)$, the natural map
\[
\left(\Ddif^{+, n}(V_S)/t^k\right)^{\Gamma_K}_{\tau_0}\otimes_SR\to \left(\Ddif^{+, n}(V_R)/t^k\right)^{\Gamma_K}_{\tau_0}
\]
has kernel and cokernel killed by a power of $f(P(k)_{\tau_0})$. 
In particular, if $f(P(k)_{\tau_0})$ is a unit, the map is an isomorphism. 
\end{lemma}

\begin{proof}
This is the same as the proof of \cite[Corollary~2.30]{Liu}, which follows from the analog of \cite[Proposition~2.29]{Liu} with $\Dsen^n(V_S), Q(0)$ and $P(k)$ replaced by $\Dsen^n(V_S)_{\tau_0}, Q(0)_{\tau_0}$ and $P(k)_{\tau_0}$ respectively. 
\end{proof} 
\section{\texorpdfstring{$\tau_0$}{tau0}-finite slope subspaces}\label{sec:fs subspace}

In this section, we refine the notion of finite slope subspace and give a criterion for the $\tau_0$-finite slope subspace to be the whole space. 

\subsection{Definition and construction}

We keep the notation as in the last part of Section~\ref{sec:Sen operator}. 
Let $X$ be a reduced and separated rigid analytic space over $E$, $V_X$ be a family of $G_K$-representations over $X$. 
Fix $\tau_0\in H_K$, assume that the $\tau_0$-component of the Sen polynomial is $TQ(T)_{\tau_0}\in \CO(X)[T]$. 

Let $\alpha\in\CO(X)^\times, f=[K_0:\BQ_p]$. 
By \cite[Lemma~2.11]{Liu}, for any morphism $g\colon\Sp(R)\to X$ and $s\geq s(V_R)$, $(\Drig(V_R))^{\varphi^f=g^*\alpha}$ is contained in $\Drigs(V_R)$. 
Thus for any $n\geq n(V_R)$, we have a natural map
\[
\Drig(V_R)^{\varphi^f=g^*\alpha}\to \Ddif^{+, n}(V_R)
\]
via the localization map $\iota_n$. 

For any rigid analytic space $Y$ over $E$ and $f\in\CO(Y)$, we denote by $Y_f$ the complement of the vanishing locus of $f$ on $Y$; it is a Zariski-open subset of $Y$. 

\begin{definition}\label{def:X tau0-fs}
For such a quadruple $(X,\alpha,V_X, \tau_0)$, we call a Zariski closed $E$-subspace $X_{\tau_0-fs}\subset X$ a \emph{$\tau_0$-finite slope subspace} of $X$ with respect to the triple $(\alpha, V_X, \tau_0)$ if the following conditions hold.
\begin{enumerate}
\item For every integer $j\le0$, the subspace $(X_{\tau_0-fs})_{Q(j)_{\tau_0}}$
is Zariski-open and dense in $X_{\tau_0-fs}$.
\item For any affinoid $E$-algebra $R$ and morphism $g\colon \Sp(R)\longrightarrow X$ which factors through $X_{Q(j)_{\tau_0}}$ for every $j\le0$, 
the morphism $g$ factors through
$X_{\tau_0-fs}$ if and only if the
natural map
\begin{equation}\label{eq:tau0-fs isom}
\iota_{n,K, \tau_0}\colon \left(K\otimes_{K_0}(D_{\mathrm{rig}}^\dagger(V_R))^{\varphi^f=g^*\alpha,\Gamma_K=1}\right)_{\tau_0} 
\to (D_{\mathrm{dif}}^{+,fn}(V_R))^{\Gamma_K}_{\tau_0}
\end{equation}
is an isomorphism for all sufficiently large $n$.
\end{enumerate}
\end{definition}

We are going to prove the following theorem, which is an analogue of \cite[Theorem~4.10]{Liu}. 

\begin{theorem}\label{thm:X tau0-fs}
The rigid analytic space $X$ has a unique $\tau_0$-finite slope subspace $X_{\tau_0-fs}$. 
\end{theorem}

\begin{lemma}\label{lem:glue prop of Xfs}
Let $\{U_j\}_{j\in J}$ be an admissible covering of $X$ by affinoid subdomains. 
Suppose that each $U_j$ has the finite slope subspace $(U_j)_{\tau_0-fs}$. 
Then they glue to form the finite slope subspace of $X$. 
\end{lemma}

\begin{proof}
The argument is the same as that in the proof of \cite[Proposition~4.4]{Liu}. 
More precisely, the analogues of \cite[Proposition~4.1 and~4.2]{Liu} hold in the present setting with fixed $\tau_0$. 
Applying these results as in the proof of \cite[Proposition~4.4]{Liu} gives the assertion. 
\end{proof}

Let $U=\Sp(S)\subset X$ be an affinoid subdomain.
For $f\in\CO(U)$, let $|f|_{\mathrm{sp}, U}$ be the spectral norm of $f$ on $U$. 

\begin{lemma}\label{lem:tau0-inj}
For any integer $n\geq n(V_S)$ and positive integer $k>\log_{|\pi_K|^{-1}}|\alpha^{-1}|_{\mathrm{sp}, U}$, the natural map 
\[
\left(K\otimes_{K_0}(D_{\mathrm{rig}}^\dagger(V_S))^{\varphi^f=\alpha}\right)_{\tau_0} 
\to (D_{\mathrm{dif}}^{+,fn}(V_S)/t^k)_{\tau_0}
\]
is injective. 
\end{lemma}

\begin{proof}
This follows from \cite[Proposition~4.6]{Liu}. 
\end{proof}

\begin{lemma}\label{lem:g(a) in image}
For any positive integers $n\geq n(V_S), k>\log_{|\pi_K|^{-1}}|\alpha^{-1}|_{\mathrm{sp}, U}$ and $a\in (\Ddif^{+, fn}(V_S)/t^k)_{\tau_0}$, 
there exists a Zariski closed $E$-subspace $\Sp(S(k, \alpha, a, n, \tau_0))$ of $\Sp(S)$, such that for any map $g\colon  S\to R$ of affinoid $E$-algebras, $g(a)\in (\Ddif^{+, fn}(V_R)/t^k)_{\tau_0}$ is contained in the image of 
\[\left(K\otimes_{K_0}D_{\mathrm{rig}}^\dagger(V_R)^{\varphi^f=g(\alpha)}\right)_{\tau_0} 
\to \left(D_{\mathrm{dif}}^{+,fn}(V_R)/t^k\right)_{\tau_0}\]
if and only if the map $\Sp(R)\to \Sp(S)$ factors through $\Sp(S(k, \alpha, a, n, \tau_0))$. 
\end{lemma}

\begin{proof}
This follows from \cite[Proposition~4.7]{Liu}. 
\end{proof}

\begin{corollary}\label{cor:n submod in image}
For any integer $n\geq n(V_S)$ and positive integer $k>\log_{|\pi_K|^{-1}}|\alpha^{-1}|_{\mathrm{sp}, U}$, 
there exists a Zariski closed $E$-subspace $\Sp(S(k, \alpha, n, \tau_0))$ of $\Sp(S)$, such that for any map $g\colon  S\to R$ of affinoid $E$-algebras, the $R$-submodule $g\left((\Ddif^{+, fn}(V_S)/t^k)^{\Gamma_K}_{\tau_0}\right)$ of $(\Ddif^{+, fn}(V_R)/t^k)^{\Gamma_K}_{\tau_0}$ is contained in the image of 
\[\left(K\otimes_{K_0}D_{\mathrm{rig}}^\dagger(V_R)^{\varphi^f=g(\alpha),\Gamma_K=1}\right)_{\tau_0} 
\to \left(D_{\mathrm{dif}}^{+,fn}(V_R)/t^k\right)^{\Gamma_K}_{\tau_0}\]
if and only if the map $\Sp(R)\to \Sp(S)$ factors through $\Sp(S(k, \alpha, n, \tau_0))$. 
\end{corollary}

\begin{proof}
This is the same as \cite[Corollary 4.8]{Liu}, i.e., we can take $\Sp(S(k, \alpha, n, \tau_0))$ to be the intersection $\Sp(S(k, \alpha, a, n, \tau_0))$ for all $a\in (\Ddif^{+, fn}(V_S)/t^k)^{\Gamma_K}_{\tau_0}$. 
\end{proof}

\begin{corollary}\label{cor:local input for Xfs}
Keep notations as in the previous corollary. 
Then there exists a Zariski closed $E$-subspace $\Sp(S(k, \alpha, \tau_0))$ of $\Sp(S)$, such that for any map $g\colon  S\to R$ of affinoid $E$-algebras, the $R$-submodule $g\left((\Ddif^{+, fn}(V_S)/t^k)^{\Gamma_K}_{\tau_0}\right)$ of $(\Ddif^{+, fn}(V_R)/t^k)^{\Gamma_K}_{\tau_0}$ is contained in the image of 
\[\left(K\otimes_{K_0}D_{\mathrm{rig}}^\dagger(V_R)\right)^{\varphi^f=g(\alpha),\Gamma_K=1}_{\tau_0} 
\to \left(D_{\mathrm{dif}}^{+,fn}(V_R)/t^k\right)^{\Gamma_K}_{\tau_0}\]
for all sufficiently large $n$ if and only if the map $\Sp(R)\to \Sp(S)$ factors through $\Sp(S(k, \alpha, \tau_0))$. 
\end{corollary}

\begin{proof}
This is the same as \cite[Corollary 4.9]{Liu}, i.e., we can take $\Sp(S(k , \alpha, \tau_0))$ to be the intersection of $\Sp(S(k , \alpha, n, \tau_0))$ for all $n\geq n(V_S)$. 
\end{proof}

\begin{proof}[Proof of Theorem~\ref{thm:X tau0-fs}]
By Lemma~\ref{lem:glue prop of Xfs}, it suffices to treat the case that $X=\Sp(S)$ is affinoid. 
Let
\[X':=\bigcap_{k>\log_{|\pi_K|^{-1}}|\alpha^{-1}|_{\mathrm{sp}, U}}\Sp(S(k, \alpha, \tau_0)).\]
Now for each $i\geq1$, let $X_i'$ be the Zariski closure of 
\[X'_{P(i)_{\tau_0}}=\bigcap_{0\leq j\leq i-1} X'_{Q(-j)_{\tau_0}}. \]
Then we claim that
\[X_{\tau_0-fs}:=\bigcap_{i\geq1} X_i'\]
is the $\tau_0$-finite slope subspace of $X$. 
Let's follow the argument of \cite[Theorem~4.10]{Liu}. 
Fix an $i_0$, such that $X'_i=X'_{i_0}$ for all $i\geq i_0$. 
Hence $X_{\tau_0-fs}=X_i'$ for all $i\geq i_0$. 
Then for any $i\geq i_0$, we have 
\[
(X_{\tau_0-fs})_{P(i)_{\tau_0}}\supset X_{\tau_0-fs}\cap X'_{P(i)_{\tau_0}}=X'_{i_0}\cap X'_{P(i)_{\tau_0}}=X'_{P(i)_{\tau_0}}. 
\]
Therefore the Zariski closure of $(X_{\tau_0-fs})_{P(i)_{\tau_0}}$ contains the Zariski closure of $X'_{P(i)_{\tau_0}}$, which is $X_i'=X_{\tau_0-fs}$; this yields that $X_{\tau_0-fs}$ satisfies (1) of Definition~\ref{def:X tau0-fs}. 

Now suppose that $g\colon \Sp(R)\to \Sp(S)$ is a morphism of affinoid spaces over $E$ which factors through $X_{Q(j)_{\tau_0}}$ for every $j\leq0$. 
It follows from Lemma~\ref{lem:tau0 Liu's cor2.27} that for each $k\geq1$ and $n\geq n(V_R)$, the natural map 
\[
\left(\Ddif^{+, fn}(V_R)\right)^{\Gamma_K}_{\tau_0}\to \left(\Ddif^{+, fn}(V_R)/t^k\right)^{\Gamma_K}_{\tau_0}
\]
is an isomorphism. 
Hence \eqref{eq:tau0-fs isom} is an isomorphism if and only if the natural map 
\begin{equation}\label{eq:tau0 Liu's 4.7}
\left(K\otimes_{K_0}(D_{\mathrm{rig}}^\dagger(V_R))^{\varphi^f=g^*\alpha,\Gamma_K=1}\right)_{\tau_0} 
\to \left(D_{\mathrm{dif}}^{+,fn}(V_R)/t^k\right)^{\Gamma_K}_{\tau_0}
\end{equation}
is surjective for some (hence any) $k\geq1$. By Lemma~\ref{lem:tau0 Liu's cor2.30}, the natural map 
\[
\left(\Ddif^{+, fn}(V_S)/t^k\right)^{\Gamma_K}_{\tau_0}\otimes_SR\to \left(\Ddif^{+, fn}(V_R)/t^k\right)^{\Gamma_K}_{\tau_0}
\]
is an isomorphism. 
Hence by Corollary~\ref{cor:n submod in image}, the map \eqref{eq:tau0 Liu's 4.7} is surjective if and only if the map $g\colon \Sp(R)\to \Sp(S)$ factors through $\Sp(S(k, \alpha, n, \tau_0))$ for each $k>\log_{|\pi_K^{-1}|}|\alpha^{-1}|_{\mathrm{sp}, U}$. 
By Corollary~\ref{cor:local input for Xfs}, we thus conclude that \eqref{eq:tau0-fs isom} is an isomorphism for all sufficiently large $n$ if and only if $g\colon \Sp(R)\to \Sp(S)$ factors through $X_{\tau_0-fs}$. 
This yields that $X_{\tau_0-fs}$ satisfies (2) of Definition~\ref{def:X tau0-fs}. 
\end{proof}

\begin{proposition}\label{prop:isom on Xtau0fs}
Suppose $U=\Sp(S)$ is an affinoid subdomain of $X_{\tau_0-fs}$.
Then for any positive integer $k>\log_{|\pi_K|^{-1}}|\alpha^{-1}|_{\mathrm{sp}, U}$, we have $S(k, \alpha, \tau_0)=S$.
In particular, for such $k$, the natural morphism
\[
\left(K\otimes_{K_0}(D_{\mathrm{rig}}^\dagger(V_S))^{\varphi^f=\alpha,\Gamma_K=1}\right)_{\tau_0}\to \left(\Ddif^{+, fn}(V_S)/t^k\right)^{\Gamma_K}_{\tau_0}
\]
is an isomorphism for all $n\geq n(V_S)$. 
\end{proposition}

\begin{proof}
This is the same as \cite[Proposition~4.11]{Liu}. 
\end{proof}

\subsection{Modified form of Liu's theorem}

Keep the notation and assumption the same as the previous section. 
In \cite[\S~5]{Liu}, Liu shows that if the family $V_X$ is weakly refined, then $X_{fs}=X$. 
We will modify his argument to establish a criterion for $X_{\tau_0-fs}=X$. 

\begin{theorem}\label{thm:modified-liu}
Suppose that the quadruple $(X,\alpha,V_X, \tau_0)$ satisfies the following conditions. 
\begin{enumerate}[label=(\arabic*), ref=(\arabic*)]
\item\label{cond:modified Liu (2)}
For any positive integer $k$, $X_{P(k)_{\tau_0}}$ is Zariski-open and dense in $X$. 
\item\label{cond:modified Liu (1)} 
For any $x\in X$, the $K_0\otimes_{\BQ_p}k(x)$-module
\begin{equation}
(\Drig(V_x))^{\varphi^f=\alpha(x),\,\Gamma_K=1}\neq 0. \label{eq:almost pointwise trianguline}
\end{equation}
Note that $\varphi$ acts on this module faithfully and permutes its components transitively under the decomposition induced by $K_0\otimes_{\BQ_p}k(x)\simeq \prod_{\sigma\in\Gal(K_0/\BQ_p)}k(x)$, this condition is equivalent to  for any $x\in X$ and any $\sigma\in \Gal(K_0/\BQ_p)$, 
\begin{equation}\label{eq:almost pointwise trianguline, component}
\dim_{k(x)}\left(\Drig(V_x)\right)^{\varphi^f=\alpha(x),\,\Gamma_K=1}_{\sigma}\geq 1, 
\end{equation}
i.e., $(\Drig(V_x))^{\varphi^f=\alpha(x),\,\Gamma_K=1}$ contains a free $K_0\otimes_{\BQ_p}k(x)$-submodule of rank $1$. 
Alternatively, this is also equivalent to the following: for any $x\in X$ and any $\tau\in H_K$, 
\begin{equation}\label{eq:almost pointwise triangulin, tau component}
\dim_{k(x)}\left(K\otimes_{K_0}(D_{\mathrm{rig}}^\dagger(V_x))^{\varphi^f=\alpha(x),\Gamma_K=1}\right)_{\tau}\geq 1. 
\end{equation}
\end{enumerate}
Then $X_{\tau_0-fs}=X$. 
\end{theorem}

\begin{proof}
By Lemma~\ref{lem:glue prop of Xfs}, it suffices to show that for any affinoid subdomain $U=\Sp(S)\subset X, U_{\tau_0-fs}=U$. 
For any positive integer $k>\log_{|\pi_K|^{-1}}|\alpha^{-1}|_{\mathrm{sp}, U}$, we claim 
\[
U_{P(k)_{\tau_0}}\subset \Sp(S(k, \alpha, \tau_0)).
\]
In fact, let $x\in U_{P(k)_{\tau_0}}$. 
For any $n\geq n(V_S)$, Lemma~\ref{lem:tau0-inj} tells us that the natural map
\begin{equation}\label{eq:Liu's injection}
\iota_{n,K, \tau_0}\colon 
\left(K\otimes_{K_0}(D_{\mathrm{rig}}^\dagger(V_x))^{\varphi^f=\alpha(x),\Gamma_K=1}\right)_{\tau_0}
\hookrightarrow\left(D_{\mathrm{dif}}^{+,fn}(V_x)/t^k\right)^{\Gamma_K}_{\tau_0}
\end{equation}
is injective. 
Following the proof of \cite[Lemma~5.2]{Liu}, we will show that it is an isomorphism. 
Indeed, since $x\in U_{P(k)_{\tau_0}}$, by Lemma~\ref{lem:tau0 Liu's cor2.27}, the natural map
\[
\left(\Ddif^{+, fn}(V_x)/t^k\right)^{\Gamma_K}_{\tau_0}\to \left(\Dsen^{fn}(V_x)\right)^{\Gamma_K}_{\tau_0}
\]
is an isomorphism. 
Moreover, $0$ is a simple root of $TQ(T)_{\tau_0}(x)$, so $\dim_{k(x)}\left(\Dsen^{fn}(V_x)\right)^{\Gamma_K}_{\tau_0}\leq1$. 
Hence 
\begin{equation}
\dim_{k(x)}\left(\Ddif^{+, fn}(V_x)/t^k\right)^{\Gamma_K}_{\tau_0}\leq 1. 
\label{eq:upper bound of D_dif}
\end{equation}
On the other hand, by the inequality (\ref{eq:almost pointwise triangulin, tau component}) in condition~\ref{cond:modified Liu (1)}, the left hand side of (\ref{eq:Liu's injection}) has $k(x)$-dimension at least $1$. 
Putting everything together, we conclude that (\ref{eq:Liu's injection}) is an isomorphism, thus $x\in \Sp(S(k, \alpha, \tau_0))$ and the claim holds. 

By our condition~\ref{cond:modified Liu (2)}, $U_{P(k)_{\tau_0}}$ is Zariski-dense in $U$. 
Since $\Sp(S(k, \alpha, \tau_0))$ is Zariski closed in $U$, $\Sp(S(k, \alpha, \tau_0))=U$. 
We conclude $U_{\tau_0-fs}=U$ following the construction of the $\tau_0$-finite slope subspace. 
\end{proof}

\begin{remark}\label{rmk:dim of Ddif+}
From the proof, we see that $\dim_{k(x)}\left(\Ddif^{+, fn}(V_x)/t^k\right)^{\Gamma_K}_{\tau_0}=1$ for any $n\geq n(V_S)$ and $x\in U_{P(k)_{\tau_0}}$. 
\end{remark} 
\section{\texorpdfstring{$G_K$}{G\_K}-version of the Diao--Liu argument}
\label{sec:Diao-Liu}

\subsection{The setup and goal}

Keep the notation as in the previous section. 
Let $S=E\langle T\rangle, \CD=\Sp(S)$ be the closed unit disc, with center $0$ defined by $T=0$, and $\CD^*=\CD\setminus\{0\}$.
Let $V_{\CD}$ be a family of $G_K$-representations on $\CD, \alpha\in\CO(\CD^*)^\times\cap\CO(\CD)$. 
Still fix $\tau_0\in H_K$. 

\begin{theorem}\label{thm:DL-general}
Assume that the quadruple $(\CD, \alpha, V_{\CD}, \tau_0)$ satisfies the following conditions. 
\begin{enumerate}
\item the $\tau_0$-component of the Sen polynomial on $\CD^*$ has the form $TQ(T)_{\tau_0}$; 
\item the conditions~\ref{cond:modified Liu (2)} and \ref{cond:modified Liu (1)} in Theorem~\ref{thm:modified-liu} are satisfied for the quadruple $(X:=\CD^*, \alpha, V_X:=V_{\CD}|_{\CD^*}, \tau_0)$. 
\end{enumerate}
Then $\alpha(0)\neq0$.
Consequently, $\alpha\in\CO(\CD)^\times$.
\end{theorem}

The proof occupies the rest of this section, which generalizes the argument of \cite{DiaoLiu}.

\subsection{De Rham periods vs. crystalline periods}

Let $U=\Sp(R)\subset\CD^*$ be an affinoid subdomain. 
It is smooth and one-dimensional.
Choose a positive integer $k>\log_{|\pi_K|^{-1}}|\alpha^{-1}|_{\mathrm{sp}, U}$. 
By Theorem~\ref{thm:modified-liu} and Proposition~\ref{prop:isom on Xtau0fs}, for any $n\geq n(V_R)$ the natural morphism
\begin{equation}\label{eq:Liu's comparison Thm}
\left(K\otimes_{K_0}(D_{\mathrm{rig}}^\dagger(V_R))^{\varphi^f=\alpha,\Gamma_K=1}\right)_{\tau_0}\xrightarrow{\sim}
\left(D_{\mathrm{dif}}^{+,fn}(V_R)/t^k\right)^{\Gamma_K}_{\tau_0}
\end{equation}
is an isomorphism. 
Therefore, the natural morphism
\begin{equation}\label{eq:cor of Liu's comparison Thm}
\left(K\otimes_{K_0}(D_{\mathrm{rig}}^\dagger(V_R))^{\varphi^f=\alpha,\Gamma_K=1}\right)_{\tau_0}\xrightarrow{\sim}
\left(D_{\mathrm{dif}}^{+}(V_R)/t^k\right)^{\Gamma_K}_{\tau_0}
\end{equation}
is also an isomorphism, and the natural inclusion from the right hand side of \eqref{eq:Liu's comparison Thm} to that of \eqref{eq:cor of Liu's comparison Thm} is an equality. 

\begin{proposition}\label{prop:rank-one}
The two sides of \eqref{eq:cor of Liu's comparison Thm} are locally free of rank one over $R$. 
\end{proposition}

\begin{proof}
Since $U=\Sp(R)$ is smooth of dimension one, the $R$-module $(\Ddif^{+}(V_R)/t^k)^{\Gamma_K}_{\tau_0}$ is locally free because it is finitely generated by \cite[Proposition~2.21(3)]{Liu} and torsion-free by \cite[Proposition~2.23]{Liu}. 
It remains to show that its rank is one.
Indeed, for any $x\in U_{P(k)_{\tau_0}}$, which is a Zariski-dense subset of $U$, by Lemma~\ref{lem:tau0 Liu's cor2.30}, the natural map
\[
(\Ddif^{+}(V_R)/t^k)^{\Gamma_K}_{\tau_0}\otimes_{R}k(x)\to (\Ddif^{+}(V_x)/t^k)^{\Gamma_K}_{\tau_0}
\]
is an isomorphism. The right hand side is of $k(x)$-dimension one by Remark~\ref{rmk:dim of Ddif+}. 
So we are done. 
\end{proof}

Passing to the inverse limit in $k$ in \eqref{eq:cor of Liu's comparison Thm} gives
\[
\left(K\otimes_{K_0}(D_{\mathrm{rig}}^\dagger(V_R))^{\varphi^f=\alpha,\Gamma_K=1}\right)_{\tau_0}\xrightarrow{\sim}\Ddif^{+}(V_R)^{\Gamma_K}_{\tau_0}.
\]
Using Theorem~\ref{thm:comparison of various D}, we obtain the following corollary. 

\begin{corollary}\label{cor:crys vs dR on punctured disc}
For every affinoid $U=\Sp(R)\subset\CD^*$, there is a natural isomorphism
\begin{equation}\label{eq:generalized Diao-Liu's Cor4.4}
\left(K\otimes_{K_0}D_{\mathrm{crys}}^+(V_R)^{\varphi^f=\alpha}\right)_{\tau_0}
\xrightarrow{\sim}D_{\mathrm{dR}}^+(V_R)_{\tau_0}.
\end{equation}
Both sides are locally free of rank one over $R$.
\end{corollary}

For integers $n\geq0$ (resp., $n'>n\geq0$), let $S_n=E\langle p^{-n}T\rangle$ (resp., $S_{n,n'}=E\langle p^{-n}T, p^{n'}T^{-1}\rangle$).
Let $V_{n}$ (resp., $V_{n, n'}$) be the restriction of $V_{\CD}$ on $\Sp(S_n)$ (resp., $\Sp(S_{n, n'})$). 
Set $\bB_{\mathrm{crys},K}^+=K\otimes_{K_0}\bB_{\mathrm{crys}}^+$.
There is a natural continuous injection $\bB_{\mathrm{crys}, K}^+\to \bB_{\mathrm{dR}}^{+}$. 

\begin{lemma}
\begin{enumerate}
\item 
For any $n\geq0$ (resp., $n'>n\geq0$), the continuous map $\bB_{\mathrm{crys}, K}^+\to \bB_{\mathrm{dR}}^{+}$ induces a natural inclusion 
$S_n\widehat{\otimes}_{\BQ_p}\bB_{\mathrm{crys}, K}^+\hookrightarrow S_n\widehat{\otimes}_{\BQ_p}\bB_{\mathrm{dR}}^+$ (resp., $S_{n, n'}\widehat{\otimes}_{\BQ_p}\bB_{\mathrm{crys}, K}^+\hookrightarrow S_{n, n'}\widehat{\otimes}_{\BQ_p}\bB_{\mathrm{dR}}^+$). 
\item\label{lem:intersection}
For any $x\in S_n\widehat{\otimes}_{\BQ_p}\bB_{\mathrm{dR}}^+$, if its image in $S_{n, n'}\widehat{\otimes}_{\BQ_p}\bB_{\mathrm{dR}}^+$ belongs to $S_{n, n'}\widehat{\otimes}_{\BQ_p}\bB_{\mathrm{crys}, K}^+$, 
then $x\in S_n\widehat{\otimes}_{\BQ_p}\bB_{\mathrm{crys}, K}^+$. 
\end{enumerate}
\end{lemma}

\begin{proof}
The proof of \cite[Section~5]{DiaoLiu} goes through verbatim, with $\bB_{\mathrm{crys}}^+$ replaced by $\bB_{\mathrm{crys}, K}^+$ throughout. 
\footnote{Our field $E$ is denoted by $L$ there. }
\end{proof}

\begin{corollary}\label{cor:crys vs dR on disc}
We have 
\[
(K\otimes_{K_0}D_{\mathrm{crys}}^+(V_\CD))_{\tau_0}=\DdRplus(V_{\CD})_{\tau_0}. 
\]
\end{corollary}

\begin{proof}
We repeat the argument of \cite[Corollary~5.7]{DiaoLiu}.  Pick an $S$-basis $e_1, \cdots, e_d$ of $V_{\CD}$. 
By Corollary~\ref{cor:crys vs dR on punctured disc}, we get $(K\otimes_{K_0}\Dcrysplus(V_{0, 1}))_{\tau_0}=\DdRplus(V_{0, 1})_{\tau_0}$. 
This implies that, for any $a_{\tau_0}=\sum_{i=1}^da_{i, \tau_0}e_i\in\DdRplus(V_{\CD})_{\tau_0}$, where $a_{i, \tau_0}\in (S\widehat{\otimes}_{\BQ_p}\bB_{\mathrm{dR}}^+)_{\tau_0}$, the images of $a_{i, \tau_0}$ in $(S_{0, 1}\widehat{\otimes}_{\BQ_p}\bB_{\mathrm{dR}}^+)_{\tau_0}$ belong to $(S_{0, 1}\widehat{\otimes}_{\BQ_p}\bB_{\mathrm{crys}, K}^+)_{\tau_0}$. 
Applying the previous lemma, we obtain that $a_{i, \tau_0}$ belongs to $S\widehat{\otimes}_{\BQ_p}\bB_{\mathrm{crys}, K}^+$, thus lies in $(S\widehat{\otimes}_{\BQ_p}\bB_{\mathrm{crys}, K}^+)_{\tau_0}$; this yields $\DdRplus(V_{\CD})_{\tau_0}\subset (K\otimes_{K_0}\Dcrysplus(V_{\CD}))_{\tau_0}$. 
Hence $\DdRplus(V_{\CD})_{\tau_0}=(K\otimes_{K_0}\Dcrysplus(V_{\CD}))_{\tau_0}$. 
\end{proof}

\subsection{Proof of Theorem~\ref{thm:DL-general}}

Applying Proposition~\ref{prop:flat bc of DdRplus}, we obtain an isomorphism $\DdRplus(V_{\CD})_{\tau_0}\otimes_{S}S_{0, 1}\xrightarrow{\sim}\DdRplus(V_{0, 1})_{\tau_0}$. 
By Corollary~\ref{cor:crys vs dR on punctured disc}, $\DdRplus(V_{0, 1})_{\tau_0}$ is a locally free $S_{0, 1}$-module of rank one. 
In particular, this implies that $\DdRplus(V_{\CD})_{\tau_0}\neq0$. 
Now pick a nonzero element $e_{\tau_0}\in \DdRplus(V_{\CD})_{\tau_0}$. 
By dividing a suitable power of $T$, we may assume that the specialization $e_{\tau_0, 0}$ of $e$ at the puncture $0$ is nonzero. 
By Corollary~\ref{cor:crys vs dR on disc} $e_{\tau_0}\in (K\otimes_{K_0}\Dcrysplus(V_{\CD}))_{\tau_0}$. 
Note that for any $K_0\otimes_{\BQ_p}E$-module, we have a functorial inclusion
\[
(K\otimes_{K_0}M)_{\tau_0}\simeq M_{\sigma_0}\hookrightarrow M, 
\]
where $\sigma_0\in \Gal(K_0/\BQ_p)$ is the restriction of $\tau_0$. 
So we have the following commutative diagram
\[
\begin{tikzcd}
(K\otimes_{K_0}\Dcrysplus(V_{\CD}))_{\tau_0} \arrow[r, hook] \arrow[d]                             & \Dcrysplus(V_{\CD}) \arrow[d]                             \\
{(K\otimes_{K_0}\Dcrysplus(V_{0, 1}))_{\tau_0}} \arrow[r, hook]                                    & {\Dcrysplus(V_{0, 1})}                                    \\
{(K\otimes_{K_0}\Dcrysplus(V_{0, 1})^{\varphi^f=\alpha})_{\tau_0}} \arrow[r, hook] \arrow[u, "\simeq", hook] & {\Dcrysplus(V_{0, 1})^{\varphi^f=\alpha}} \arrow[u, hook]
\end{tikzcd}
\]
Let $e\in \Dcrysplus(V_{\CD})$ be the image of $e_{\tau_0}$ under the above inclusion. 
Then the specialization $e_0$ of $e$ at the puncture $0$ is also nonzero. 
By Corollary~\ref{cor:crys vs dR on punctured disc}, the left bottom inclusion is an isomorphism. 
An easy diagram chase shows that the image of $e$ in $\Dcrysplus(V_{0, 1})$ belongs to $\Dcrysplus(V_{0, 1})^{\varphi^f=\alpha}$. 
That is, $\varphi^f(e)=\alpha e$ on $\Sp(S_{0, 1})$. 
Since $\Sp(S_{0, 1})$ is Zariski-dense in $\CD$, we must have $\varphi^f(e)=\alpha e$ on the entire disk. 
In particular, $\varphi^f(e_0)=\alpha(0)e_0$. 
Using the fact that $\varphi$ is injective on $\Dcrysplus(V_0)$, we conclude that $\alpha(0)\neq0$. 
Thus $\alpha\in\CO(\CD)^\times$.  
\section{A variant of the Herr complex}\label{sec:phi-f-herr}

\subsection{The setup and goal}

Let $S$ be a $\BQ_p$-affinoid algebra and let $D$ be a finite projective $(\varphi,\Gamma_K)$-module over
$\bB_{\mathrm{rig},K,S}^{\dagger}$. 
Let $\Delta_K\subset\Gamma_K$ be the torsion subgroup, choose $\gamma_K\in \Gamma_K$ whose image in $\Gamma_K/\Delta_K$ is a topological generator.
The goal of this section is to prove the following theorem.
\begin{theorem}\label{thm:herr-main}
Let $\alpha\in S^\times$.  Then
\[
C^\bullet_{\varphi^f-\alpha,\gamma_K}(D)
=
\left[
D^{\Delta_K}
\xrightarrow{(\varphi^f-\alpha,\ \gamma_K-1)}
D^{\Delta_K}\oplus D^{\Delta_K}
\xrightarrow{(\gamma_K-1)\oplus(\alpha-\varphi^f)}
D^{\Delta_K}
\right]
\] 
is a perfect complex of $S$-modules concentrated in degrees in $[0, 2]$. 
In particular, the function $s\mapsto \dim_{k(s)}\left(D\widehat{\otimes}_Sk(s)\right)^{\varphi^f=\alpha(s), \Gamma_K=1}$ is upper semi-continuous on $\Sp(S)$ with respect to the Zariski topology. 
\end{theorem}

\begin{remark}
If we only assume $\alpha\in S$, the resulting statement may not hold. 
For example, let $K=\BQ_p, \chi\colon  \Gamma:=\Gamma_{K}\to \BZ_p^\times$ be the cyclotomic character, $p>2$. 
Write $\chi(g)=\omega(\chi(g))\langle\chi(g)\rangle$, where $\omega(\chi(g))\in\mu_{p-1}, \langle\chi(g)\rangle\in (1+p\BZ_p)$. 
Then there exists a unique $z(g)\in \BZ_p$, such that $\langle\chi(g)\rangle=(1+p)^{z(g)}$, and the map $z\colon \Gamma\to \BZ_p$ is a continuous homomorphism. 
Let $S=\BQ_p\langle X, Y\rangle$. Consider the continuous character 
\[
\eta\colon \Gamma\to S^\times, \quad g\mapsto (1+pX)^{z(g)}=\sum_{n\geq0}\binom{z(g)}{n}p^nX^n. 
\]
Let $D$ be the associated $(\varphi, \Gamma)$-module, i.e. $D$ is free of rank one on $\bB_{\rig, K, S}^{\dagger}$ with a basis $e$, where $\varphi(e)=e, g(e)=\eta(g)e$ for any $g\in\Gamma$. 
For any $m\in \BZ_{\geq1}$, let $k_m=(p-1)p^{m-1}, x_m=\frac{(1+p)^{-k_m}-1}{p}, y_m=p^{k_m}$ and $e_m=t^{k_m}e$. 
Then it's easy to check that after specializing to $s_m=(x_m, y_m)\in\Sp(S)$, we have $g(e_m)=e_m$ and $\varphi(e_m)=y_me_m$, hence $\dim_{k(s_m)}(D\widehat{\otimes}_{S}k(s_m))^{\varphi=Y(s_m), \Gamma=1}\geq1$. 
However, at $s_0=(0, 0), Y(s_0)=0$, so $\dim_{k(s_0)}(D\widehat{\otimes}_{S}k(s_0))^{\varphi=Y(s_0)}=0$.
Since $s_m\to s_0$ as $m\to \infty$, $s_0$ lies in the Zariski closure of $\{s_m\}_{m\geq1}$. 
So the function $s\mapsto \dim_{k(s)}\left(D\widehat{\otimes}_Sk(s)\right)^{\varphi^f=Y(s), \Gamma_K=1}$ is not upper semi-continuous. 
\end{remark}

\begin{corollary}\label{cor:derived bc of herr cpx}
Let $g\colon \Sp(R)\to \Sp(S)$ be a morphism of $\BQ_p$-affinoid spaces. 
Then the canonical morphism 
\[
C^\bullet_{\varphi^f-\alpha,\gamma_K}(D)\otimes^{\BL}_{S}R\to C^\bullet_{\varphi^f-g^*(\alpha),\gamma_K}(D\widehat{\otimes}_SR)
\]
is an isomorphism. 
In particular, if $S=E'\to R=E''$ are finite extensions of $\BQ_p$, we have 
\[
D^{\varphi^f=\alpha, \Gamma_K=1}\otimes_{E'}E''\simeq (D\widehat{\otimes}_{E'}E'')^{\varphi^f=\alpha, \Gamma_K=1}.
\]
\end{corollary}

\begin{proof}
This is the same as the proof of \cite[Theorem~4.4.3]{KPX}. 
More precisely, one first treats the case where $R=k(x)$ and $g$ is the natural map for a closed point $x\in\Sp(S)$, and then deduces the general case using Theorem~\ref{thm:herr-main} and \cite[Lemma~4.1.5]{KPX}. 
For details, see \cite[Theorem~4.4.3]{KPX}. 
\end{proof}

\subsection{\texorpdfstring{$\varphi^f$}{phi-f}-modules}

For any $r\geq s>0$, iterating the operator $\varphi$, we get
\[
\varphi^f\colon \bB_{K,S}^{[s,r]}\longrightarrow\bB_{K,S}^{[qs,qr]}, 
\]
where $q=p^f$. This induces
\[
\varphi^f\colon \bB_{\mathrm{rig},K,S}^{\dagger,s}\longrightarrow\bB_{\mathrm{rig},K,S}^{\dagger,qs}
\]
and 
\[
\varphi^f\colon \bB_{\mathrm{rig},K,S}^{\dagger}\longrightarrow\bB_{\mathrm{rig},K,S}^{\dagger}. 
\]
They are $S$-linear and commute with the action of $\Gamma_K$. 
For a finite projective module $D^s$ over $\bB_{\mathrm{rig},K,S}^{\dagger,s}$, we write $D^r:=\bB_{\mathrm{rig},K,S}^{\dagger,r}\otimes_{\bB_{\mathrm{rig},K,S}^{\dagger,s}} D^s$ and $D^{[s, r]}:=\bB_{K,S}^{[s, r]}\otimes_{\bB_{\mathrm{rig},K,S}^{\dagger,s}} D^s$. 

\begin{definition}
\begin{itemize}
\item 
A \emph{$\varphi^f$-module over $\bB_{\mathrm{rig},K,S}^{\dagger,s}$} is a finite projective module $D^s$ together with an isomorphism
$\Phi_{D^s}\colon  (\varphi^f)^*D^s\xrightarrow{\sim}D^{qs}$.
Equivalently, $\Phi_{D^s}$ is a $\varphi^f$-semilinear map $D^s\to D^{qs}$, whose linearization is an isomorphism. 
\item 
Let $0<qs\le r$.
Note that $[s, r]\cap [qs, qr]=[qs, r]$. 
A \emph{$\varphi^f$-module over $\bB_{K,S}^{[s, r]}$} is a finite projective module $D^{[s,r]}$ together with an isomorphism
$\Phi_{D^{[s, r]}}\colon  \bB_{K,S}^{[qs,r]}\otimes_{\bB_{K,S}^{[qs,qr]}}(\varphi^f)^*D^{[s,r]}\xrightarrow{\sim}
D^{[qs, r]}:=\bB_{K,S}^{[qs, r]}\otimes_{\bB_{K,S}^{[s, r]}} D^{[s, r]}$.
Equivalently, $\Phi_{D^{[s, r]}}$ is a $\varphi^f$-semilinear map $D^{[s, r]}\to D^{[qs, r]}$, whose linearization, after extending the scalar to $\bB_{K,S}^{[qs, r]}$, is an isomorphism. 
\end{itemize}
\end{definition}

Note that if $(D^s, \Phi_{D^s})$ is a $\varphi^f$-module over $\bB_{\mathrm{rig},K,S}^{\dagger, s}$, then for any $r\geq qs, (D^{[s, r]}, \Phi_{D^{[s, r]}}:=1\otimes \Phi_{D^s})$ is a $\varphi^f$-module over $\bB_{K,S}^{[s, r]}$. 

\begin{lemma}\label{lem:phi-f-descent}
Fix $0<qs\le r$.
Then the functor $\bB_{K,S}^{[s, r]}\otimes_{\bB_{\mathrm{rig},K,S}^{\dagger, s}}(-)$ induces an equivalence of categories between $\varphi^f$-modules over
$\bB_{\mathrm{rig},K,S}^{\dagger,s}$ and $\varphi^f$-modules over $\bB_{K,S}^{[s, r]}$. 
\end{lemma}

\begin{proof}
This is  the same as  \cite[Lemma~3.5]{Bellovin}.
Also, see the proof of \cite[Lemma 2.10]{KedlayaPottharst}. 
\end{proof}

\begin{corollary}\label{cor:qisom by equiv of cat}
Let $D^s$ be a $\varphi^f$-module over $\bB_{\mathrm{rig},K,S}^{\dagger,s}$, then the natural morphism of complexes
\[
C_{\Phi_{D^s}}^{\bullet}(D^s):=[D^s\xrightarrow{\Phi_{D^s}-1}D^{qs}]\to[D^{[s, r]}\xrightarrow{\Phi_{D^{[s, r]}}-1} D^{[qs, r]}]=:C_{\Phi_{D^{[s, r]}}}^{\bullet}(D^{[s, r]})
\]
is a quasi-isomorphism for any $r\geq qs$. 
\end{corollary}

\begin{proof}
Let $\mathbf{1}^s$ denote the rank-one trivial $\varphi^f$-module over $\bB_{\mathrm{rig},K,S}^{\dagger,s}, \Ext^1_{\varphi^f}(\mathbf{1}, D^s)$ be the groups of extensions of $\mathbf{1}$ by $D^s$ as $\varphi^f$-modules.
Define $\mathbf{1}^{[s, r]}$, $\Ext^1_{\varphi^f}(\mathbf{1}^{[s, r]}, D^{[s, r]})$ similarly. 
Then the cohomology groups of two complexes can be identified with $\Hom_{\varphi^f}(\mathbf{1}^s, D^s), \Ext^1_{\varphi^f}(\mathbf{1}^s, D^s)$ and $\Hom_{\varphi^f}(\mathbf{1}^{[ s,r ]}, D^{[s, r]})$, $\Ext^1_{\varphi^f}(\mathbf{1}^{[s, r]}, D^{[s, r]})$ respectively. 
Thus the conclusion follows from the previous lemma. 
\end{proof}

\subsection{Cohomology of the variant of the Herr complex}
For the rest of this section, we fix a $(\varphi, \Gamma_K)$-module $(D, \varphi_D)$ over $\bB_{\mathrm{rig},K,S}^{\dagger}$. 
It is the base change of a finite projective module $D^s$ over $\bB_{\mathrm{rig},K,S}^{\dagger,s}$ to $\bB_{\mathrm{rig},K,S}^{\dagger}$ for some $s$, equipped with a $\varphi$-semilinear map $\varphi_{D^s}\colon D^s\to D^{ps}$ whose linearization is an isomorphism $\varphi_{D^s}\colon  \varphi^*D^s\xrightarrow{\sim} D^{ps}$. 
Iterating $\varphi_{D^s}$, we get a $\varphi^f$-module structure of $D^s$ over $\bB_{\mathrm{rig},K,S}^{\dagger,s}$, i.e. 
\[
\Phi_{D^s}\colon  (\varphi^f)^*D^s\xrightarrow{\sim} D^{qs}. 
\]
More precisely, for any $i\geq1$, base change $\varphi_{D^s}$ along the natural morphism $\bB_{\mathrm{rig},K,S}^{\dagger,ps}\to \bB_{\mathrm{rig},K,S}^{\dagger,p^{i+1}s}$, we get 
\[
\varphi_{D^{p^is}}\colon  \varphi^*D^{p^is}\simeq \bB_{\mathrm{rig},K,S}^{\dagger,p^{i+1}s}\otimes_{\bB_{\mathrm{rig},K,S}^{\dagger,ps}} (\varphi)^*D^s\xrightarrow{\sim} D^{p^{i+1}s}. 
\]
Then $\Phi_{D^s}:=\varphi_{D^{p^{f-1}s}}\circ\cdots \circ (\varphi^{f-2})^*\varphi_{D^{ps}}\circ(\varphi^{f-1})^*\varphi_{D^s}$ is the composition 
\[
\begin{tikzcd}[column sep=large]
(\varphi^f)^*D^s \arrow[r, "(\varphi^{f-1})^*\varphi_{D^s}"] & (\varphi^{f-1})^*D^{ps} \arrow[r, "(\varphi^{f-2})^*\varphi_{D^{ps}}"] & \cdots \arrow[r, "\varphi_{D^{p^{f-1}s}}"] & D^{qs}. 
\end{tikzcd}
\]
It's easy to check $1\otimes \Phi_{D^s}=\varphi_D^f$ on $D$. 
Abusing the notation, we still denote various (linearization of) $\varphi_{D^{p^{i}s}}$ and their pullback along $\varphi^j$ by $\varphi_D$. 

Recall that there is a natural morphism 
\[
\psi\colon  \bB_{\mathrm{rig},K,S}^{\dagger,ps}\to \bB_{\mathrm{rig},K,S}^{\dagger,s}
\]
which is a left inverse of $\varphi$. 
Let
\[\psi_{D^s}\colon  D^{ps}\simeq \bB_{\mathrm{rig},K,S}^{\dagger,ps}\otimes_{\varphi, \bB_{\mathrm{rig},K,S}^{\dagger,s}}D^s \to D^s, \quad 
b\otimes d \mapsto \psi(b)d. \]
Then $\psi_{D^s}$ is a left inverse of $\varphi_D$. 
Similarly, for each $i\geq1$ we can define 
\[\psi_{D^{p^{i}s}}\colon  D^{p^{i+1}s}\simeq \bB_{\mathrm{rig},K,S}^{\dagger,p^{i+1}s}\otimes_{\varphi, \bB_{\mathrm{rig},K,S}^{\dagger,p^{i}s}}D^{p^{i}s} \to D^{p^{i}s}, \quad
b\otimes d \mapsto \psi(b)d. \]
Then $\psi$ is a left inverse of $\varphi_D$. 
Put $\Psi_{D^s}:=\psi_{D^s}\circ\cdots\circ \psi_{D^{p^{f-1}s}}$. 
The upshot is that we have the following picture
\[
\begin{tikzcd}[column sep=large]
D^s
  \arrow[r, bend left=18, "\varphi_D"'{pos=.35}]
  \arrow[rrrr, bend left=25, "\Phi"]
&
D^{ps}
  \arrow[r, bend left=18, "\varphi_D"]
  \arrow[l, bend left=18, "\psi_D"{pos=.35}]
&
\cdots
  \arrow[r, bend left=18, "\varphi_D"]
  \arrow[l, bend left=18, "\psi_D"]
&
D^{p^{f-1}s}
  \arrow[r, bend left=18, "\varphi_D"{pos=.35}]
  \arrow[l, bend left=18, "\psi_D"]
&
D^{qs}
  \arrow[l, bend left=18, "\psi_D"'{pos=.35}]
  \arrow[llll, bend left=25, "\Psi"]
\end{tikzcd}
\]
and $\Psi_{D^s}\circ\Phi_{D^s}=1$, where we abuse the notation again to denote various $\psi_{D^{p^i}s}$ by $\psi_D$. 

\begin{lemma}
For $1\leq j\leq f$, let
\[
K_j(s):=\ker(\psi_D\colon  D^{p^js}\to D^{p^{j-1}s}). 
\]
Then 
\[
\ker\Psi_{D^s}=K_f(s)\oplus \varphi K_{f-1}(s)\oplus\cdots\oplus \varphi^{f-1}K_1(s). 
\]
\end{lemma}

\begin{proof}
For $x\in D^{qs}$, write $x=(x-\varphi\psi x)+\varphi(\psi x)$. 
The first summand lies in $K_f(s)$, hence 
\[
\ker(\Psi_{D^s})=K_f(s)\oplus \varphi\ker(\psi_D^{f-1}\colon  D^{p^{f-1}s}\to D^s). 
\]
Iterating proves the lemma. 
\end{proof}

Since the action of $\Gamma_K$ commutes with $\varphi_D$, it also commutes with $\psi$, hence it induces an action of $\Gamma_K$ on $K_j(s), \ker\Psi_{D^s}$ and $(\ker\Psi_{D^s})^{\Delta_K}$. 

\begin{lemma}\label{lem:invertible action}
There exists $s_1\geq s$, such that for any $s'\geq s_1$, $\gamma_K-1$ is invertible on $\ker\Psi_{D^{s'}}$ and $(\ker\Psi_{D^{s'}})^{\Delta_K}$. 
\end{lemma}

\begin{proof}
By \cite[Theorem~3.1.1]{KPX}, there exists $s_1\geq s$, such that for any $s'\geq s_1$, $\gamma_K-1$ is invertible on each $K_j(s')$, hence it is invertible on $\ker(\Psi_{D^{s'}})$ and $(\ker\Psi_{D^{s'}})^{\Delta_K}$. 
\end{proof}

For any $r'\geq qs'\geq qs_1$, $\Phi_{D^{s'}}$ induces a $\varphi^f$-module structure on $D^{[s', r']}$. We denote it by $\Phi_{D^{[s', r']}}$. 
Let $\alpha\in S^{\times}$. 
\begin{proposition}\label{prop:closed annulus herr}
The cohomology of the complex 
\[
C^\bullet_{\alpha^{-1}\varphi^f,\gamma_K}(D)
:=\left[D^{\Delta_K}\xrightarrow{(\alpha^{-1}\varphi^f-1, \gamma_K-1)}D^{\Delta_K}\oplus D^{\Delta_K}\xrightarrow{(\gamma_K-1)\oplus(1-\alpha^{-1}\varphi^f)}D^{\Delta_K}\right]
\] 
can be computed by the following complex
\[
\left[(D^{[s', r']})^{\Delta_K}\xrightarrow{(\alpha^{-1}\Phi_{D^{[s', r']}}-1, \gamma_K-1)}(D^{[qs', r']})^{\Delta_K}\oplus (D^{[s', r']})^{\Delta_K}\xrightarrow{(\gamma_K-1)\oplus(1-\alpha^{-1}\Phi_{D^{[s', r']}})}(D^{[qs', r']})^{\Delta_K}\right]
\]
which is denoted by $C^{\bullet}_{\alpha^{-1}\Phi_{D^{[s', r']}},\gamma_K}(D^{[s', r']})$, for any $r'\geq qs'\geq qs_1$.
\end{proposition}

\begin{proof}
Let's follow the argument of \cite[Proposition~3.8]{Bellovin}. 
Define $C^{\bullet}_{\alpha^{-1}\Phi_{D^{s'}},\gamma_K}(D^{s'})$ similarly to the above complexes.
Note that 
\[
C^\bullet_{\alpha^{-1}\varphi^f,\gamma_K}(D)=\varinjlim_{s'\geq s_1} C^{\bullet}_{\alpha^{-1}\Phi_{D^{s'}},\gamma_K}(D^{s'}). 
\]
We claim that the transition maps are quasi-isomorphisms. 
In fact, there are two chain maps 
\[
1, \alpha^{-1}\Phi_{D^{s'}}\colon  C^{\bullet}_{\alpha^{-1}\Phi_{D^{s'}},\gamma_K}(D^{s'})\to C^{\bullet}_{\alpha^{-1}\Phi_{D^{qs'}},\gamma_K}(D^{qs'}), 
\]
where $1$ is the transition map mentioned above, induced by the natural morphism $D^{s'}\to D^{qs'}$. 
Consider the following diagram
\[
\begin{tikzcd}[column sep=large,row sep=large]
(D^{s'})^{\Delta_K} \arrow[rr, "{(\alpha^{-1}\Phi_{D^{s'}}-1, \gamma_K-1)}"] \arrow[d, "\alpha^{-1}\Phi_{D^{s'}}", shift left] \arrow[d, "1"', shift right] &  & (D^{qs'})^{\Delta_K}\oplus (D^{s'})^{\Delta_K} \arrow[rr, "(\gamma_K-1)\oplus(1-\alpha^{-1}\Phi_{D^{s'}})"] \arrow[d, "\alpha^{-1}(\Phi_{D^{qs'}}\oplus\Phi_{D^{s'}})", shift left] \arrow[lld, "\mathrm{pr}_1" description] \arrow[d, "1"', shift right] &  & (D^{qs'})^{\Delta_K} \arrow[d, "\alpha^{-1}\Phi_{D^{qs'}}", shift left] \arrow[lld, "{(0, -1)}" shift left] \arrow[d, "1"', shift right] \\
(D^{qs'})^{\Delta_K} \arrow[rr, "{(\alpha^{-1}\Phi_{D^{qs'}}-1, \gamma_K-1)}"']                                                         &  & (D^{q^2s'})^{\Delta_K}\oplus (D^{qs'})^{\Delta_K} \arrow[rr, "(\gamma_K-1)\oplus(1-\alpha^{-1}\Phi_{D^{qs'}})"']                                                                                                &  & (D^{q^2s'})^{\Delta_K}                                                                                              
\end{tikzcd}
\]
It is easy to see that the two chain maps are homotopic.
Moreover, since $\alpha\Psi_{D^{s'}}\circ \alpha^{-1}\Phi_{D^{s'}}=1$, the cokernel of the chain map $\alpha^{-1}\Phi_{D^{s'}}$ is given by 
\[
\ker(\Psi_{D^{s'}})^{\Delta_K}\xrightarrow{(-1, \gamma_K-1)} \ker(\Psi_{D^{qs'}})^{\Delta_K}\oplus \ker(\Psi_{D^{s'}})^{\Delta_K}\xrightarrow{(\gamma_K-1)\oplus1}\ker(\Psi_{D^{qs'}})^{\Delta_K}. 
\]
By Lemma~\ref{lem:invertible action}, this complex is acyclic, hence the claim holds, and the cohomology of $C^\bullet_{\alpha^{-1}\varphi^f,\gamma_K}(D)$ can be computed by that of $C^\bullet_{\alpha^{-1}\Phi_{D^{s'}},\gamma_K}(D^{s'})$ for any $s'\geq s_1$. 
Apply Corollary~\ref{cor:qisom by equiv of cat} to $D^{s'}$ with the $\varphi^f$-module structure defined by $\alpha^{-1}\Phi_{D^{s'}}$, we see that the natural morphism 
\begin{align*}
&\Fib\left(\gamma_K-1\colon  C_{\alpha^{-1}\Phi_{D^{s'}}}^{\bullet}(D^{s'})\to C_{\alpha^{-1}\Phi_{D^{s'}}}^{\bullet}(D^{s'})\right)\\
\to &\Fib\left(\gamma_K-1\colon  C_{\alpha^{-1}\Phi_{D^{[s', r']}}}^{\bullet}(D^{[s', r']})\to C_{\alpha^{-1}\Phi_{D^{[s', r']}}}^{\bullet}(D^{[s', r']})\right)
\end{align*}
is a quasi-isomorphism for any $r'\geq qs'$. 
Taking the $\Delta_K$-invariant, and since we are in characteristic zero, the resulting natural morphism
\[
C^{\bullet}_{\alpha^{-1}\Phi_{D^{s'}},\gamma_K}(D^{s'})\to C^{\bullet}_{\alpha^{-1}\Phi_{D^{[s', r']}},\gamma_K}(D^{[s', r']})
\]
is still a quasi-isomorphism. So we are done. 
\end{proof}

\begin{proof}[Proof of Theorem~\ref{thm:herr-main}]
Consider the following diagram 
\[
\begin{tikzcd}[column sep=large,row sep=large]
D^{\Delta_K} \arrow[rr, "(\varphi^f-\alpha)\oplus(\gamma_K-1)"] \arrow[d, "1"] &  & D^{\Delta_K}\oplus D^{\Delta_K} \arrow[rr, "(\gamma_K-1)\oplus(\alpha-\varphi^f)"] \arrow[d, "{(\alpha^{-1}, 1)}"] &  & D^{\Delta_K} \arrow[d, "\alpha^{-1}"] \\
D^{\Delta_K} \arrow[rr, "(\alpha^{-1}\varphi^f-1)\oplus(\gamma_K-1)"']         &  & D^{\Delta_K}\oplus D^{\Delta_K} \arrow[rr, "(\gamma_K-1)\oplus(1-\alpha^{-1}\varphi^f)"']                          &  & D^{\Delta_K}.                          
\end{tikzcd}
\]
We see that the complex $C^\bullet_{\varphi^f-\alpha, \gamma_K}(D)$ is isomorphic to $C^{\bullet}_{\alpha^{-1}\varphi^f, \gamma_K}(D)$. 
So it suffices to show the latter complex is a perfect complex of $S$-modules concentrated in degrees in $[0, 2]$. 
Indeed, choose $s', r', s'', r''$ with $r'\geq qs'\geq qs_1, r''\geq q^2s'$ and $[qs', r'']\subset(s', r')$. 
Then we have the following commutative square
\[
\begin{tikzcd}
{C^{\bullet}_{\alpha^{-1}\Phi_{D^{s'}},\gamma_K}(D^{s'})} \arrow[r] \arrow[d]   & {C^{\bullet}_{\alpha^{-1}\Phi_{D^{qs'}},\gamma_K}(D^{qs'})} \arrow[d]     \\
{C^{\bullet}_{\alpha^{-1}\Phi_{D^{[s', r']}},\gamma_K}(D^{[s', r']})} \arrow[r] & {C^{\bullet}_{\alpha^{-1}\Phi_{D^{[qs', r'']}},\gamma_K}(D^{[qs', r'']})}
\end{tikzcd}
\]
By the proof of Proposition~\ref{prop:closed annulus herr}, the top horizontal map and two vertical maps are quasi-isomorphisms, so the bottom map is also a quasi-isomorphism. 
Moreover, since $[qs', r'']\subset(s', r')$, each individual map in the bottom is a completely continuous morphism of Banach spaces over $S$, hence their cohomology groups are finite $S$-modules by \cite[Lemma~1.10]{KedlayaLiu}. 
Thus by Proposition~\ref{prop:closed annulus herr}, the same property holds for the cohomology groups of $C_{\varphi^f-\alpha, \gamma_K}^{\bullet}$. 
Finally, as in the proof of \cite[Theorem~4.4.5(1)]{KPX}, the desired conclusion follows from \cite[Corollary~2.1.5]{KPX}, which gives the flatness of $D^{\Delta_K}$ over $S$. 
\end{proof}
\section{Overconvergent modular forms for \texorpdfstring{$G^*$}{G*}}\label{sec:oc MFs for G*}

In this section, we review the construction of families of cuspidal overconvergent modular forms for $G^*$ in
\cite[\S~3]{AIP} and establish the base-change results needed for our analysis of the spectral variety.

In particular, we study base change from a fixed weight and along regular closed immersions.
These results (Proposition~\ref{prop:constant family for Gstar} and Proposition~\ref{prop:reg closed im bc of oc MFs of G*}) will be used in the next section to identify the restriction of the spectral variety to certain closed weight loci, which naturally appear in the cases beyond the scope of the Diao--Liu method when there is only one prime ideal of $\CO_F$ above $p$. 

As mentioned in the proof of \cite[Proposition~3.22]{AIP}, to ensure that the Banach modules of overconvergent modular forms satisfy the property $(Pr)$, 
it is better to take the global sections of the families of modular sheaves on the affinoid strict neighbourhood of the ordinary locus defined by the total Hasse invariant. 
So we present the construction of \cite[\S~3]{AIP} over this locus in the following.


\subsection*{Notation and conventions}

Throughout this section to Section~\ref{sec:properness}, we will use the notations and conventions introduced below. 
Additional notation will be introduced as needed and, unless otherwise specified, will remain in force throughout these sections. 

Let $F$ be a totally real number field of degree $g$ over $\BQ$ with ring of integers $\CO_F$, $\CD_F$ be the different ideal of $F$.
Fix an integer $N\geq4$ with $p\nmid N$.
We denote $\fc$ to be a fractional ideal of $F$ and $\fc^+$ be the cone of totally positive elements.
Let $E$ be a finite extension of $\BQ_p$ inside $\overline{\BQ_p}$ which splits $F$, $I:=\Hom(F, \overline{\BQ_p})=\Hom(F, E)$ and $I_{\fp}=\Hom(F_{\fp}, \overline{\BQ_p})$ be the set of all field embeddings.
Then $I=\coprod_{\fp|p}I_{\fp}$. 
Let $\mathbb{T}:=\operatorname{Res}_{\mathcal{O}_F / \mathbb{Z}}(\mathbb{G}_m)$. 
We write $G:=\Res_{F/\BQ}(GL_2)$ and $G^*:=G\times_{\Res_{F/\BQ}\BG_m}\BG_m$, where $G\to \Res_{F/\BQ}\BG_m$ is the determinant morphism and $\BG_m\to \Res_{F/\BQ}\BG_m$ is the natural one. 

\subsection{Hilbert modular variety}\label{subsec:HMV}

For a scheme $S$ over $\CO_E$, a Hilbert-Blumenthal abelian scheme(HBAS) with $\fc$-polarization and $\mu_N$-level structure over $S$ is a quadruple $(A, \iota, \lambda, \psi)$, where 
\begin{enumerate}
    \item 
    $A \to S$ is an abelian scheme of relative dimension $g$ over $S$,
    \item 
    $\iota\colon\mathcal{O}_F \to\operatorname{End}_S(A)$ is a injective ring homomorphism, 
    \item 
    $\lambda$ is a $(\fc, \fc^+)$-polarization, i.e. an isomorphism of \'etale sheaves $\lambda\colon(P, P^+) \cong (\mathfrak{c}, \mathfrak{c}^+)$ as invertible $\mathcal{O}_{F}$-modules with a notion of positivity, where
    $P \subset \operatorname{Hom}_{\mathcal{O}_F}(A, A^\vee)$ is the sheaf for the \'etale topology on $S$ of symmetric $\mathcal{O}_F$-linear homomorphisms from $A$ to the dual abelian scheme $A^\vee$ and $P^+ \subset P$ is the subset of polarizations, such that the induced map $A\otimes_{\mathcal{O}_{F}}\mathfrak{c}\to A^{\vee}$ is an isomorphism of abelian schemes (the so called Deligne--Pappas condition).
    \item 
    $\psi\colon\CD_F^{-1}\otimes \mu_N\to A$ is an $\CO_F$-linear closed immersion of group schemes, which we call a $\mu_N$ level structure. 
\end{enumerate}

Let $M(\mu_N,\fc)$ be the Hilbert modular variety over $\CO_E$ which parametrizes the isomorphism classes of such quadruple $(A, \iota, \lambda, \psi)$.
We denote by $A^\un$ the universal HBAS over $M(\mu_N,\fc)$. 

Write $\overline{M}(\mu_{N},\mathfrak{c})$ and $\overline{M}^*(\mu_N, \fc)$ for a normal projective toroidal compactification of $M(\mu_{N},\mathfrak{c})$, respectively the minimal or Satake compactification of $M(\mu_N, \fc)$. 
Then there is a versal semiabelian scheme $\bar{A}$ with real multiplication by $\mathcal{O}_{F}$ over $\overline{M}(\mu_{N},\mathfrak{c})$.
Let $D=\overline{M}(\mu_{N},\mathfrak{c})\setminus M(\mu_{N},\mathfrak{c})$ be the boundary divisor in the toroidal compactification.

For any $\kappa=(k_\sigma)_{\sigma\in I}\in \BZ^I$, one can associate an invertible modular sheaf $ \Omega^{\kappa} $ on $ \overline{M}^{\mathrm{R}}(\mu_N, \mathfrak{c}) $ by setting
\[
\Omega^{\kappa} = \bigotimes_{\sigma \in I} \omega_{\bar{A}, \sigma}^{\otimes k_{\sigma}},
\]
where $\omega_{\bar{A}, \sigma}=\omega_{\bar{A}} \otimes_{\CO_E \otimes_{\mathbb{Z}} \mathcal{O}_{F}, 1\otimes\sigma} \CO_E $.

\subsection{Strict neighborhoods of the ordinary locus}

Let $\overline{\mathfrak{M}}(\mu_N, \mathfrak{c})$ (resp. $\overline{{\fM}}^*(\mu_N, \fc)$) be the formal completion of $\overline{M}(\mu_N, \mathfrak{c})$ along its special fiber, $\overline{\mathcal{M}}(\mu_N, \mathfrak{c})$ (resp. $\overline{\CM}^*(\mu_N, \fc)$) be its Raynaud generic fiber. 
Then there are natural morphisms
\[
\overline{\fM}(\mu_N, \fc)\to \overline{\fM}^*(\mu_N, \fc), \qquad 
\overline{\CM}(\mu_N, \fc)\to \overline{\CM}^*(\mu_N, \fc).
\]
Let $h_{\tot}$ be the (total) Hasse invariant. 
The idempotents $e_{\fp_1}, \cdots, e_{\fp_r}$ of $\CO_F\otimes \BZ_p$ associated to the prime ideals $\fp_1, \cdots, \fp_r$ of $\CO_F$ above $p$ induce a decomposition of $h_{\tot}$ into products $h_\tot=\prod_{i=1}^r h_{\fp_i}$. 

For $\underline{v}=(v_i)_{i=1}^r\in \BQ^r$ with $0\leq v_i<\frac{1}{p}$, we define the admissible open subset
\begin{align*}
\overline{\CM}(\mu_N,\fc)(\underline{v}):=\{x\in \overline{\CM}(\mu_N,\fc):|\widetilde{h}_{\fp_i}(x)|\geq p^{-v_i}, \forall 1\leq i\leq r\}\subset \overline{\CM}(\mu_N,\fc)
\end{align*}
where $\tilde{h}_{\fp_i}$ is a local lift of $h_{\fp_i}$ in an open neighborhood of $x$. 
Similarly, we can define an admissible open subset $\overline{\CM}^*(\mu_N,\fc)(\underline{v})\subset \overline{\CM}^*(\mu_N,\fc)$ and there is a natural morphism
\[
\overline{\CM}(\mu_N,\fc)(\underline{v})\to \overline{\CM}^*(\mu_N,\fc)(\underline{v}). 
\]
Similarly, using $h_\tot$, we can also define $\overline{\CM}(\mu_N, \fc)(v_\tot), \overline{\CM}^*(\mu_N, \fc)(v_\tot)$ and
\[
\overline{\CM}(\mu_N, \fc)(v_\tot)\to \overline{\CM}^*(\mu_N, \fc)(v_\tot)
\]
for $v\in[0, 1)\cap \BQ$.

We construct the formal model $\overline{\mathfrak{M}}(\mu_N,\mathfrak{c})(v_{\tot})$ of $\overline{\CM}(\mu_N,\fc)(v_{\tot})$ as follows. 
Let $\overline{\mathfrak{M}}(\mu_N,\mathfrak{c})$ be the $p$-adic formal completion of the toroidal compactification $\overline{M}(\mu_N,\mathfrak{c})$. 
Write $v=a/b$ with non-negative integers $a, b$ and $b\neq0$. 
Let $\fI$ be the ideal sheaf of $\overline{\mathfrak{M}}(\mu_N,\mathfrak{c})$ generated by $p^{a}$ and a lift $\widetilde{h_{\tot}}^{b}$ of $h_{\tot}^{b}$. 
Consider the blow-up of $\overline{\mathfrak{M}}(\mu_N,\mathfrak{c})$ along $\fI$, and then take the locus where the pullback of $\fI$ is generated by that of $\widetilde{h_{\tot}}^{b}$. 
Taking the normalization of this formal scheme in $\overline{\mathcal{M}}(\mu_N,\mathfrak{c})(v_{\tot})$, we get the desired formal model $\overline{\mathfrak{M}}(\mu_N,\mathfrak{c})(v_{\tot})$. 

We analogously define the formal model of $\overline{\CM}^*(\mu_N, \fc)(v_{\tot})$, denoted by $\overline{\mathfrak{M}}^*(\mu_N,\mathfrak{c})(v_{\tot})$, by performing the same construction on the minimal compactification $\overline{M}^*(\mu_N,\mathfrak{c})$. Then there is a natural morphism
\[
\overline{\mathfrak{M}}(\mu_N,\mathfrak{c})(v_{\tot}) \longrightarrow \overline{\mathfrak{M}}^*(\mu_N,\mathfrak{c})(v_{\tot}). 
\]

\subsection{Canonical subgroups and their trivialization}

Let $v\in [0, 1)\cap \BQ$. 
There is a semi-abelian formal scheme on $\overline{\fM}(\mu_N, \fc)(v_{\tot})$, which we still denote by $\bar{A}$. 
Fix an integer $n>0$ and assume that $0<v<1/p^n$. 
Note that $|\widetilde{h_{\fp_i}}|\geq|\widetilde{h_{\tot}}|$, so as in \cite[\S3.2.2]{AIP}, Fargues' theory in \cite{Fargues} provides a canonical subgroup
\[
H_n \hookrightarrow \bar{A}[p^n]
\]
over $\overline{\mathfrak{M}}(\mu_N,\mathfrak{c})(v_{\tot})$.
This subgroup is stable under the $\mathcal{O}_F$-action, and after base change to $\Spec(E)$, it is \'etale locally isomorphic to $\mathcal{O}_F/p^n\mathcal{O}_F$ as an $\mathcal{O}_F$-module.

Keep the above assumption, let
\[
\overline{\mathcal{M}}(\Gamma_1(p^n),\mu_N,\mathfrak{c})(v_{\tot})
:= \underline{\operatorname{Isom}}_{\overline{\mathcal{M}}(\mu_N,\mathfrak{c})(v_{\tot})}
\bigl( H_n,\; \mu_{p^n}\otimes_{\mathbb{Z}} \mathcal{D}_F^{-1} \bigr),
\]
where the isomorphisms are required to be $\mathcal{O}_F$-linear. 
It is a finite \'etale Galois covering of $\overline{\CM}(\mu_N, \fc)(v_{\tot})$ with Galois group $\mathbb{T}(\mathbb{Z}/p^n\mathbb{Z})$, where $\mathbb{T}:=\operatorname{Res}_{\mathcal{O}_F/\mathbb{Z}}\mathbb{G}_m$.
Similarly, we can define $\overline{\mathcal{M}}^*(\Gamma_1(p^n),\mu_N,\mathfrak{c})(v_{\tot})$, and there are various natural morphisms that fitting the following commutative diagram
\[
\begin{tikzcd}
{\overline{\mathcal{M}}(\Gamma_1(p^n),\mu_N,\mathfrak{c})(v_{\tot})} \arrow[r] \arrow[d] & {\overline{\mathcal{M}}^*(\Gamma_1(p^n),\mu_N,\mathfrak{c})(v_{\tot})} \arrow[d] \\
{\overline{\CM}(\mu_N,\fc)(v_{\tot})} \arrow[r]                                          & {\overline{\CM}^*(\mu_N,\fc)(v_{\tot})}                                         
\end{tikzcd}
\]

Taking the normalization of $\overline{\fM}(\mu_N, \fc)(v_{\tot})$ (resp. $\overline{\fM}^*(\mu_N, \fc)(v_{\tot})$) in $\overline{\CM}(\Gamma_1(p^n), \mu_N, \fc)(v_{\tot})$ (resp. $\overline{\CM}^*(\Gamma_1(p^n), \mu_N, \fc)(v_{\tot})$), we get a formal scheme $\overline{\fM}(\Gamma_1(p^n), \mu_N, \fc)(v_{\tot})$ (resp. $\overline{\fM}^*(\Gamma_1(p^n), \mu_N, \fc)(v_{\tot})$). 
There are various natural morphisms that fitting the following commutative diagram:
\[
\begin{tikzcd}
{\overline{\mathfrak{M}}(\Gamma_1(p^n),\mu_N,\mathfrak{c})(v_{\tot})} \arrow[r, "\rho_n"] \arrow[d, "h_n"'] & {\overline{\mathfrak{M}}^*(\Gamma_1(p^n),\mu_N,\mathfrak{c})(v_{\tot})} \arrow[d, "h_n'"] \\
{\overline{\fM}(\mu_N,\fc)(v_{\tot})} \arrow[r]                                                           & {\overline{\fM}^*(\mu_N,\fc)(v_{\tot})},                                                    
\end{tikzcd}
\]
which is a formal model of the previous diagram.
Put
\[
\rho'_{n}=h_n'\circ\rho_n \colon \overline{\mathfrak{M}}(\Gamma_1(p^n),\mu_N,\mathfrak{c})(v_\tot)
\longrightarrow
\overline{\mathfrak{M}}^*(\mu_N,\mathfrak{c})(v_\tot).
\]

\subsection{The sheaf $\mathcal{F}$ and the torsor $\mathfrak{IW}_w^+$}

Let $\omega_{\bar{A}} := e^*\Omega^1_{\bar{A}/\overline{\mathfrak{M}}(\mu_N,\mathfrak{c})(v_{\tot})}$. 
Then $\omega_{\bar{A}}$ is a locally free sheaf of $\CO_{\overline{\mathfrak{M}}(\mu_N,\mathfrak{c})(v_{\tot})}$-module of rank $g=[F:\mathbb{Q}]$, equipped with a natural action of $\mathcal{O}_F$.

Keep the notation the same as in the previous subsection. Choose $w\in(n-1, n-v\frac{p^n}{p-1}]\cap \BQ$. 
Let $H_n^D$ be the Cartier dual of the canonical subgroup.

\begin{proposition}
There exists a unique subsheaf $\mathcal{F}$ of $\omega_{\bar{A}}$ over $\overline{\mathfrak{M}}(\Gamma_1(p^n),\mu_N,\mathfrak{c})(v_{\tot})$ which is locally free of rank $1$ as $\mathcal{O}_F \otimes_{\mathbb{Z}} \mathcal{O}_{\overline{\mathfrak{M}}(\Gamma_1(p^n), \mu_N, \fc)(v_{\tot})}$-module and contains $p^{\frac{v}{p-1}}\omega_{\bar{A}}$. 
Moreover, for all $w$ as above, there are $\mathcal{O}_F$-linear maps:
\[
\mathrm{HT}_w \colon H_n^D \longrightarrow h_n^*(\mathcal{F})/p^w h_n^*(\mathcal{F}),
\]
which induce isomorphisms:
\[
\HT_w\otimes1\colon H_n^D\otimes_\BZ \mathcal{O}_{\overline{\mathfrak{M}}(\Gamma_1(p^n), \mu_N, \fc)(v_{\tot})}/(p^w)\xrightarrow{\sim} h_n^*(\CF)/p^wh_n^*(\CF). 
\]
\end{proposition}

\begin{proof}
This is the same as \cite[Proposition 3.4]{AIP}.
\end{proof}

We now define the formal torus that will act on trivializations of $\mathcal{F}$. 
Let $\widehat{\mathbb{T}}$ be the formal completion of $\mathbb{T}=\operatorname{Res}_{\mathcal{O}_F/\mathbb{Z}}\mathbb{G}_m$ along its unit section. 
For $w\in\BQ$, let $\mathbb{T}_w^0 \subset \widehat{\mathbb{T}}$ be the closed formal subgroup whose $R$-points (for a $p$-adically complete $\CO_E$-algebra $R$) are the elements of $(\mathcal{O}_F\otimes_{\mathbb{Z}} R)^\times$ that are congruent to $1$ modulo $p^w$.
Let $\mathbb{T}_w$ be the subgroup of $\widehat{\mathbb{T}}$ generated by $\mathbb{T}_w^0$ and $\mathbb{T}(\mathbb{Z}_p)$; equivalently, its $R$-points are the units whose reduction modulo $p^w$ lies in the image of $\mathbb{T}(\mathbb{Z}/p^n\mathbb{Z})$.

\begin{lemma}\label{lem:torsor}
There exists a formal scheme $\mathfrak{IW}_w^+$ together with a morphism
\[
\gamma_w \colon \mathfrak{IW}_w^+ \longrightarrow \overline{\mathfrak{M}}(\Gamma_1(p^n),\mu_N,\mathfrak{c})(v_{\tot})
\]
which is a formal torsor under the group $\mathbb{T}_w^0$, whose $R$-valued points over a morphism $\gamma\colon \operatorname{Spf}(R)\to \overline{\mathfrak{M}}(\Gamma_1(p^n),\mu_N,\mathfrak{c})(v_{\tot})$ classify isomorphisms
\[
\alpha \colon \gamma^*(\mathcal{F}) \cong \mathcal{O}_F\otimes_{\mathbb{Z}} R
\]
of $\mathcal{O}_F\otimes_{\mathbb{Z}} R$-modules such that the composite
\[
\mathcal{O}_F/p^n\mathcal{O}_F \cong H_n^D(R)
\xrightarrow{\mathrm{HT}_w}
\gamma^*(\mathcal{F})/p^w\gamma^*(\mathcal{F})
\xrightarrow{\alpha \bmod p^w}
\mathcal{O}_F\otimes_{\mathbb{Z}} R / p^w R
\]
sends $1$ to $1$.
Moreover, the action of $\mathbb{T}_w^0$ extends to an action of the larger group $\mathbb{T}_w$ on $\mathfrak{IW}_w^+$ over $\overline{\mathfrak{M}}(\mu_N,\mathfrak{c})(v_{\tot})$ via the composition
\[\pi_w\colon\mathfrak{IW}_w^+\xrightarrow{\gamma_w}\overline{\fM}(\Gamma_1(p^n), \mu_N, \fc) \xrightarrow{h_n}\overline{\fM}(\mu_N, \fc)(v_{\tot}). \]
\end{lemma}

\begin{proof}
This is the same as \cite[Lemma 3.5]{AIP} and subsequent discussions after that. 
\end{proof}

\subsection{Modular sheaves for families}

Let $\CW^*$ be the weight space for $G^*$, i.e. the rigid analytic space over $E$ associated to the completed group algebra $\CO_E[[\BT(\BZ_p)]]$, with universal character 
\[\kappa^\un\colon\BT(\BZ_p)\to \CO_E[[\BT(\BZ_p)]]^\times. \]
Each element of $\BZ[I]$ naturally gives rise to an element of $\CW^*(E)$, which is called an \emph{algebraic weight for $G$}. 
More precisely, any $\tau\in I$ induces a homomorphism
\[
\tau\colon  \CO_F\otimes_\BZ\BZ_p\cong\prod_{\fq|p}\CO_{F_\fq} \to \CO_E, 
\]
which factors through the natural projection $\prod_{\fq|p}\CO_{F_\fq}\to \CO_{F_{\fp}}$ when $\tau\in I_{\fp}$. 
So for $\nu=(\nu_\tau)_{\tau\in\BZ}\in \BZ^I$, we have the associated character 
\[
H=\BT(\BZ_p)\to \CO_E^\times, \quad a\mapsto \prod_{\tau\in I}\tau(a)^{\nu_\tau},  
\]
which defines an $E$-point of $\CW$. 

Let $\mathcal{U}$ be a reduced affinoid rigid space over $E$, equipped with a morphism $\mathcal{U}\to \mathcal{W}$. 
Let $\kappa^{\mathcal{U}}\colon \mathbb{T}(\mathbb{Z}_p)\to A^\times$ be the universal character restricted to $\mathcal{U}$, where $A$ is the algebra of power-bounded functions on $\mathcal{U}$. 
We will always identify the morphism $\CU\to\CW^*$ with the associated character $\kappa^\CU$, which is called a (family of) weight. 
Choose an integer $n$ large enough so that $\kappa^{\mathcal{U}}$ is $n$-analytic, i.e. the restriction of $\kappa^\CU$ to $\BT_n^0(\BZ_p)$ factors as the composite $\exp\circ \psi\circ \log_F$, where $\log_F$ is the $p$-adic logarithm
\[\log_F\colon1+p^n\CO_F\otimes_\BZ\BZ_p\to p^n\CO_F\otimes_\BZ\BZ_p, \]
$\psi$ is some $\BZ_p$-linear map:
\[\psi\colon p^n\CO_F\otimes_\BZ\BZ_p\to pA\]
and $\exp$ is the $p$-adic exponential $pA\to 1+pA$.
Let $\mathfrak{U}:=\operatorname{Spf} A$ be a formal model of $\mathcal{U}$ and $\chi^{\mathcal{U}}$ denote the restriction of $\kappa^{\mathcal{U}}$ to $\mathbb{T}_w^0$.

With the above choice of $n$, we keep the notation the same as in the previous subsection. 
Using the torsor $\gamma_w$, we define the sheaf of $\CO_{\overline{\fM}(\Gamma_1(p^n), \mu_N, \fc)(v_{\tot})}$-modules
\[
\Omega^{\chi^{\mathcal{U}}}_w
:= (\gamma_w\times 1_{\fU})_* \mathcal{O}_{\mathfrak{IW}_w^+\times \mathfrak{U}}[-\chi^{\mathcal{U}}],
\]
where $[-\chi^{\mathcal{U}}]$ denotes the subsheaf on which $\mathbb{T}_w^0$ acts through the character $-\chi^{\mathcal{U}}$.
Furthermore, we define the sheaf of $\CO_{\overline{\fM}(\mu_N, \fc)(v_{\tot})}$-modules on $\overline{\fM}(\mu_N, \fc)(v_{\tot})$ by taking $\mathbb{T}(\mathbb{Z}/p^n\mathbb{Z})$-invariants after twisting by the character $\kappa^{\mathcal{U}}$:
\[
\Omega^{\kappa^{\mathcal{U}}}_w
:= \left( (h_n\times 1_{\fU})_* \Omega^{\chi^{\mathcal{U}}}_w(\kappa^{\mathcal{U}}) \right)^{\mathbb{T}(\mathbb{Z}/p^n\mathbb{Z})}.
\]
Equivalently, $\Omega^{\kappa^{\mathcal{U}}}_w$ is the subsheaf of $(h_n\times 1)_*\Omega^{\chi^{\mathcal{U}}}$ on which $\mathbb{T}_w$ acts through the character $-\kappa^{\mathcal{U}}$.

\begin{proposition}
$\Omega^{\chi^{\mathcal{U}}}_w$ is an invertible $\mathcal{O}_{\overline{\mathfrak{M}}(\Gamma_1(p^n),\mu_N,\mathfrak{c})(v_{\tot})\times \mathfrak{U}}$-module, independent of the choice of $w$. 
\end{proposition}

\begin{proof}
This is the same as the proof of \cite[Proposition 3.8]{AIP}. 
\end{proof}

\begin{proposition}
$\Omega^{\kappa^{\mathcal{U}}}_w$ is a coherent $\mathcal{O}_{\overline{\mathfrak{M}}(\mu_N,\mathfrak{c})(v_{\tot})\times \mathfrak{U}}$-module, independent of the choice of $w$. 
The restriction of $\Omega^{\kappa^{\mathcal{U}}}_w$ to the rigid generic fiber is an invertible $\CO_{\overline{\CM}(\mu_N, \fc)(v_{\tot})\times \CU}$-module, independent of the choice of $n$. 
\end{proposition}

\begin{proof}
This is the same as \cite[Proposition 3.13]{AIP}. 
\end{proof}

We will simply write $\Omega^{\chi^{\mathcal{U}}}_w$ and $\Omega^{\kappa^{\mathcal{U}}}_w$ as $\Omega^{\chi^{\mathcal{U}}}$ and $\Omega^{\kappa^{\mathcal{U}}}$ respectively.

\begin{proposition}\label{prop:classical sheaves}
When $\CU=\Sp(E)$ and $\CU\to \CW^*$ is an algebraic weight defined by some $\kappa\in \BZ[I]$, the sheaf $\Omega^{\kappa^\CU}$ coincides with the classical modular sheaf of weight $\kappa$ defined in Section~\ref{subsec:HMV}.  
\end{proposition}

\begin{proof}
This is the same as \cite[Corollary~3.10]{AIP}. 
\end{proof}

Let $D_n$ be the boundary divisor in $\overline{\fM}(\Gamma_1(p^n), \mu_N, \fc)(v_{\tot})$, then $\CO(-D_n)=h_n^*\CO(-D)$. 
We define $\Omega^{\chi^{\mathcal{U}}}(-D_n):=\Omega^{\chi^{\mathcal{U}}}\otimes_{\CO}\CO(-D_n)$ on $\overline{\fM}(\Gamma_1(p^n), \mu_N, \fc)(v_{\tot})\times \fU$. 

\subsection{Overconvergent modular forms for $G^*$ and compatibility with base change}

Keep the notation the same as in the previous subsection. 
Recall that $D$ is the boundary divisor in $\overline{M}(\mu_N, \fc)$. 
Let 
\begin{align*}
M(\mu_N, \fc, \kappa^\CU)(v_{\tot})
&:=H^0(\overline{\CM}(\mu_N, \fc)(v_{\tot})\times \CU, \Omega^{\kappa^\CU}), \\
S(\mu_N, \fc, \kappa^\CU)(v_{\tot})
&:=H^0(\overline{\CM}(\mu_N, \fc)(v_{\tot})\times \CU, \Omega^{\kappa^\CU}(-D)).
\end{align*}
They are Banach $\CO(\CU)$-modules, called the space of (cuspidal) $v_{\tot}$-overconvergent modular forms of weight $\CU$, tame level $N$ for $G^*$. 
Put 
\[
S(\mu_N, \fc, \kappa^\CU)^{\dagger}:=\varinjlim_{v>0}S(\mu_N, \fc, \kappa^\CU)(v_{\tot}). 
\]
Note that
\begin{align*}
\Omega^{\kappa^U}(-D)&=\left(h_{n,*}\Omega^{\chi^U}(\kappa^U)\right)^{\BT(\BZ/p^n\BZ)}\otimes\CO(-D)\\
&\cong\left(h_{n,*}\bigl(\Omega^{\chi^U}\otimes h_n^\ast\CO(-D)\bigr)(\kappa^U)\right)^{\BT(\BZ/p^n\BZ)} \\
&=\left(h_{n,*}\Omega^{\chi^U}(-D_n)(\kappa^U)\right)^{\BT(\BZ/p^n\BZ)}.
\end{align*}
Hence
\begin{equation}\label{eq:cusp forms descent}
S(\mu_N, \fc, \kappa^\CU)(v_{\tot})
\cong \left(H^0(\overline{\CM}(\Gamma_1(p^n), \mu_N, \fc)(v_{\tot})\times \CU, \Omega^{\chi^\CU}(-D_n))(\kappa^U)\right)^{\mathbb{T}(\mathbb{Z}/p^n\mathbb{Z})}.
\end{equation}

Let $\mathcal{U}'$ be a reduced affinoid rigid space over $E$ with a map $f\colon \mathcal{U}'\to \mathcal{U}$, and $\kappa^{\mathcal{U}'}$ be the pullback of universal character.

\begin{proposition}\label{prop:base change of modular sheaves}
There are canonical isomorphisms of coherent sheaves
\[(1\times f)^*\Omega^{\chi^\CU}\xrightarrow{\sim} \Omega^{\chi^{\CU'}}\]
on $\overline{\fM}(\Gamma_1(p^n), \mu_N, \fc)(v_\tot)\times \fU'$ and
\[
(1\times f)^*\Omega^{\kappa^{\mathcal{U}}}\xrightarrow{\sim} \Omega^{\kappa^{\mathcal{U}'}}
\]
on $\overline{\mathcal{M}}(\mu_N,\mathfrak{c})(v_{\tot}) \times \mathcal{U}'$. 
In particular, we have a specialization map
\[
S(\mu_N, \fc, \kappa^{\CU})(v_\tot)\to S(\mu_N, \fc, \kappa^{\CU'})(v_\tot), 
\]
which induces a natural morphism 
\[
\alpha_{\CU, \CU'}\colon S(\mu_N, \fc, \kappa^{\CU})(v_\tot)\widehat{\otimes}_{\CO(\CU)}\CO(\CU')\rightarrow S(\mu_N, \fc, \kappa^{\CU'})(v_\tot). 
\]
\end{proposition}

\begin{proof}
This is the same as \cite[Proposition 3.14]{AIP} and its proof. 
\end{proof}

\begin{corollary}\label{cor:flat bc of oc MFs of G*}
Assume
$f\colon \CU'\to \CU$ is an affinoid subdomain. 
Then the morphism $\alpha_{\CU, \CU'}$ is an isomorphism. 
\end{corollary}

\begin{proof}
Let 
\[
X:=\overline{\CM}(\mu_N, \fc)(v_{\tot}), \quad Y:=\overline{\CM}^*(\mu_N, \fc)(v_{\tot})
\]
and $q: X\to Y$ be the natural morphism. 
Consider the following Cartesian diagram
\[
\begin{tikzcd}
X\times\CU' \arrow[r, "1\times f"] \arrow[d, "q_{\CU'}"] & X\times\CU \arrow[d, "q_{\CU}"] \\
Y\times \CU' \arrow[r, "1\times f"]                      & Y\times \CU.                     
\end{tikzcd}
\]
Note that $q$ is proper, so is $q_{\CU}$. 
Hence by Kiehl's proper mapping theorem, $(q_{\CU})_*\Omega^{\kappa^\CU}$ is a coherent sheaf, and we have 
\begin{align*}
H^0(Y\times\CU, (q_{\CU'})_*\Omega^{\kappa^\CU})\widehat{\otimes}_{\CO(Y\times\CU)}\CO(Y\times\CU')
&\xrightarrow{\sim}H^0(Y\times\CU', (q_{\CU'})_*(\Omega^{\kappa^{\CU}}|_{X\times\CU'})) \\
&\cong H^0(Y\times\CU', (q_{\CU'})_*\Omega^{\kappa^{\CU'}})
\end{align*}
by Proposition~\ref{prop:base change of modular sheaves}. 
Since the (total) Hasse invariant is a section of an ample line bundle on $\overline{\CM}^*(\mu_N, \fc)$, $Y$ is affinoid. 
So we get the isomorphism $\alpha_{\CU, \CU'}$. 
\end{proof}

\begin{proposition}\label{prop:constant family for Gstar}
Assume $\CU=\Sp(E')$ for some finite extension of $E$, then the morphism $\alpha_{\CU, \CU'}$ is an isomorphism. 
\end{proposition}

\begin{proof}
We use the arguments of \cite[Lemma~3.11]{Hattori} and \cite[Lemma~3.12]{Hattori}. 
Let $R'=\CO(\CU')$. Take an admissible flat formal model $\fU'=\Spf(R_0')$ of $\CU'$ over $\fU:=\Spf(\CO_{E'})$.
Let 
\[
\fX_n=
\overline{\mathfrak{M}}(\Gamma_1(p^n),\mu_N,\mathfrak{c})(v_\tot), \quad \fY:=\overline{\mathfrak{M}}^*(\mu_N,\mathfrak{c})(v_\tot).
\]
Consider the Cartesian diagram 
\[
\begin{tikzcd}
\fX_n\times \fU' \arrow[d, "\rho_n'\times1"] \arrow[r, "1\times f"] & \fX_n\times \fU \arrow[d, "\rho_n'\times1"] \\
\fY\times \fU' \arrow[r, "1\times f"]                              & \fY\times\fU                              
\end{tikzcd}
\]
Since $\rho_n'$ is proper, $\Omega^{\chi^\CU}(-D_n)$ is coherent, by the flatness of $1\times f$ and \cite[Proposition~A.1]{Hall}, 
the natural base change map 
\[
(1\times f)^*(\rho_n'\times1)_*\Omega^{\chi^\CU}(-D_n) \to (\rho_n'\times1)_*(1\times f)^*\Omega^{\chi^\CU}(-D_n)
\]
is an isomorphism of coherent sheaves on $\fY\times\fU'$. 
Taking the Raynaud generic fiber and using Proposition~\ref{prop:base change of modular sheaves}, we have
\[
(1\times f)^*(\rho_n'\times1)_*\Omega^{\chi^\CU}(-D_n)\cong (\rho_n'\times1)_*\Omega^{\chi^{\CU'}}(-D_n)
\]
as coherent sheaves on $\overline{\mathcal{M}}^*(\mu_N,\mathfrak{c})(v_{\tot}) \times \mathcal{U}'$, which is affinoid. 
So taking global sections yields the isomorphism 
\begin{align*}
&H^0\left(\overline{\CM}(\Gamma_1(p^n), \mu_N, \fc)(v_{\tot})\times \CU',  \Omega^{\chi^{\CU'}}(-D_n)\right)\\
\cong&
H^0\left(\overline{\mathcal{M}}^*(\mu_N,\mathfrak{c})(v_{\tot}) \times \mathcal{U}', (1\times f)^*(\rho_n'\times1)_*\Omega^{\chi^\CU}(-D_n)\right) \\
\cong&
H^0\left(\overline{\CM}(\Gamma_1(p^n), \mu_N, \fc)(v_{\tot})\times \CU, \Omega^{\chi^{\CU}}(-D_n)\right)\widehat{\otimes}_{\CO(\CU)}\CO(\CU'). 
\end{align*}
Taking the $\BT(\BZ/p^n\BZ)$-equivariant part and using \eqref{eq:cusp forms descent}, one can check the resulting isomorphism is $\alpha_{\CU, \CU'}$. 
\end{proof}

\begin{proposition}\label{prop:AIP's higher vanishing}
Consider the natural morphism
\[
\rho'_n\times 1 \colon \overline{\mathfrak{M}}(\Gamma_1(p^n),\mu_N,\mathfrak{c})(v_\tot)\times \mathfrak{U}\;\longrightarrow\;
\overline{\mathfrak{M}}^*(\mu_N,\mathfrak{c})(v_\tot)\times \mathfrak{U}.
\]
Then for every $q>0$, we have the vanishing of the higher direct image sheaves:
\[
R^q(\rho'_n\times 1)_*\bigl(\Omega^{\chi^{\mathcal{U}}}(-D_n)\bigr) = 0. 
\]
\end{proposition}

\begin{proof}
This is the same as \cite[Corollary~3.20]{AIP}, which follows from the analogy of \cite[Lemma~3.18]{AIP} and \cite[Lemma~3.19]{AIP}. 
More precisely, we claim that the analogy of \cite[Theorem~3.17]{AIP} for $\rho_n$ still holds, i.e. $R^q(\rho_{n})_*\CO_{\overline{\mathfrak{M}}(\Gamma_1(p^n),\mu_N,\mathfrak{c})(v_\tot)}(-D_n)=0$. 
In fact, since $\rho_n$ is an isomorphism away from the boundary, the theorem on formal functions reduces the desired vanishing to the formal fibers over the cusps.
These are the same toric formal schemes as those appearing in the proof of \cite[Theorem~3.17]{AIP}.
So the required vanishing therefore follows from \cite[Proposition~6.4]{AIP}.
Note that $h_n'$ is finite, thus we further have $R^q(\rho_{n}')_*\CO_{\overline{\mathfrak{M}}(\Gamma_1(p^n),\mu_N,\mathfrak{c})(v_\tot)}(-D_n)=0$, which is an analogy of \cite[Lemma~3.18]{AIP}. 
It's clear that the analogy of \cite[Lemma~3.19]{AIP} for $\Omega^{\chi^\CU}/p$ on $\overline{\fM}(\Gamma_1(p^n), \mu_N, \fc)(v_\tot)\times \fU$ also holds, i.e. $\Omega^{\chi^\CU}/p\simeq \CO_{\overline{\fM}(\Gamma_1(p^n), \mu_N, \fc)(v_\tot)\times \fU}/p$, by the same proof there. 
The desired conclusion follows from the same argument as in \cite[Corollary~3.20]{AIP}. 
\end{proof}

\begin{proposition}\label{prop:reg closed im bc of oc MFs of G*}
Assume $f\colon \mathcal{U}'\to \mathcal{U}$ is a regular closed immersion, i.e. a closed immersion so that locally, the ideal $I$ of $\CO(\CU)$ defining $\CU'$ is generated by a regular sequence. 
Then $\alpha_{\CU, \CU'}$ is an isomorphism. 
\end{proposition}

\begin{proof}
We repeat the argument of \cite[Lemma 3.13]{Hattori}. 
Write $R:=\CO(\CU)$. 
By our assumption, the $R$-module $I$ is perfect, so is $R/I$, and there is a finite projective resolution
\[
0\to P_m\xrightarrow{d_m} \cdots \xrightarrow{d_2} P_1\xrightarrow{d_1}R\to R/I \to 0
\]
of $R/I$ as a $R$-module. Let $\widetilde{P_m}$ be the locally free coherent sheaf on $\CU$ associated to $P_m$, we get a resolution of $f_*\CO_{\CU'}$ by coherent sheaves
\[
0\to \widetilde{P_m}\xrightarrow{d_m} \cdots \xrightarrow{d_2} \widetilde{P_1}\xrightarrow{d_1}\CO_\CU\to f_*\CO_{\CU'}\to 0
\]
on $\CU$. 
Let 
\[X_n:=\bar{\mathcal{M}}\bigl(\Gamma_1(p^n),\mu_N,\mathfrak{c}\bigr)(v_{\mathrm{tot}}), \quad Y:=\overline{\CM}^*(\mu_N, \fc)(v_{\tot}).
\]
Pullback along the natural projection $p_\CU: X_n\times\CU\to  \CU$, we get a resolution
\[
0\to p_\CU^*\widetilde{P_m}\xrightarrow{d_m} \cdots \xrightarrow{d_2} p_\CU^*\widetilde{P_1}\xrightarrow{d_1}\CO_{X_n\times\CU}\to (1\times f)_*\CO_{X_n\times \CU'}\to 0. 
\]
Let $\CF_{\CU}:=\Omega^{\chi^{\CU}}(-D_n), \CF_{\CU'}:=\Omega^{\chi^{\CU'}}(-D_n)$.
By Proposition~\ref{prop:base change of modular sheaves}, 
\[
\CF_\CU\otimes_{\CO}(1\times f)_*\CO_{X_n\otimes\CU'}\cong (1\times f)_*(1\times f)^*\CF_\CU\cong (1\times f)_*\CF_{\CU'}. 
\]
So tensoring the resolution with the locally free sheaf $\CF_\CU$, we get a resolution
\[
0\to \CF_m\xrightarrow{d_m}\cdots\xrightarrow{d_2}\CF_1\xrightarrow{d_1} \CF_U\to (1\times f)_*\CF_{\CU'}\to 0, 
\]
on $X_n\times\CU$, where $\CF_j=\CF_\CU\otimes_{\CO}p_\CU^*\widetilde{P_j}$.
Note that $\CF_j$ are direct summand of $\CF_\CU^{\oplus N}$ for some $N$, by Proposition~\ref{prop:AIP's higher vanishing}, for any $q>0, R^q(\rho_n'\times1)\CF_i=0$. 
A standard induction argument breaking the long exact sequence into short ones tells us that, after pushing forward along the proper morphism 
\[
\rho_n'\times1: X_n\times\CU\to Y\times\CU
\]
the resulting sequence of coherent sheaves on $Y\times\CU$ is still exact. 
Now $Y\times\CU$ is affinoid, so taking global sections gives us a long exact sequence
\[
0\to H^0(X_n\times\CU, \CF_m)\xrightarrow{d_m}\cdots\xrightarrow{d_2}H^0(X_n\times\CU, \CF_1)\xrightarrow{d_1} H^0(X_n\times\CU, \CF_U)\xrightarrow{\alpha} H^0(X_n\times\CU', \CF_{\CU'})\to 0.  
\]
On the other hand, we have an exact sequence
\[
H^0(X_n\times\CU, \CF_\CU)\otimes_{R}P_1\xrightarrow{d_1} H^0(X_n\times\CU, \CF_\CU)\otimes_RR\to H^0(X_n\times\CU, \CF_\CU)\otimes_{R}R/I\to 0. 
\]
Note that for any finite projective $R$-module $P$, the natural morphism 
\[
P=H^0(U, \widetilde{P})\to H^0(X_n\times\CU, p_\CU^*\widetilde{P})
\]
induces a morphism
\[
\beta_{P}: H^0(X_n\times\CU, \CF_\CU)\otimes_RP\to H^0(X_n\times\CU, \CF_{\CU}\otimes p_\CU^*\widetilde{P}). 
\]
The construction of $\beta_P$ is functorial in $P$, so by writing $P$ as a direct summand of $R^{\oplus N}$ for some $N$, we see that $\beta_P$ is an isomorphism. 
Therefore, we obtain a commutative diagram
\[
\begin{tikzcd}
{H^0(X_n\times\CU, \CF_\CU)\otimes_RP} \arrow[r, "d_1"] \arrow[d, "\beta_{P_1}"] & {H^0(X_n\times\CU, \CF_\CU)\otimes_RR} \arrow[r] \arrow[d, "\beta_R"] & {H^0(X_n\times\CU, \CF_\CU)\otimes_RR/I} \arrow[r] \arrow[d, "\cong", dashed] & 0 \\
{H^0(X_n\times\CU, \CF_1)} \arrow[r, "d_1"]                                      & {H^0(X_n\times\CU, \CF_\CU)} \arrow[r, "\alpha"]                      & {H^0(X_n\times\CU', \CF_\CU')} \arrow[r]                                       & 0
\end{tikzcd}
\]
Since $\alpha$ is continuous, $\ker\alpha=IH^0(X_n\times\CU, \CF_\CU)$ is closed, thus $H^0(X_n\times\CU, \CF_\CU)\otimes_RR/I\cong H^0(X_n\times\CU, \CF_\CU)\widehat{\otimes}_RR/I$. 
Finally, taking the $\BT(\BZ/p^n)$-equivariant part and using \eqref{eq:cusp forms descent}, it's easy to check the last dotted arrow in the above diagram induces our desired isomorphism $\alpha_{\CU, \CU'}$. 
\end{proof}

\subsection{Hecke operators}\label{sec:Hecke for Gstar}

By \cite[\$3.7, \$4.3]{AIP}, we have the following Hecke operators. 
\footnote{
In \cite[\$3.7, \$4.3]{AIP}, Andreatta--Iovita--Pilloni constructed Hecke operators for their module $M(\mu_N, \fc, \kappa^\CU)(\underline{v})$, which is defined by taking global sections over $\overline{\CM}(\mu_N, \fc)(\underline{v})\times\CU$ instead of $\overline{\CM}(\mu_N, \fc)(v_{tot})\times\CU$. 
It is easy to check their constructions are also valid for $M(\mu_N, \fc, \kappa^\CU)(v_{\tot})$. 
Also, see \cite[\$3.3.3]{Hattori}. 
}
\begin{itemize}
    \item 
    For any ideal $\ell\subset \CO_F$ prime to $Np$, there are continuous homomorphisms of Banach $\CO(\CU)$-modules
    \begin{align*}
    T_{\ell, \CU}: S(\mu_N, \ell\fc, \kappa^\CU)(v_\tot)\to S(\mu_N, \fc, \kappa^\CU)(v_\tot), \\
    S_{\ell, \CU}: S(\mu_N, \ell^2\fc, \kappa^\CU)(v_\tot)\to S(\mu_N, \fc, \kappa^\CU)(v_\tot). 
    \end{align*}
    \item 
    For any prime ideal $\fp|p$ of $\CO_F$, choose $x_\fp\in F^{\times, >0}$ such that $v_{\fp}(x_\fp)=1$ and $v_{\fq}(x_\fp)=0$ for all $\fq\neq\fp$ above $p$. 
    Then for any fractional ideal $\fc$ which is prime to $p$, there is a continuous homomorphism of Banach $\CO(\CU)$-modules
    \[
    U_{\fp, \CU}: S(\mu_N, \fp x_{\fp}^{-1}\fc, \kappa^\CU)(v_\tot)\to S(\mu_N, \fc, \kappa^\CU)(v_\tot). 
    \]
    Moreover, there is a compact operator
    \[U_{p, \CU}: S(\mu_N, \fc, \kappa^\CU)(v_\tot)\to S(\mu_N, \fc, \kappa^\CU)(v_\tot), \]
    which can be factored into a composition 
    \begin{equation}\label{eq:factor Up for Gstar}
    S(\mu_N, \fc, \kappa^\CU)(v_\tot)\to S(\mu_N, \fc, \kappa^\CU)(p^{-1}v_\tot)\to S(\mu_N, \fc, \kappa^\CU)(v_\tot),
    \end{equation}
    where the first arrow is the natural restriction map, and we still denote the second arrow by $U_{p, \CU}$. 
\end{itemize}
By \cite[Remark~3.28]{AIP}, the Hecke operators above are functorial in $\CU$.
In particular, for $\CU'\to \CU\to \CW^*$ as in Proposition~\ref{prop:constant family for Gstar} or Proposition~\ref{prop:reg closed im bc of oc MFs of G*}, we have $U_{p, \CU}\widehat{\otimes}1=U_{p, \CU'}$ under the isomorphism $\alpha_{\CU, \CU'}$. 
\section{Overconvergent modular forms for \texorpdfstring{$G$}{G} and the cuspidal Hilbert eigenvariety}\label{sec:eigenvariety}

In this section, we follow \cite[\S 4]{AIP} to review $p$-adic families of cuspidal overconvergent modular forms for $G$ and the cuspidal Hilbert eigenvariety. 
The main result is Theorem~\ref{thm:spectral}, which describes the spectral variety over the closed weight loci $\CW_\kappa$. 
It provides the geometric input for our unconditional properness theorem.

\subsection{Overconvergent modular forms for $G$}\label{subsec:oc MFs for G}

Let $\CW$ be the weight space for $G$, i.e. the rigid analytic space associated to $\CO_E[[\BT(\BZ_p)\times\BZ_p^\times]]$, with universal character
\[
(\nu^\un, w^\un)\colon \BT(\BZ_p)\times\BZ_p^\times\to \CO_E[[\BT(\BZ_p)\times\BZ_p^\times]]^\times\subset \CO(\CW)^\times. 
\]
The group homomorphism 
\[\BT(\BZ_p)\to \BT(\BZ_p)\times \BZ_p^\times, \quad x\mapsto (x^2, \Norm_{F/\BQ}(x))\]
induces a morphism
\[k\colon \CW\to \CW^*\]. 
To define $p$-adic families of modular forms for the group $G$, let $\CU$ be an affinoid rigid space over $E$ with a morphism $\CU\to \CW$. 
Composing this morphism with $k$, we obtain a morphism $\CU\to \CW^*$. 
From these two morphisms, we get two continuous characters 
\[
(\nu^\CU, w^\CU)\colon \BT(\BZ_p)\times \BZ_p^\times\to A^\times, \quad
\kappa^\CU\colon \BT(\BZ_p)\to A^\times, 
\]
by restricting the universal one, where $A$ is the algebra of power-bounded functions on $\CU$. 
Then 
\[
\kappa^{\CU}=2\nu^\CU+w^\CU\circ \Norm_{F/\BQ}. 
\]
We will always identify the morphism $\CU\to\CW$ with the associated character $(\nu^\CU, w^\CU)$, which is called \emph{a (family of) weight}. 
Take $n\geq1$ so that $\kappa^\CU$ is $n$-analytic. 
With the same notation as in Section~\ref{sec:oc MFs for G*}, we get the spaces $M(\mu_N, \fc, \kappa^\CU)(v_{\tot})$ and $S(\mu_N, \fc, \kappa^\CU)(v_{\tot})$. 

Let $\CO_F^{\times, +}$ be the subgroup of totally positive units in $\CO_F^\times$, $U_N$ be its subgroup consisting of elements that are congruent to $1$ modulo $N$. 
For the above weight $(\nu^\CU, w^\CU)$ of $G$, there is a natural action of the finite group $\Delta:=\CO_F^{\times, +}/U_N^2$ on $M(\mu_N, \fc, \kappa^\CU)(v_{\tot})$ defined in \cite[\S 4.1]{AIP}. 
More precisely, any $f\in M(\mu_N, \fc, \kappa^\CU)(v_{\tot})$ can be seen as a rule which associates to any tuple $(A_{/R}, \iota, \lambda, \psi, u, \alpha)$ an  element $f(A_{/R}, \iota, \lambda, \psi, u, \alpha)\in R[1/p]$, such that the association depends only on the isomorphism class of the tuple, commutes with base change and satisfies 
\[
f(A_{/R}, \iota, \lambda, \psi, t^{-1}u, t^{-1}\alpha)=\kappa^\CU(t)f(A_{/R}, \iota, \lambda, \psi, u, \alpha), \forall t\in \BT(\BZ_p), 
\]
where 
\begin{itemize}
    \item $R$ denotes a normal, $p$-adically complete and separated, flat $\CO_E$-algebra, which is topologically of finite type; 
    \item $A\to\Spf(R)$ is an abelian scheme of relative dimension $g$ such that $|\widetilde{h_{\tot}}(A_x)|\geq p^{-v}$ for any $x\in \Sp(R[1/p])$; 
    \item $\iota$ is an $\CO_E$-multiplication on $A$; 
    \item $\psi$ is a $\mu_N$-level structure; 
    \item $\lambda$ is a $(\fc, \fc^+)$-polarization; 
    \item $u: H_{n, E}\cong \mu_{n, E}\otimes \CD_F^{-1}$ is an isomorphism of group schemes over $R[1/p]$; 
    \item $\alpha: \gamma^*\CF\cong \CO_F\otimes R$ is an isomorphism of $\CO_F\otimes_\BZ R$-modules such that the composition of $\HT_w$ and $\alpha(\mod p^w)$ (see Lemma~\ref{lem:torsor}) sends $1\in \CO_F/p^n\CO_F\cong H_n^D(R)$ to $1\in \CO_F\otimes R/p^wR$ for all $0<w\leq vp^n/(p-1)$, where $\gamma: \Spf(R)\to \overline{\fM}(\Gamma_1(p^n), \mu_N, \fc)(v_{\tot})$ is the morphism induced by $(A_{/R}, \iota, \psi, \lambda, u)$. 
\end{itemize}
Then there is an action of $\Delta$ on $M(\mu_N, \fc, \kappa^\CU)(v_{\tot})$, defined by 
\[
(\epsilon\cdot f)(A_{/R}, \iota, \lambda, \psi, u, \alpha):=\nu^\CU(\epsilon)f(A_{/R}, \iota, \epsilon^{-1}\lambda, \psi, u, \alpha). 
\]
Consequently, we have a projector $e^\CU$ on $M(\mu_N, \fc, \kappa^\CU)(v_{\tot})$, given by 
\[
(\fe^\CU\cdot f)(A_{/R}, \iota, \lambda, \psi, u, \alpha):=\frac{1}{|\Delta|}\sum_{\epsilon\in\Delta}\nu^\CU(\epsilon)(A_{/R}, \iota, \epsilon^{-1}\lambda, \psi, u, \alpha) 
\]
and we define 
\[
M^G(\mu_N, \fc, (\nu^\CU, w^\CU))(v_{\tot}):=\Im(\fe^\CU)=\left(M(\mu_N, \fc, \kappa^\CU)(v_{\tot})\right)^\Delta.  
\] 
Moreover, the action of $\Delta$ is compatible with the restriction to a smaller radius $v'<v$ of overconvergence, functorial in $\CU$ and preserve the submodule $S(\mu_N, \fc, \kappa^\CU)(v_{\tot})$. 
So we can define
\begin{align*}
S^G(\mu_N, \fc, (\nu^\CU, w^\CU))(v_{\tot})&:=\left(S(\mu_N, \fc, \kappa^\CU)(v_{\tot})\right)^\Delta, \\
S^G(\mu_N, \fc, (\nu^\CU, w^\CU)))^{\dagger}&:=\varinjlim_{v>0}S^G(\mu_N, \fc, (\nu^\CU, w^\CU)))(v_\tot). 
\end{align*}
Let $\Frac^{(p)}$ be the group of fractional ideals prime to $p$, $\mathrm{Princ}(F)^{+, (p)}$ be the group of positive elements which are $p$-adic units. 
Then the quotient $\Frac^{(p)}/\mathrm{Princ}(F)^{+, (p)}$ is the strict class group $\Cl^+(F)$ of $F$. 
We define the Banach $\CO(\CU)$-module of $v_\tot$-overconvergent modular forms for $G$ of weight $\CU$ and tame level $N$ to be 
\[
M^G(\mu_N, (\nu^\CU, w^\CU))(v_\tot):=\left(\bigoplus_{\fc\in \Frac(F)^{(p)}} M^G(\mu_N, \fc, \kappa^\CU)(v_\tot)\right)/\left(L_x(f)-f\right)_{x\in \mathrm{Princ}(F)^{+, (p)}}, 
\]
where 
\begin{align*}
L_x: M^G(\mu_N, \fc, \kappa^\CU)(v_\tot)&\xrightarrow{\sim} M^G(\mu_N, x\fc, \kappa^\CU)(v_\tot), \\
f&\mapsto [(A_{/R}, \iota, \lambda, \psi, u, \alpha)\mapsto \nu^\CU(x)f(A_{/R}, \iota, x^{-1}\lambda, \psi, u, \alpha)]
\end{align*}
only depend on the principal ideal with positivity $(x)$ (see \cite[Lemma~4.5]{AIP}), and $f$ runs through the elements of $M^G(\mu_N, \fc, \kappa^\CU)(v_\tot)$ for $\fc\in \Frac(F)^{(p)}$. 
Moreover, $L_x$ respects the cuspidality, i.e. it takes $S^G(\mu_N, \fc, \kappa^\CU)(v_\tot)$ into $S^G(\mu_N, x\fc, \kappa^\CU)(v_\tot)$, so we can define
\[
S^G(\mu_N, (\nu^\CU, w^\CU))(v_\tot):=\left(\bigoplus_{\fc\in \Frac(F)^{(p)}} S^G(\mu_N, \fc, \kappa^\CU)(v_\tot)\right)/\left(L_x(f)-f\right)_{x\in \mathrm{Princ}(F)^{+, (p)}}
\]
where $f$ runs through the elements of $S^G(\mu_N, \fc, \kappa^\CU)(v_\tot)$ for $\fc\in \Frac(F)^{(p)}$. 
By the proof of \cite[Proposition~3.22]{AIP}, $S(\mu_N, \fc, \kappa^\CU)(v_\tot)$, is a Banach $\CO(\CU)$-module satisfying the property $(Pr)$, hence so is $S^G(\mu_N, (\nu^\CU, w^\CU))(v_\tot)$. 
Finally, we put 
\[
S^G(\mu_N, (\nu^\CU, w^\CU))^{\dagger}:=\varinjlim_{v>0} S^G(\mu_N, (\nu^\CU, w^\CU))(v_\tot), 
\]
called \emph{the space of overconvergent cusp forms of weight $\CU$, tame level $N$ for $G$}. 

\begin{proposition}\label{prop:bc of oc MFs of G}
Let $\CU'$ be a reduced affinoid rigid space over $E$ with a morphism $f:\CU'\to \CU$. 
Then there is a natural morphism 
\[
\alpha_{\CU, \CU'}^{G}\colon S^G(\mu_N, \fc, (\nu^\CU, w^\CU))(v_\tot)\widehat{\otimes}_{\CO(\CU)}\CO(\CU')\rightarrow S^G(\mu_N, \fc, (\nu^{\CU'}, w^{\CU'}))(v_\tot). 
\]
Assume 
\begin{enumerate}[label=(\arabic*), ref=(\arabic*)]
\item\label{prop:affinoid subdomain bc}
$f:\CU'\to \CU$ is an affinoid subdomain, or
\item\label{prop:reg im bc}
$f$ is a regular closed immersion. 
\end{enumerate} 
Then $\alpha^G_{\CU, \CU'}$ isomorphism. 
\end{proposition}

\begin{proof}
Note that $L_x$ are functorial in $\CU$. 
The desired $\alpha_{\CU, \CU'}^G$ is induced from $\alpha_{\CU, \CU'}$ and the functoriality of $\fe^\CU$ in $\CU$. 
The claim on isomorphism follows from Corollary~\ref{cor:flat bc of oc MFs of G*} 
and Proposition~\ref{prop:reg closed im bc of oc MFs of G*}. 
\end{proof}

\subsection{Hecke operators}\label{sec:Hecke operators}

Let $\CH^{Np}$ be the commutative $E$-algebra generated by the formal variables $T_{\ell}$ and $S_{\ell}$ for all ideals $\ell\subset\CO_F$ prime to $Np$, 
$\CU_p$ be the $E$-algebra generated by the formal variables $U_p$ and $U_{\fp}$ for the prime ideal $\fp|p$ of $\CO_F$, $\CH:=\CH^{Np}\otimes_{E}\CU_p$.
From \cite[\S~4]{AIP}, we have the following facts.
\begin{itemize}
\item
The Hecke operators $T_{\ell, \CU}, S_{\ell, \CU}$ for $\ell$ prime to $Np$ and $U_{\fp, \CU}$ for $\fp|p$ in Section~\ref{sec:Hecke for Gstar} are compatible with the projector $\fe^\CU$ and $L_x$, 
so there are natural induced operators on $S^G(\mu_N, (\nu^\CU, w^\CU))(v_{\tot})$. 
\item
The compact operator $U_{p, \CU}$ in Section~\ref{sec:Hecke for Gstar} is compatible with the projector $\fe^\CU$ and $L_x$, hence there is a natural induced compact operator on $S^G(\mu_N, (\nu^\CU, w^\CU))(v_{\tot})$. 
More precisely, $U_{p, \CU}$ preserves the subspace $S^G(\mu_N, \fc, \kappa^\CU)(v_\tot)$ of $S(\mu_N, \fc, \kappa^\CU)(v_\tot)$. 
After choosing a subset $\{c_i\}_{1\leq i\leq h^+}\subset \Frac^{(p)}$ of representatives for $\Cl^+(F)$, we have
\begin{equation}\label{eq:SG cong}
S^G(\mu_N, (\nu^\CU, w^\CU))(v_\tot)\simeq \bigoplus_{i=1}^{h^+} S^G(\mu_N, \fc_i, \kappa^\CU)(v_\tot). 
\end{equation}
Write $U_{p, \CU, \fc_i}$ for the operator $U_{p, \CU}$ on the above direct summands, then
\begin{equation}\label{eq:decom Up on SG}
U_p=\oplus_{i}U_{p, \CU, \fc_i}. 
\end{equation}
\end{itemize}
Therefore, we get a natural action of $\CH$ on $S^G(\mu_N, (\nu^\CU, w^\CU))(v_\tot)$ by letting $T_\ell, S_\ell, U_p$ and $U_\fp$ act by $T_{\ell, \CU}, S_{\ell, \CU}$ and $U_{\fp, \CU}$ respectively. 
It satisfies the following properties. 
\begin{itemize}
\item 
By \eqref{eq:factor Up for Gstar}, the operator $U_p$ can be factored into a composition
\begin{equation}\label{eq:factor Up}
S^G(\mu_N, (\nu^\CU, w^\CU))(v_\tot)\xrightarrow{\mathrm{res}} S^G(\mu_N, (\nu^\CU, w^\CU))(p^{-1}v_\tot)\xrightarrow{U_p} S^G(\mu_N, (\nu^\CU, w^\CU))(v_\tot). 
\end{equation}
\item
For the $x_\fp\in F^{\times, >0}$ chosen in Section~\ref{sec:Hecke for Gstar}, we have
\begin{equation}\label{eq:Up and Ufp}
U_{p, \CU}=\nu^{\CU}(p\prod_{\fp|p}x_\fp^{-e_{\fp}})\prod_{\fp|p} U_{\fp, \CU}^{e_\fp}, 
\end{equation}
as operators on $S^G(\mu_N, (\nu^\CU, w^\CU))(v_\tot)$, where $e_{\fp}$ is the ramification index at $\fp$. 
\item
The action of $\CH$ is functorial in $\CU$, compatible with the change of the radius of overconvergence. 
In particular, the isomorphisms in Proposition~\ref{prop:bc of oc MFs of G} are Hecke equivariant. 
\end{itemize}

If $f$ is an eigenvector for elements in $\CH$ simultaneously, we call it \emph{an eigenform}. 

\begin{proposition}\label{prop:integrality of Hecke eigenvalues}
Suppose $x\in\CW$ is a closed point and $\CU=\Sp(k(x))\to \CW$ is the natural map.
Then for any $\fp|p$, the eigenvalues for the action of $U_p$ and $U_\fp$ on the non-zero eigenforms in $S^G(\mu_N, (\nu^\CU, w^\CU))^{\dagger}$ lie in the ring of integers $\CO_{k(x)}$ of $k(x)$. 
\end{proposition}

\begin{proof}
This follows from \cite[Corollary 4.10]{Hattori}. 
\footnote{Although Hattori have assumed that $p$ is unramified in $F$ in \cite{Hattori}, this assumption is not needed for the integrality argument for the operators $U_{\mathfrak p}$ considered here. Indeed, the calculations in \cite[\S~4.2]{Hattori} about $U_\mathfrak{p}$ work for general $F$ and $p$. 
In the proof of \cite[Proposition~4.9]{Hattori}, he used the fact that the smooth scheme $\overline{\fM}^{\mathrm{ord}}_{\overline{\BF_p}}$ is irreducible. In general, we still know that it is contained in the Rapoport locus, so it's irreducibility follows from \cite[\S~2]{DP} (see also \cite[\$~2.1, 2.2]{Andreatta-Goren}), which is enough for the proof. \cite[Lemma~4.8]{Hattori} is purely formal and does not require any unramifiedness assumption. \cite[Lemma~3.4]{Hattori} for $\overline{\CM}(\mu_N, \fc)(v_\tot)$ and general $F$ and $p$ follows from the same argument, with the Deligne-Pappas integral model replaced by Pappas-Rapoport splitting model, and the $t_\beta$ used in the proof replaced by the generalized partial Hasse invariants $\dot{h_\tau}$ in \cite[Theorem~3.10]{Reduzzi-Xiao}. Then the same proof with these ingredients gives \cite[Corollary~4.10]{Hattori} for $U_\mathfrak{p}$. }
\end{proof}

\subsection{The cuspidal Hilbert eigenvariety}\label{sec:AIP's eigenvariety}

In \cite{AIP}, they use Buzzard's eigenvariety machine in \cite{Buzzard} to construct the cuspidal Hilbert eigenvariety. 
More precisely, we consider the following datum. 
\begin{itemize}
\item 
Let $\CH$ be the same as in Section~\ref{sec:Hecke operators} with the distinguished element $U_p$.
\item 
Let $\CW$ be the weight space for $G$ as before. 
\item 
For each $n\in\BZ_>0$, fix $v_n\in (0, \frac{1}{p^n})$ such that $v_{n+1}\leq v_n$. 
For every affinoid subdomain $\CU=\Sp R$ of $\CW$, we take the smallest positive integer $n(\CU)$, such that $\kappa^U$ is $n(\CU)$-analytic, and define $M_\CU:=S^G(\mu_N, (\nu^\CU, w^\CU))((v_{n(\CU)})_\tot)$, equipped with the natural action of $\CH$ as in Section~\ref{sec:Hecke operators}. 
\item
For two affinoid subdomains $\CU'\subset \CU\subset \CW$, 
let $\beta_{\CU, \CU'}$ be the composition 
\[
M_{\CU'}=S^G(\mu_N, (\nu^{\CU'}, w^{\CU'}))((v_{n(\CU')})_\tot)\to S^G(\mu_N, (\nu^{\CU'}, w^{\CU'}))((v_{n(\CU)})_\tot)\xrightarrow{(\alpha_{\CU, \CU'}^G)^{-1}}M_{\CU}\widehat{\otimes}_{\CO(\CU)}\CO(\CU'),
\]
where the first arrow is the restriction map. 
\end{itemize}
Then one can check that $\beta_{\CU, \CU'}$ is a link using \eqref{eq:factor Up}, and these define an input datum for the eigenvariety machine in \cite{Buzzard}. 
Also, see \cite[\S3.3.3]{Hattori}. 
Therefore, we get the following theorem. 

\begin{theorem}\cite[Theorem~5.1]{AIP}
There is a rigid analytic space $\wt\colon  \CE\to \CW$, called \emph{the cuspidal Hilbert eigenvariety}, with a morphism $\Theta\colon  \CH\to \CO_\CE$, satisfying the following propositions. 
\begin{itemize}
\item $\CE$ is equidimensional of dimension $g+1$. 
\item The map $\wt$ is, locally on $\CE$ and $\CW$, finite and surjective. 
\item For any $x=(\nu, w)\in\CW$ and any finite field extension $L$ of $k(x)$, the set of $L$-points $z$ of $\CE$ such that $\wt(z)=x$ is in natural bijection with the systems of eigenvalues of $\CH$ appearing in $S^G(\mu_N, (\nu, w))^{\dagger}\widehat{\otimes}_{k(x)}L$ such that the eigenvalue of $U_p$ is non-zero. 
\end{itemize}
In particular, by construction, there is a Zariski closed subvariety $\CZ\subset \CW\times\BA^1_{\rig}$, called \emph{the spectral variety}, such that the morphism $(\wt, \Theta(U_p)^{-1})\colon \CE\to \CW\times \BA^1_{\rig}$ factors through a finite morphism $\CE\to \CZ$. 
\end{theorem}

As a corollary of Proposition~\ref{prop:integrality of Hecke eigenvalues}, we have: 

\begin{proposition}\label{prop:power boundedness}
For any $x\in\CE$, we have $|\Theta(U_p)(x)|, |\Theta(U_\fp)(x)|\leq 1$ for all $\fp|p$. 
\end{proposition} 

\subsection{Classical points with cohomological weights on the eigenvariety}

Let $X:=(\BZ^I\oplus \BZ)\subset \CW(E)$ be the set of algebraic weights. 
More precisely, for $(\nu, w)=((\nu_\tau)_{\tau\in I}, w)\in \BZ^I\oplus \BZ$, we have the associated character 
\[
H=\BT(\BZ_p)\times\BZ_p^\times\to \CO_E^\times, \quad (a, b)\mapsto \prod_{\tau\in I}\tau(a)^{\nu_\tau}b^w, 
\]
which defines an $E$-point of $\CW$. 
In this case, let $S^G(\Gamma_0(\fP), \mu_N, (\nu, w))$ be the space of classical cusp forms for $G$ of weight $(\nu, w)$, tame level $N$ and Iwahori level $\Gamma_0(\fP)$, where $\fP:=\prod_{\fp|p}\fp$. 
Then the theory of canonical subgroups gives a natural inclusion 
\[
S^{G}\bigl(\Gamma_0(\mathfrak{P}), \mu_N,(\nu,w)\bigr)\hookrightarrow S^G(\mu_N, (\nu, w))^{\dagger}. 
\]
An overconvergent cusp form $f\in S^G(\mu_N, (\nu, w))^{\dagger}$ is called \emph{classical} if it lies in the image of the above embedding. 
For details, see \cite{Bij}, where Bijakowski studied the classicality of overconvergent modular forms for the group $G^*$, and one can use the recipe in Section~\ref{subsec:oc MFs for G} to define the classical modular forms for $G$ in terms of that for $G^*$. 
For a comparison with the adelic definition, see \cite{P}. 

\begin{definition}
Let $x\in\CE$ be a closed point.
Suppose $\wt(x)=(\nu, w)$ is an algebraic weight. 
We say that $x$ is a \emph{classical point} if the system of Hecke eigenvalues corresponding to $x$ appears in the subspace $S^G(\Gamma_0(\fP), \mu_N, (\nu, w))\otimes_E k(x)$ of $S^G(\mu_N, (\nu, w))^{\dagger}\otimes_E k(x)$. 
\end{definition}

The following theorem is included in \cite[Theorem 5.1]{AIP}, which follows from the main theorem of \cite{Bij}.  

\begin{theorem}\label{thm:classicality}
Let $(\nu,w)\in (\BZ^I\oplus\BZ)\subset \CW(E)$ be an algebraic weight. 
Write $\nu=(\nu_\tau)_{\tau\in I}$ and set $k_{\sigma} = 2\nu_{\sigma}+w$.
Let $f \in S^G(\mu_{N},(\nu,w))^{\dagger}$
be a finite-slope overconvergent eigenform, with $U_\fp$-eigenvalue $s_\fp$. 
Assume that for each prime $\fp|p$ of $F$,
\[
\inf_{\sigma\in I_\fp} k_{\sigma}>e_\fp(v_p(s_\fp)+f_\fp),
\]
where $f_{\fp} = [\mathcal{O}_{F}/\mathfrak{p}\colon  \mathbb{F}_{p}] $.
Then $f$ is a classical cuspidal eigenform for the group $G$. 
\end{theorem}

\begin{corollary}\label{cor:classicality}
Keep the notation as in the previous theorem. 
Let $s_p$ be the $U_p$-eigenvalue of $f$, $M:=v_p(s_p)+\max_{\fp}e_\fp f_\fp$. 
If $k_\sigma>M$ for all $\sigma\in I$, then $f$ is a classical cuspidal eigenform for the group $G$. 
\end{corollary}

\begin{proof}
By Proposition~\ref{prop:power boundedness} and \eqref{eq:Up and Ufp}, we have $e_\fp v_p(s_\fp)\leq v_p(s_p)$. 
So we conclude by Theorem~\ref{thm:classicality}. 
\end{proof}

We define the set of cohomological weights to be
\[
X_{\coh}:=\left\{((\nu_\tau)_\tau,w)\in X:\nu_\tau\le0,\quad 2\nu_\tau+w\ge2\ \text{for every }\tau\in I\right\} \subset X, 
\]
and for $M\in\BR$, let
\[
X_{\coh}(M)=\left\{((\nu_\sigma)_\sigma,w)\in X_{\coh}:
2\nu_\sigma+w>M\quad\text{for every }\sigma\in I
\right\}\subset X_{\coh}.
\]

\begin{proposition}\label{prop:density of coh wts}
For every $M\in\BR$ and every $x\in X_{\coh}$, there is a fundamental system of affinoid neighbourhoods $V_n$ of $\chi_x$ in $\CW$, such that $X_{\coh}(M)\cap V_n$ is Zariski dense in $V_n$ for every $n$.
\end{proposition}

\begin{proof}
See Appendix~\ref{app:coh wts}. 
\end{proof}

\begin{proposition}\label{prop:density of classical points with coh wts}
The subset $\CC$ of classical points with cohomological weight is very Zariski-dense in $\CE$. 
\end{proposition}

\begin{proof}
This is the same as \cite[Proposition~3.8.6]{Bellaiche}, which follows from Corollary~\ref{cor:classicality} and Proposition~\ref{prop:density of coh wts}.
\end{proof}

\subsection{Closed subvarieties of $\CZ$}\label{sec:spectral variety}

Recall that $\CZ\subset \CW\times\BA^1_{\rig}$ is the spectral variety of the eigenvariety $\CE$. 
In this subsection, we study certain closed subvarieties of it. 

For each $\fp|p$, let
\[
\nu_{\fp}^{\un}\colon \CO_{F_{\fp}}^\times\to \CO_E[[\CO_{F_\fp}^{\times}]]^\times\subset \CO_E[[\BT(\BZ_p)]]^\times
\]
be the universal character of $\CO_{F_{\fp}}^\times$. 
Then $\nu^{\un}=\prod_{\fp|p}\nu_\fp^{\un}$. 
Any element in the image of $\nu_\fp^\un$ defines a power bounded function on $\CW$. 
Recall that for any rigid analytic space $X$ over $E$ and any continuous character $\delta\colon  F_\fp^\times\to \CO(X)^\times$, we have the notion of the weight of $\delta$. 
More precisely, if $X=\Sp(S)$ is affinoid, then there is a family of elements $(\wt_\tau(\delta))_{\tau\in I_\fp}$ in $\CO(X)$, such that 
\[
\lim_{\substack{a\to0 \\ a\in\CO_{F_\fp}}}\frac{|\delta(1+a)-1-\sum_{\tau\in I_\fp}\wt_{\tau}(\delta)\tau(a)|}{|a|}=0,
\]
its image under $\prod_{\tau\in I_\fp}\CO(X)\cong F_\fp\otimes_{\BQ_p}\CO(X)$ is denoted by $\wt(\delta)$, see \cite[Definition~6.1.6]{KPX}. 
Setting $\nu_\fp^\un(x_\fp)=1$, we view $\nu_\fp^\un$ as a character $F_{\fp}^\times\to \CO(\CW)^\times$ and get $(\wt_\tau(\nu_\fp^\un))_{\tau\in I_\fp}$. 
For example, if $x=(\nu, w)\in \CW$ is an algebraic weight with $\nu=(\nu_\tau)_{\tau\in I_\fp}$, then $\wt_\tau(\nu_\fp^\un)$ has value $\nu_\tau$ at $x$. 
Similarly, we have $\wt(w^\un)\in\CO(\CW)$. 
For any $\tau\in I_\fp$, let 
\[
K_{\tau}=1-2\wt_{\tau}(\nu_\fp^\un)-\wt(w^\un)\in \CO(\CW).
\]
Fix $\kappa=(k_\tau)_{\tau\in I}\in\BZ^I$. Let
\[
\CW_{\kappa}:=V(K_\tau-k_\tau: \tau\in I)\subset\CW
\]
be the Zariski closed subset defined by $K_\tau-k_\tau$. 

\begin{lemma}\label{lem:reg cl im}
$\CW_\kappa$ is reduced, and the natural morphism $\CW_{\kappa}\to \CW$ is a regular closed immersion. 
\end{lemma}

\begin{proof}
See Proposition~\ref{prop:reg cl im} in Appendix~\ref{app:coh wts}. 
\end{proof}

For any affinoid subdomain $\CV\subset\CW_{\kappa}$, take the smallest positive integer $n(\CV)$ such that $(\nu^\CV, w^\CV)$ is $n(\CV)$-analytic. 
We obtain a new eigenvariety datum with the same $\CH$ and $U_p\in\CH$ as in Section~\ref{sec:AIP's eigenvariety}, but the weight space becomes $\CW_{\kappa}$ and $M_\CV:=S^G(\mu_N, (\nu^\CV, w^\CV))((v_{n(\CV)})_\tot)$. 
Let $\CZ_{\kappa}:=\CZ\times_{\CW}\CW_{\kappa}\subset \CW_{\kappa}\times\BA^1_\rig$. 

\begin{proposition}\label{prop:bc of spec var}
$\CZ_{\kappa}\to \CW_{\kappa}$ is the spectral variety associated to the above new eigenvariety datum. 
\end{proposition}

\begin{proof}
Consider an affinoid subdomain $\CU\subset\CW$, then $\CU':=\CU\times_\CW\CW_{\kappa}$ is an affinoid subdomain of $\CW_{\kappa}$.
By construction, $\CZ\times_\CW\CU\subset \CU\times \BA^1_\rig$ is the closed subvariety defined by the characteristic power series of the compact operator $U_p$ on the $\CO(\CU)$-module $S^G(\mu_N, (\nu^{\CU}, w^{\CU}))((v_{n(\CU)})_\tot)$.
By Lemma~\ref{lem:reg cl im}, the natural morphism $\CU'\to \CU$ is a regular closed immersion. 
So the $\CH$-equivariant isomorphism $\alpha_{\CU, \CU'}^G$ given by Proposition~\ref{prop:bc of oc MFs of G} tells us that 
$\CZ_{\CU'}:=\CZ\times_{\CW}\CU'\to \CU'$ is the spectral variety associated to the compact operator $U_p$ on the $\CO(\CU')$-module $S^G(\mu_N, (\nu^{\CU'}, w^{\CU'}))((v_{n(\CU)})_\tot)$. 
Using the factorization in \eqref{eq:factor Up}, it is easy to see that this $\CO(\CU')$-module and $M_{\CU'}$ are linked via the natural restriction map. 
Since $\CZ_{\kappa}$ can be covered by these $\CZ_{\CU'}$ when $\CU$ running through an admissible cover of $\CW$, we are done. 
\end{proof}

Let $\CV$ be any connected affinoid subdomain of $\CW_{\kappa}$. 

\begin{lemma}\label{lem:loc alg wt}
There exists a closed point $x_{\CV}\in\CW^*$, such that the composition $k|_\CV: \CV\to\CW\xrightarrow{k}\CW^*$ factors through the natural morphism $\Sp(k(x_{\CV}))\to \CW^*$. 
\end{lemma}

\begin{proof}
Recall that the continuous character associated to $\CV\to \CW^*$ is given by 
\[
\kappa^{\CV}=2\nu^{\CV}+w^{\CV}\circ \Norm_{F/Q}: (\CO_F\otimes_\BZ\BZ_p)^\times=\prod_{\fp|p}\CO_{F_{\fp}}^\times \to \CO(\CV)^\times. 
\]
For each $\fp|p$, let $(\kappa^\CV)_\tau: \CO_{F_\fp}^\times\to \CO(\CV)^\times$ be the $\fp$-component of $\kappa^\CV$. 
Then $\kappa^\CV=\prod_{\fp|p}(\kappa^\CV)_\fp$ and 
\[
(\kappa^\CV)_\fp=\iota_\CV^*\circ (2\nu^\un+w^\un\circ \Norm_{\CO_{F_{\fp}}/\BZ_p}), 
\]
where $\iota_\CV^*: \CO(\CW)\to \CO(\CV)$ is the morphism induced by $\iota_\CV: \CV\to \CW$. 
Note that the derivation of $\Norm_{\CO_{F_{\fp}}/\BZ_p}$ is given by $\Tr=\sum_{\tau\in I_\fp}\tau: F_\fp\to \BQ_p$. 
Therefore, after identifying $\kappa':=(1-k_\tau)_{\tau\in I}\in\BZ^I\subset\CW^*(E)$ with its associated character 
\[
\prod_{\fp|p}\kappa_\fp': \BT(\BZ_p)=\prod_{\fp|p}\CO_{F_\fp}^\times\to \CO(\CW^*)^\times, 
\]
we have
\begin{align*}
\wt_\tau(\kappa^\CV_\fp-(k|_\CV)^*\circ \kappa_\fp')
&=\iota_\CV^*\left(2\wt_\tau(\nu^\un)+\wt(w^\un)\right)-(k|_\CV)^*(\wt_\tau(\kappa_\fp'))\\
&=1-k_\tau-(1-k_\tau) \\
&=0. 
\end{align*}
So $\chi:=\kappa^\CV-(k|_\CV)^*\circ\kappa'$ is trivial on some open subgroup of $\BT(\BZ_p)$, and it factors through some finite quotient of $\BT(\BZ_p)$. 
Each element in the image of $\chi$ has finite order in $\CO(\CV)$, thus they generate a finite \'etale $E$-subalgebra $A$ of $\CO(\CV)$. 
Since $\CV$ is connected, there are no idempotents in $\CO(\CV)$, hence $A$ is a field. 
The character $\kappa^\CV=(k|_{\CV})^*\circ\kappa'+\chi$ is locally algebraic with value in the finite extension $A$ of $E$, giving the desired $x_\CV$, and we have $A=k(x_\CV)$. 
\end{proof}

Apply Proposition~\ref{prop:constant family for Gstar} to $\CV\to \Sp(k(x_\CV))\to \CW^*$, we get the isomorphism
\[
S(\mu_N, \fc, \kappa^\CV)((v_{n(\CV)})_\tot)\cong S(\mu_N, \fc, \kappa^{x_\CV})((v_{n(\CV)})_\tot)\widehat{\otimes}_{k(x_\CV)}\CO(\CV). 
\]
Recall that there is a projector $\fe^\CV$ on $S(\mu_N, \fc, \kappa^\CV)((v_{n(\CV)})_\tot)$. 

\begin{lemma}\label{lem:descent feCV}
There exists a finite field extension $E_\CV$ of $k(x_\CV)$, such that the $E_\CV$-subspace 
\[
M_{\fc, \CV, E_\CV}:=S(\mu_N, \fc, \kappa^{x_\CV})((v_{n(\CV)})_\tot)\widehat{\otimes}_{k(x_\CV)}E_\CV
\] 
of $S(\mu_N, \fc, \kappa^{x_\CV})((v_{n(\CV)})_\tot)\widehat{\otimes}_{k(x_\CV)}\CO(\CV)$
is stable under the action of $\Delta$. 
In particular, let $\fe_{\fc, \CV, E_\CV}$ be the restriction of $\fe^\CV$ to $M_{\fc, \CV, E_\CV}$, 
then 
\[\fe^{\CV}=\fe_{\fc, \CV, E_\CV}\widehat{\otimes}_{E_\CV}1_{\CO(\CV)}\]
on $S(\mu_N, \fc, \kappa^\CV)((v_{n(\CV)})_\tot)\cong M_{\fc, \CV, E_\CV}\widehat{\otimes}_{E_\CV}\CO(\CV)$ and
\begin{equation}\label{eq:descent SG}
S^G(\mu_N, \fc, \kappa^\CV)((v_{n(\CV)})_\tot)\cong M_{\fc, \CV, E_\CV}^{\Delta}\widehat{\otimes}_{E_\CV}\CO(\CV). 
\end{equation}
\end{lemma}

\begin{proof}
For any $f\in S(\mu_N, \fc, \kappa^{x_\CV})((v_{n(\CV)})_\tot)$, any tuple $(A_{/R}, \iota, \lambda, \psi, u, \alpha)$ as in Section~\ref{subsec:oc MFs for G}, and $\epsilon\in \CO_F^{\times, +}$, by definition, 
\[
(\epsilon\cdot (f\otimes1))(A_{/R}, \iota, \lambda, \psi, u, \alpha)=\nu^\CV(\epsilon)f(A_{/R}, \iota, \epsilon^{-1}\lambda, \psi, u, \alpha). 
\]
So it suffices to show that $\{\nu^{\CV}(\epsilon)\in \CO(\CV)^\times: \epsilon\in \CO_F^{\times, +}\}$ is contained in some finite extension $E_\CV$ of $k(x_\CV)$. 
In fact, for any $\epsilon\in \CO_F^{\times, +}\subset \BT(\BZ_p)$, $\Norm_{F/Q}(\epsilon)=1$. 
Thus 
\[
\nu^\CV(\epsilon)^2=\nu^\CV(\epsilon)^2\cdot w^\CV\circ\Norm_{F/Q}(\epsilon)=\kappa^\CV(\epsilon)\in k(x_\CV)^\times. 
\]
By Dirichlet unit theorem, we can take a finite set of generators $\{\epsilon_i\}_{i=1}^{g-1}$ of $\CO_{F}^{\times, +}$. 
The same argument as in the proof of Lemma~\ref{lem:loc alg wt} shows that the $k(x_\CV)$-subalgebra of $\CO(\CV)$ generated by $\{\nu^\CV(\epsilon_i)\}_{i=1}^{g-1}$ is a finite field extension $E_\CV$ of $k(x_\CV)$, which is the desired one. 
\end{proof}

Consider $\CZ_{\CV}:=\CZ_\kappa\times_{\CW_\kappa}\CV\to \CV$.
By Proposition~\ref{prop:bc of spec var}, it is the closed subvariety of $\CV\times\BA^1_\rig$ defined by the characteristic power series $F_{\CV}\in \CO(\CV\times \BA^1_{\rig})$ of the compact operator $U_p$ on $M_\CV$. 

\begin{theorem}\label{thm:spectral}
$\CZ_{\CV}\to \CV$ is isomorphic to $\coprod_{i\geq1}\CZ_{\CV, i}\to \CV$ for some finite flat morphisms $\CZ_{\CV, i}\to \CV$. 
\end{theorem}

\begin{proof}
By \eqref{eq:SG cong} and \eqref{eq:descent SG}, 
\[
M_{\CV}\cong \bigoplus_{i=1}^{h^+} S^G(\mu_N, \fc_i, \kappa^\CV)((v_{n(\CV)})_\tot)
\cong \bigoplus_{i=1}^{h^+}M_{\fc_i, \CV, E_\CV}^{\Delta}\widehat{\otimes}_{E_\CV}\CO(\CV). 
\]
By \eqref{eq:decom Up on SG}, $U_p=\oplus_i U_{p, \CV, \fc_i}$ through the first isomorphism. 
Since $\fe^{\CV}=\fe_{\fc, \CV, E_\CV}\widehat{\otimes}_{E_\CV}1_{\CO(\CV)}$ commutes with 
\[U_{p, \CV, \fc}=U_{p, x_\CV, \fc}\widehat{\otimes}_{k(x_\CV)}1_{\CO(\CV)}
=U_{p, x_\CV, \fc}\widehat{\otimes}_{k(x_\CV)}1_{E_\CV}\widehat{\otimes}_{E_\CV}1_{\CO(\CV)}
\]
on
\begin{align*}
S(\mu_N, \fc, \kappa^\CV)((v_{n(\CV)})_\tot)
&\cong S(\mu_N, \fc, \kappa^{x_\CV})((v_{n(\CV)})_\tot)\widehat{\otimes}_{k(x_\CV)}\CO(\CV) \\
&\cong S(\mu_N, \fc, \kappa^{x_\CV})((v_{n(\CV)})_\tot)\widehat{\otimes}_{k(x_\CV)}E_\CV\widehat{\otimes}_{E_\CV}\CO(\CV), 
\end{align*}
and $E_\CV\to \CO(\CV)$ is faithfully flat, $\fe_{\fc, \CV, E_\CV}$ also commutes with $U_{p, x_\CV, \fc}\widehat{\otimes}_{k(x_\CV)}1_{E_\CV}$. 
So $M_{\fc, \CV, E_\CV}^{\Delta}$ is stable under $U_{p, x_\CV, \fc}\widehat{\otimes}_{k(x_\CV)}1_{E_\CV}$. 
Let $U_{p, x_{\CV}, \fc}'$ be the restriction of $U_{p, x_\CV, \fc}\widehat{\otimes}_{k(x_\CV)}1_{E_\CV}$ to $M_{\fc, \CV, E_\CV}^{\Delta}$, then $U_{p, \CV, \fc}=U_{p, x_\CV, \fc}'\widehat{\otimes}_{E_\CV}1_{\CO(\CV)}$
under the isomorphism \eqref{eq:descent SG}. 
Hence 
\[
U_p=\left(\bigoplus_{i=1}^{h^+} U_{p, x_\CV, \fc_i}'\right)\widehat{\otimes}_{E_\CV}1_{\CO(\CV)}
\]
on $M_\CV$. 
$U_{p, x_\CV, \fc}$ is a compact operator on the Banach $E_\CV$-module $M_{\fc, \CV, E_\CV}$ satisfying the property $(Pr)$, thus the direct summation above is a compact operator on the Banach $E_\CV$-module $\oplus_{i}M_{\fc_i, \CV, E_\CV}^{\Delta}$ which also satisfies the property $(Pr)$. 
Let $F_{x_\CV}\in \CO(\BA^1_{\rig, E_\CV})$ be its characteristic power series, which defines a Zariski closed subvariety $\CZ_{x_\CV}\subset \BA^1_{\rig, E_\CV}$. 
Then $F_\CV$ is the image of $F_{x_\CV}$ under the natural map $\CO(\BA^1_{\rig, E_\CV})\to \CO(\CV\times_E \BA^1_{\rig})$, and 
\[
\CZ_{\CV}\cong \CZ_{x_\CV}\times_{E_\CV} \CV\to \CV 
\]
is the base change of $\CZ_{x_\CV}\to \Sp(E_\CV)$. 
The desired conclusion follows easily. 
\end{proof} 
\section{Galois representations associated to overconvergent eigenforms}\label{sec:Gal rep}

\subsection{Families of pseudo representations on the cuspidal Hilbert eigenvarieties}

Recall that in Section~\ref{sec:Hecke operators}, we have chosen $x_\fp$ for each place $\fp$ of $F$ above $p$. 
Let $U_\fp^{\mathrm{class}}$ be the classical Hecke operator at $\fp$ acting on $S^G(\Gamma_0(\fP), \mu_N, (\nu, w))$. Then by \cite[Remark 4.7]{AIP}, we have $U_{\fp}^{\mathrm{class}}=\nu(x_\fp)^{-1}U_{\fp}$.
\footnote{There is a typo in \cite[Remark 4.7]{AIP}, where they have written $U_{\fp}^{\mathrm{class}}=\nu(x_\fp)U_{\fp}$. 
To verify the correct formula, one may compare $U_{\fp}$ with the operator in \cite[\S 3.18]{TianXiao}, where Tian--Xiao give a definition of $U_{\fp}^{\mathrm{class}}$ using Hecke correspondence. 
Alternatively, one can use Hattori's formula \cite[Lemma~4.5(2)]{Hattori} on $U_\fp$ in terms of $q$-expansions. 
}

\begin{theorem}\label{thm:Gal rep ass to classical cusp forms}
Let $(\nu, w)$ be a cohomological weight, $f\in S^G(\Gamma_0(\fP), \mu_N, (\nu, w))$ be a classical cuspidal eigenform of weight $(\nu, w)$.
Then there exists a continuous odd irreducible representation $\rho_f\colon G_F\to \GL_2(\overline{\BQ_p})$, such that the following properties are satisfied: 
\begin{enumerate}
\item 
For every finite place $\ell\nmid pN$, if $\lambda_{\ell}$ denotes the eigenvalue of the Hecke operator $T_\ell$ on $f$, then $\rho_f$ is unramified at $\ell$ and the characteristic polynomial of $\rho_f(\Frob_{\ell})$ is $T^2-\lambda_\ell T+s_\ell(f)N_{F/\BQ}(\ell)$, where $\Frob_\ell$ is a geometric Frobenius, $s_\ell(f)$ is the eigenvalue of $S_\ell$ on $f$. 
\item 
For a place $\fp|p$ of $F$, if $\lambda_\fp$ is the eigenvalue of the $U_{\fp}$-operator, then the local Galois representation $\rho_f|_{G_{F_\fp}}$ is potentially semistable with Hodge-Tate weights $\nu_\tau, 1-\nu_\tau-w$ in the $\tau$-direction for each $\tau\in I_\fp$, and $\Dcrysplus(\rho_f|_{G_{F_{\fp}}})^{\varphi^{f_\fp}=\nu(x_\fp)^{-1}\lambda_{\fp}}\neq0$. 
\footnote{We normalize so that the cyclotomic character has Hodge-Tate weight $+1$. }
\end{enumerate}
\end{theorem}

\begin{proof}
The construction of $\rho_f$ is due to Carayol~\cite{Carayol}, Taylor~\cite{Taylor89}, and Blasius-Rogawski~\cite{Blasius-Rogawski}. 
The verification of the properties for $\rho_f$ was done by \cite{Carayol} for places outside $p$ and by Saito~\cite{Saito} (plus a special case handled independently by T.~Liu~\cite{TLiu} and Skinner~\cite{Skinner}) at places above $p$. 
\end{proof}

\begin{theorem}\cite[Theorem~5.1]{AIP}\label{thm:pschar on E}
\begin{enumerate}
\item 
There exists a continuous two-dimensional pseudo-character
\[
T \colon G_{F} \longrightarrow \CO(\CE)^0:=\{f\in\CO(\CE)\mid |f(x)|\leq1, \forall x\in \CE\}, 
\]
satisfying $T(\mathrm{Frob}_{\ell}) = \Theta(T_{\ell})$ for every prime ideal $\ell$ of $\mathcal{O}_{F}$ with $\ell \nmid Np$. 
Moreover, $T$ is odd, i.e., let $c\in G_F$ be a complex conjugation, then $T(c)=0$. 
\item
For every $x\in \mathcal{E}$, there is a semisimple Galois representation
\[
\rho_{x}^{\ss} \colon G_F \longrightarrow \mathrm{GL}_{2}\bigl(k(x)\bigr)
\]
characterized by $\operatorname{Tr}(\rho_{x}^{\ss})=T|_{x}$. 
\end{enumerate}
\end{theorem}

\begin{proof}
By the proof of \cite[Theorem~6.10]{Birkbeck}, the eigenvariety $\CE$ is reduced. 
\footnote{In \cite[Theorem~6.10]{Birkbeck}, Birkbeck assumed that $p$ is unramified in $F$. 
This condition is only used to apply the control theorem \cite[Theorem~4.22]{Birkbeck}. 
In general, one can use the same argument as in his proof to show the reducedness of $\CE$. }
So the first point follows from Theorem~\ref{thm:Gal rep ass to classical cusp forms}, Proposition~\ref{prop:density of classical points with coh wts} and \cite[Proposition~7.1.1]{ChenevierGLn}. 
The second follows from \cite[Proposition~1.1]{Hida}, using the oddness $T(c)(x)=0$. Also, see \cite[p.560]{Wiles}. 
\end{proof}

\subsection{Pointwise trianguline property}

Fix a prime ideal $\fp|p$ of $\CO_F$.
We take $x_\fp$ as a uniformizer of $F_\fp$. 
By local class field theory, which is normalized so that a geometric Frobenius corresponds to $x_\fp$, we can view $\nu_\fp^\un$ as a character of $W_{F_\fp}$ valued in $\CO(\CW)^\times$, which extends to $G_{F_\fp}$ by continuity. 
Consider it as a $p$-adic family of one-dimensional representations of $G_{F_\fp}$ on $\CW$. 
For each algebraic weight $x=(\nu, w)\in \CW(E)$, write $\nu=(\nu_\tau)_{\tau\in I}$ and decompose $\nu=\prod_{\fp|p}\nu_\fp$. 
Then $(\nu_\fp^\un)^{-1}|_x$ has Hodge-Tate weights $-\nu_\tau$ in the $\tau$-direction for each $\tau\in I_\fp$, and $\varphi^{f_\fp}$ acts on $\Dcrys((\nu_\fp^\un)^{-1}|_x)$ by $\prod_{\tau\in I_\fp}\tau(x_\fp)^{\nu_\tau}$. 

For any $\fq\neq\fp$, since $x_\fp$ is a $\fq$-adic unit, we can evaluate $\nu_\fq^\un$ on it.
Pulling back to $\CE$, we get $\nu_{\fq}^\un(x_\fp)\in\CO(\CE)^\times$. 
Let $\alpha_\fp=\Theta(U_\fp)\prod_{\fq\neq\fp}\nu_\fq^\un(x_\fp)^{-1}\in \CO(\CE)^\times$. 
The pullback of $\nu_\fp^\un$ to $\CE$ is still denoted by $\nu_\fp^\un$. 
For any $x\in\CE$, let $V_{\fp, x}:=\rho_x^{\ss}|_{G_{F_\fp}}\otimes(\nu_\fp^\un)^{-1}|_x$, which is a representation of $G_{F_\fp}$ over $k(x)$. 

\begin{theorem}\label{thm:pointwise trianguline}
For any $x\in\CE$, we have $\Drig(V_{\fp, x})^{\varphi^{f_\fp}=\alpha_\fp(x), \Gamma_{F_\fp}=1}\neq0$. 
\end{theorem}

\begin{proof}
If $x\in\CC$, i.e. $x$ is a classical point with cohomological weight, then the conclusion follows from Theorem~\ref{thm:Gal rep ass to classical cusp forms} and Theorem~\ref{thm:comparison of various D}. 
For general $x\in \CE$, we use the idea of the proof of \cite[Proposition~6.6]{McDonald}. 
By \cite[Lemma~7.2.11(1)]{BC}, the eigenvariety $\CE$ is nested, so it admits an admissible covering by some affinoid subdomains $\{U_i\}_{i\geq 0}$, such that $U_i\subset U_{i+1}$. 
Let $V_m:=\overline{\CC\cap U_m}\subset U_m$ be the Zariski closure of $\CC\cap U_m$. 

Take any $x\in \CE$, then $x\in U_{m'}$ for some $m'\geq0$. 
By \cite[Lemma~6.4]{McDonald}, there exists some $m\geq m'$, such that $U_{m'}\subset V_m$, which is still affinoid. 
Consider the restriction $T_{m}$ of $T$ to $V_m$. 
Then by \cite[Lemma~7.8.11]{BC}, there exists a reduced, separated, quasi-compact rigid space $Y$ equipped with a proper, generically finite, surjective map $f\colon Y\to V_m$ such that there exists a family of $p$-adic representations $V_Y$ of $G_F$ of dimension two on $Y$, whose associated pseudo-character is $f^*T_{m}$. 
Consider the Zariski-closure $Y'=\overline{f^{-1}(\CC\cap U_m)}\subset Y$.
Since $f$ is proper, by \cite[\S~9.6.3, Proposition~3]{BGR}, it maps Zariski-closed subsets to Zariski-closed subsets, so $f|_{Y'}\colon Y'\to V_m$ is still surjective. 
Replacing $Y$ by $Y'$, we may assume that $f^{-1}(\CC\cap V_m)=f^{-1}(\CC\cap U_m)$ is Zariski-dense in $Y$. 
Let $\nu_\fp|_Y$ be the pullback of $\nu_\fp^\un$ to $Y$.

For any $y\in f^{-1}(\CC\cap V_m)$, by Theorem~\ref{thm:Gal rep ass to classical cusp forms} and Theorem~\ref{thm:pschar on E}, $V_y^{\ss}\simeq \rho_{f(y)}^{\ss}\otimes_{k(f(y))}k(y)$ is absolutely irreducible, hence $V_y\simeq V_y^{\ss}$ and $\Drig(V_y|_{G_{F_{\fp}}}\otimes (\nu_\fp^\un)^{-1}|_{y})^{\varphi^{f_\fp}=f^*\alpha_\fp(y), \Gamma_{F_\fp}=1}\neq0$ by the similar result for $f(y)\in\CC$ and Corollary~\ref{cor:derived bc of herr cpx}. 
Therefore, applying Theorem~\ref{thm:herr-main} to the family of $(\varphi, \Gamma_{F_\fp})$-modules associated to $V_Y|_{G_{F_\fp}}\otimes (\nu_\fp)^{-1}|_Y$, we get the same non-vanishing result for all $y\in Y$, and it's easy to see that it also holds with $V_y$ replaced by $V_y^{\ss}$. 
For any $y\in Y$, $V_y^{\ss}\simeq \rho_{f(y)}^{\ss}\otimes_{k(f(y))}k(y)$, we conclude by Corollary~\ref{cor:derived bc of herr cpx} and the surjectivity of $f$. 
\end{proof}

\subsection{Sen polynomial}

Consider the functions $K_{\tau}\in\CO(\CW)$, 
we still denote their pullback to $\CE$ by the same notation. 

\begin{proposition}\label{prop:Sen poly}
For any $x\in\CE, \tau\in I_\fp$, the $\tau$-component of the Sen polynomial $P(T)(x)\in F_{\fp}\otimes_{\BQ_p}k(x)[T]$ of $V_{\fp, x}|_{G_{F_\fp}}$ is given by $TQ_{\fp}(T)_{\tau}(x)$, where $Q_{\fp}(T)_{\tau}=T-K_{\tau}$.
\end{proposition}

\begin{proof}
When $x\in\CC$, the statement follows from Theorem~\ref{thm:Gal rep ass to classical cusp forms}. 
Since $\CC\subset\CE$ is very Zariski-dense, we conclude by \cite[Lemma~7.5.12]{BC}. 
\footnote{In \cite[Lemma~7.5.12]{BC}, they only consider the representations of $G_{\BQ_p}$. 
But it is easy to see the same argument works here.}
\end{proof} 
\section{Proof of the main theorem}\label{sec:properness}

\begin{theorem}\label{thm:main}
Let $\CD$ be the closed unit disc over a finite extension $E'/E$, and let
$\CD^*=\CD\setminus\{0\}$ be the punctured disc. 
Suppose that $h\colon \CD^*\rightarrow \CE$ is a morphism of rigid analytic spaces such that $\wt\circ h$ extends to $h'\colon\CD\to\CW$.
Assume that, for every prime $\fp\mid p$ of $\CO_F$, there exists an embedding  $\tau_{0,\fp}\colon F_{\fp}\hookrightarrow\bQp$ such that, for every $j\in\BZ_{\leq 0}$, the pullback 
$h^*\bigl(Q_{\fp}(j)_{\tau_{0,\fp}}\bigr)$ is not identically zero on $\CD^*$. 
Then $h$ extends to a morphism $\widetilde h\colon \CD\longrightarrow\CE$ compatible with $h'$. 
If $\CO_F$ has only one prime ideal above $p$, the conclusion holds without the non-vanishing assumption on the Sen polynomials.
\end{theorem}

Let $h$ be the morphism in the above theorem. 
Then the pullback $h^*T$ of $T$ to $\CD^*$ takes value in $\CO(\CD^*)^0=\{f\in\CO(\CD^*)\mid |f(x)|\leq1, \forall x\in\CD^*\}$, hence in $\CO(\CD)$ by \cite[Lemma~2.1]{DiaoLiu}, i.e. we get a two-dimensional continuous pseudo-character $T_{\CD}$ of $G_F$ on $\CD$ extending $h^*T$.
Then $T_{\CD}(c)=0$. 

\begin{lemma}
There exists a $p$-adic family of representations $V_{\CD}$ of $G_F$ and a locally finite subset $E_{\mathrm{exc}}\subset \CD$, such that for any $x\in \CD\backslash E_{\mathrm{exc}}, V_{x}$ is semisimple, and for any $x\in\CD^*, \Tr(V_x)=\Tr(\rho_{h(x)}^\ss)$. 
In particular, when $x\in \CD^*\backslash E_{\mathrm{exc}}$, $V_x\simeq \rho_{h(x)}^\ss\otimes_{k(h(x))}k(x)$. 
\end{lemma}

\begin{proof}
This follows from applying \cite[Theorem~5.1.2]{ColemanMazur} and the subsequent remark to a certain pseudo-representation defined by using the odd pseudo-character $T_\CD$ and the formulas in \cite[p.71]{ColemanMazur}. 
See Theorem~\ref{thm:global realization and generic ss} in Appendix~\ref{app:global realization}. 
Also, see \cite[Remark~3.1.2]{JohanssonNewton}.
\end{proof}

\begin{proof}[Proof of Theorem~\ref{thm:main}]
\medskip
\noindent\textbf{The general case: under the Sen-polynomial hypothesis.}
To simplify the notation, we will denote $\Theta(U_p), \Theta(U_\fp)$ by $U_p$ and $U_\fp$ respectively. 
We are going to apply Theorem~\ref{thm:DL-general}. 
Firstly, by Proposition~\ref{prop:power boundedness}, $h^*U_\fp$ and $h^*U_p$ are bounded functions on $\CD^*$, thus can be extended to $\CD$ by \cite[Lemma~2.1]{DiaoLiu}, and the same holds for $h^*\alpha_\fp$. 
We will denote the extended functions by the same notation. 
Secondly, by \eqref{eq:Up and Ufp}, since $U_p$ is an invertible function on $\CE$, so is $U_\fp$ and their pullback to $\CD^*$. 
Since $\nu_\fq^\un(x_\fp)\in\CO(\CW)^\times$, their pullbacks to $\CD^*$ are also invertible. 
Hence $h^*\alpha_\fp$ is also invertible on $\CD^*$. 
Thirdly, we note that $\nu_{\fp, \CD}:=(h')^*\nu_\fp^\un$ is a one-dimensional $p$-adic family of representations of $G_{F_\fp}$ on $\CD$ extending $h^*\nu_\fp^\un$. 

For any $\fp|p$, consider the quadruple $(\CD, h^*\alpha_\fp, V_{\CD}|_{G_{F_\fp}}\otimes \nu_{\fp, \CD}^{-1}, \tau_{0, \fp})$. 
The condition~\ref{cond:modified Liu (1)} in Theorem~\ref{thm:modified-liu} holds for $x\in \CD^*\backslash E_{\mathrm{exc}}$ by Theorem~\ref{thm:pointwise trianguline} and Corollary~\ref{cor:derived bc of herr cpx}, hence it holds for all $x\in \CD^*$ by Theorem~\ref{thm:herr-main}. 
By Proposition~\ref{prop:Sen poly} and the Zariski density of $\CD^*\backslash E_{\mathrm{exc}}$ in $\CD^*$, the Sen polynomial of $V_{\CD}|_{G_{F_\fp}}\otimes \nu_{\fp, \CD}^{-1}$ is given by the pullback of $(TQ_{\fp}(T)_\tau)_{\tau\in I_\fp}$. 
So the condition~\ref{cond:modified Liu (2)} in Theorem~\ref{thm:modified-liu} follows from our assumption. 
Therefore, we can apply Theorem~\ref{thm:DL-general} to $(\CD, h^*\alpha_\fp, V_{\CD}|_{G_{F_\fp}}\otimes \nu_{\fp, \CD}^{-1}, \tau_{0, \fp})$ and get $(h^*\alpha_\fp)(0)\neq0$ for all $\fp$. 

Note that $(h')^*((\nu_\fq^\un(x_\fp)))\in\CO(\CD)^\times$, hence $h^*(U_\fp)(0)\neq0$. 
Similarly, $(h')^*\left(\nu^\un(p\prod_{\fp|p}x_\fp^{-e_\fp})\right)\in\CO(\CD)^\times$, thus $h^*(U_p)(0)\neq0$ by \eqref{eq:Up and Ufp}, which give us morphisms 
$h^*U_p^{-1}\colon  \CD\to \BA^1_{\rig}$ and $(h', h^*U_p^{-1})\colon  \CD\to \CW\times \BA^1_\rig$. 
After being restricted to $\CD^*$, this is the composite of $h$ with 
\[
(\wt, U_{p}^{-1})\colon  \CE\to \CW\times\BA^1_\rig, 
\]
which factors through the spectral variety $\CZ\hookrightarrow \CW\times \BA^1_{\rig}$. 
Since it is a Zariski closed subvariety, the morphism $(h', h^*U_p^{-1})$ also factors through $\CZ$. 
So we obtain the following commutative diagram with solid arrows. 
\[
\begin{tikzcd}[column sep=large,row sep=large]
\CD^* \arrow[dd] \arrow[r, "h"]                              & \CE \arrow[d, "{(\wt,U_p^{-1})}" description] \arrow[dd, bend left=45, "\wt"{pos=0.8, swap}] \arrow[rd] &                                         \\
                                                                                                & \CZ \arrow[r, hook]                                                                      & \CW\times\BA^1_\rig \arrow[ld, "\pr_1"] \\
\CD \arrow[r, "h'"] \arrow[ru, "{(h',h^*U_p^{-1})}" description] \arrow[ruu, dashed, bend left] & \CW                                                                                      &                                        
\end{tikzcd}
\]
Since the morphism $(\wt, U_p^{-1})\colon \CE \to\CZ$ is finite, by \cite[Lemma~6.5]{Bellovin}, there is always a dotted arrow that fits the diagram, which gives the desired extension of $h$ compatible with $\wt\circ h$. 

\medskip
\noindent\textbf{The exceptional case: when there is a unique prime above $p$.}
We may assume that for any $\tau\in I_\fp=I$, there exists some $k_\tau\in\BZ_{\leq0}$, such that $h^*(Q_\fp(k_\tau)_{\tau})=0\in \CO(\CD^*)$. 
Then for $K_\tau\in\CO(\CW)$, $(h')^*K_\tau|_{\CD^*}=h^*\wt^*K_\tau\equiv k_\tau$. 
By Zariski density of $\CD^*$ in $\CD$, $(h')^*K_\tau\equiv k_\tau$, so $h'$ factors through $\CW_\kappa$ for $\kappa=(k_\tau)_\tau\in\BZ^I$. 
By shrinking the disc $\CD$, we may assume that $h'$ factors through some connected affinoid subdomain $\CV$ of $\CW_{\kappa}$. 
Then we have a commutative diagram 
\[
\begin{tikzcd}
\CD^* \arrow[dd] \arrow[rr, "h", bend left] \arrow[r] \arrow[rd, "f" description]                                                     & \CE\times_{\CW}\CV \arrow[d] \arrow[r, "\iota_{\CE}"', hook] & \CE \arrow[d] \arrow[dd, "\wt", bend left] \\
                                                                                                                                      & \CZ_{\CV} \arrow[d] \arrow[r, hook]                          & \CZ \arrow[d]                              \\
\CD \arrow[r] \arrow[rr, "h'"', bend right] \arrow[ru, "\widetilde{h}_1"', dashed] \arrow[ruu, "\widetilde{h}_2" description, dashed] & \CV \arrow[r, "\iota", hook]                                 & \CW                                       
\end{tikzcd}
\]
with solid arrows, where the right rectangles are Cartesian. 
By Theorem~\ref{thm:spectral} and connectedness of $\CD^*$, the morphism $f$ factors through some $\CZ_{\CV, i}$, which is finite over $\CV$. 
So we can apply \cite[Lemma~6.5]{Bellovin} twice to get the two dotted arrows $\widetilde{h_1}, \widetilde{h_2}$ that fit the diagram. 
The desired extension is given by $\iota_{\CE}\circ\widetilde{h_2}$. 
\end{proof}

\begin{remark}\label{rmk:on main thm}
\begin{enumerate}
\item
Let $x_0:=h'(0)$. 
If for every $\fp\mid p$, there exists $\tau_{0, \fp}\in I_\fp$, such that $K_{\tau_{0, \fp}}(x_0)\notin\BZ_{\leq0}$, then the non-vanishing assumption in the theorem is automatically satisfied.
Thus the theorem gives a properness result away from the non-positive integral Sen weights.
\item 
The non-vanishing assumption is also necessary for $\CD^*$ to be the $\tau_{0, \fp}$-finite slope subspace of $(h^*\alpha_\fp, V_{\CD}|_{G_{F_\fp}}\otimes \nu_{\fp, \CD}^{-1}, \tau_{0, \fp})$ for some $\tau_{0, \fp}$. 
Indeed, this follows from the first condition in the definition of the $\tau_{0}$-finite slope subspace.
\item Theorem~\ref{thm:main} and \cite[Theorem~1.1]{Hattori} have overlapping ranges of applicability. 
To see this, suppose that the hypotheses of Hattori's theorem on $F$ and $p$ are satisfied, and that $x_0=h'(0)$ is a $1$-integral weight in his sense (with the additional
$1$-doubly even assumption when $p=2$).
Write the algebraic part of $x_0$ as $((\nu_\tau)_{\tau\in I},w)\in(\BZ^I\oplus\BZ)$ and put $k_\tau=2\nu_\tau+w$.
If, for every $\fp\mid p$, there exists $\tau_{0, \fp}\in I_\fp$ such that $k_{\tau_{0, \fp}}\leq 0$, then $K_{\tau_{0, \fp}}(x_0)=1-k_{\tau_{0, \fp}}\notin\mathbb Z_{\leq0}$.
Thus both properness results apply in this situation.
\end{enumerate}
\end{remark}


\appendix
\section{Density of cohomological weights in the weight space}\label{app:coh wts}

\subsection{The setup and goal}

Recall that $E$ is a finite extension of $\BQ_p$ inside $\overline{\BQ_p}$ which splits the totally real field $F$, and the weight space $\CW$ is the rigid analytic $E$-space parameterizing continuous characters of
\[
H:=\BT(\BZ_p)\times\BZ_p^\times=\prod_{\fp|p}\CO_{F_\fp}^\times\times \BZ_p^\times. 
\]
Let $X:=\BZ^I\oplus\BZ\subset \CW(E)$ be the set of algebraic weights. 
For simplicity, if $\lambda=((\nu_\tau)_{\tau\in I},w)\in X$, we write $\chi_\lambda$ for the corresponding character of $H$. 
We have defined the subset of cohomological weights as
\[
X_{\coh}:=\left\{((\nu_\tau)_\tau,w)\in X:\nu_\tau\le0,\quad 2\nu_\tau+w\ge2\ \text{for every }\tau\in I\right\}\subset X, 
\]
and for $M\in\BR$, we have
\[
X_{\coh}(M):=\left\{((\nu_\sigma)_\sigma,w)\in X_{\coh}:
2\nu_\sigma+w>M\quad\text{for every }\sigma\in I
\right\}\subset X_{\coh}.
\]
The goal of this appendix is to prove the following theorem. 
\begin{theorem}\label{thm:density of coh wts, app}
For every $M\in\BR$ and every $x\in X_{\coh}$, there is a fundamental system of affinoid neighbourhoods $V_n$ of $\chi_x$ in $\CW$, such that $X_{\coh}(M)\cap V_n$ is Zariski dense in $V_n$ for every $n$.
\end{theorem}

\subsection{Components and coordinates}

Let $\Delta=H_{\tors}$ be the subgroup of $H$ consisting of torsion elements.
Then $H/\Delta$ is a finite free $\BZ_p$-module. 
To write down a coordinate of $\CW$, we need to choose a basis for it. 

For every prime $\fp\mid p$, put $d_\fp=[F_\fp:\BQ_p]=e_\fp f_\fp$. 
Let $\Delta_\fp=(\CO_{F_\fp}^{\times})_{\tors}$.
The quotient $\CO_{F_\fp}^{\times}/\Delta_\fp$ is a free $\BZ_p$-module of rank $d_\fp$. 
Since the natural map
\[
1+\fp\CO_{F_\fp}
\longrightarrow
\CO_{F_\fp}^{\times}/\Delta_\fp
\]
is surjective, we may therefore choose
\[
\gamma_{\fp,1},\ldots,\gamma_{\fp,d_\fp}
\in1+\fp\CO_{F_\fp}
\]
whose classes form a $\BZ_p$-basis of the quotient.
Through the natural inclusion $\CO_{F_\fp}^\times\to H$, we regard $\gamma_{\fp,i}\in H$ which equals
$\gamma_{\fp,i}$ in the $\fp$-factor and $1$ in every other factor.
Let $\Delta_0=(\BZ_p^\times)_{\tors}$.
Choose $\gamma_0\in 1+2p\BZ_p$ whose class is a $\BZ_p$-basis of $\BZ_p^\times/\Delta_0$, and we also regard $\gamma_0\in H$. 
Let $\Gamma$ be the topological subgroup of $H$ generated by 
\[
\gamma_0,
\qquad
\gamma_{\fp,i}\quad
(\fp\mid p,\ 1\le i\le d_\fp). 
\]
Then $\Gamma\simeq H/\Delta\simeq \BZ_p^{g+1}$, and these elements form a basis of $\Gamma$ as a $\BZ_p$-module. 
Fix an ordering and write them as
\[
\gamma_1,\ldots,\gamma_d, 
\]
where $d=g+1$. The connected components $\CW_{\epsilon}$ of $\CW$ are indexed by the characters
\[
\varepsilon\colon \Delta\longrightarrow E^\times.
\]
We denote the neutral component consisting of characters trivial on $\Delta$ by $\CW^0$. 
Any point $\chi\in\CW^0$ can be uniquely determined by $\chi(\gamma_j)-1$ for $1\leq j\leq d$, so we define coordinates $T_1,\ldots,T_d$ on $\CW^0$ by
\[
T_j(\chi):=\chi(\gamma_j)-1,
\]
for $\chi\in \CW^0$. 
Then $\CW^0$ is isomorphic to the $d$-dimensional open unit polydisc $\{|T_j|<1\}$. 
Since $H$ is a commutative group, there is naturally a structure of commutative rigid analytic group on $\CW$ and $\CW^0$. 
We denote 
\[
\mu\colon  \CW\times \CW\to \CW
\]
to be the multiplication map. 
For $\chi, \chi'\in\CW^0$, we have $\mu(\chi, \chi')\in \CW^0$ and 
\begin{equation}\label{eq:coord of multi}
T_j(\mu(\chi, \chi'))+1=(T_j(\chi)+1)(T_j(\chi')+1). 
\end{equation}

\subsection{An invertible matrix}

Let $J=\{0\}\sqcup I$.
Consider the standard basis of $X$: 
\[
\mathbf e_0=((0)_\tau,1),
\qquad
\mathbf e_\sigma=
\bigl((\delta_{\sigma,\tau})_{\tau\in I},0\bigr)
\quad(\sigma\in I).
\]
Choose an integer $q_0>0$ divisible by the exponent of $\Delta$, and
put
\[
\theta_i=\chi_{q_0\mathbf e_i},
\qquad i\in J.
\]
Then every $\theta_i$ belongs to $\CW^0$ and $\theta_i(\gamma_j)\in 1+\varpi\CO_E$, where $\varpi\in E$ is a uniformizer.  
Define
\[
a^{(0)}_{ij}=\log\bigl(\theta_i(\gamma_j)\bigr)\in E,
\qquad A^{(0)}=(a^{(0)}_{ij})_{i\in J,\ 1\le j\le d}\in M_{d}(E), 
\]
where $\log$ is the usual $p$-adic logarithm for $E$. 

\begin{proposition}\label{prop:A0invertible}
The matrix $A^{(0)}$ is invertible.
\end{proposition}

\begin{proof}
For each $\fp|p$, consider the matrix
\[
M_\fp=
\left(
\log\bigl(\sigma(\gamma_{\fp,i})\bigr)
\right)_{\sigma\in I_\fp,\ 1\le i\le d_\fp}\in M_{d_\fp}(E).
\]
Note that for $\tau\in I_\fp$, the induced morphism 
\[
\CO_F\otimes_{\BZ}\BZ_p\simeq \prod_{\fp|p}\CO_{F_{\fp}}\to \CO_E
\]
factors through the natural projection $\prod_{\fp|p}\CO_{F_{\fp}}\to \CO_{F_\fp}$. 
Therefore, up to permutations of rows and columns, the matrix $A^{(0)}$ has the block form
\[
A^{(0)}=q_0
\begin{pmatrix}
M_{\fp_1}&&&0\\
&M_{\fp_2}&&0\\
&&\ddots&\vdots\\
0&0&\cdots&\log(\gamma_0)
\end{pmatrix}.
\]
Since $\gamma_0\in 1+2p\BZ_p, \log(\gamma_0)\neq0$, it remains to show that $M_{\fp}$ are invertible. 
Indeed, for any $\fp\mid p$, choose an integer $m_\fp>\frac{e_\fp}{p-1}$. 
Then the $p$-adic logarithm $\log\colon  1+\fp\CO_{F_\fp}\to F_\fp$ restricts to an isomorphism
\[
\log\colon 1+\fp^{m_\fp}\CO_{F_\fp}\xrightarrow{\sim}\fp^{m_\fp}\CO_{F_\fp}
\]
with inverse given by the $p$-adic exponential. 
Choose $N\in\BN$, such that for any $1\leq i\leq d_\fp, \gamma_{\fp, i}^N\in 1+\fp^{m_{\fp}}\CO_{F_{\fp}}$, then they span a $\BZ_p$-lattice of rank $d_\fp$. 
Thus $\{\log(\gamma_{\fp, i}^N)\}_{1\leq i\leq d_\fp}$ also span a $\BZ_p$-lattice of rank $d_\fp$ in $\fp^{m_{\fp}}\CO_{F_{\fp}}$, which implies that $\{\log(\gamma_{\fp, i})\}_{1\leq i\leq d_\fp}$ forms a basis of $F_{\fp}$ over $\BQ_p$. 
Finally, for any $\sigma\in I_\fp, \log\bigl(\sigma(\gamma_{\fp,i})\bigr)=\sigma\bigl(\log(\gamma_{\fp,i})\bigr)$, so we conclude by standard algebraic number theory. 
\end{proof}

\subsection{One-dimensional subgroups}

For
\[
a=(a_1,\ldots,a_d)\in E^d, \qquad a\ne0,
\]
choose $j_0$ with $a_{j_0}\ne0$. Put
\[
F_j=a_{j_0}\log(1+T_j)-a_j\log(1+T_{j_0})\in\CO(\CW^0), 
\]
and
\[
Z(a)=V\left(F_j:\ j\ne j_0\right)\subset \CW^0. 
\]

\begin{lemma}\label{lem:curve-a}
The space $Z(a)$ is a smooth one-dimensional closed rigid analytic subgroup of $\CW^0$.
Its connected component $L(a)$ containing the identity is a connected smooth rigid analytic curve. 
\end{lemma}

\begin{proof}
By (\ref{eq:coord of multi}) and the standard property of $\log$, each function $\log(1+T_j)$ defines a group homomorphism from $\CW^0$ to $\BG_a$, so does $F_j$.
Hence the group structure on $Z(a)$ follows. 
For $j\ne j_0$, 
\[
\frac{\partial F_j}{\partial T_j}
=
\frac{a_{j_0}}{1+T_j},
\qquad
\frac{\partial F_j}{\partial T_k}=0
\quad(k\ne j,j_0).
\]
The Jacobian minor in the variables $T_j$, $j\ne j_0$, has
determinant
\[
a_{j_0}^{d-1}\prod_{j\ne j_0}(1+T_j)^{-1},
\]
which is nowhere zero.
Thus the equations cut out a smooth one-dimensional rigid space.
The last statement is clear. 
\end{proof}

For $\theta\in\CW^0(E)$, put
\[
a(\theta)=\bigl(\log(\theta(\gamma_1)),\ldots,\log(\theta(\gamma_d))\bigr).
\]
If $a(\theta)\ne0$, then $\theta\in Z(a(\theta))$. 
By the continuity of $\theta$, we have
\[
\theta^{p^s}\longrightarrow1
\qquad(s\to\infty). 
\]
So for all sufficiently large $s$, $\theta^{p^s}\in L(a(\theta))(E)$. 
Moreover, 
\[
a(\theta^{p^s})=p^sa(\theta), \qquad
Z(p^sa)=Z(a), \qquad L(p^sa)=L(a). 
\]
Recall that for any $\kappa=(k_\tau)_{\tau\in I}\in\BZ^I$, $\fp\mid p$ and $\tau\in I_\fp$, we have defined $K_\tau=1-2\wt_\tau(\nu_\fp^\un)-\wt(w^\un)\in \CO(\CW)$ in Section~\ref{sec:spectral variety}. 

\begin{proposition}\label{prop:reg cl im}
The differentials $(dK_\tau)_{\tau\in I}$ are linearly independent at any closed point of $\CW$. 
In particular, $\CW_\kappa$ is a smooth curve, and the natural closed immersion $\CW_{\kappa}\to \CW$ is regular. 
\end{proposition}

\begin{proof}
It is easy to see that $\wt_\tau(\nu_\fp^\un)$ and $\wt(w^\un)$ are additive as morphisms of rigid analytic groups from $\CW$ to $\BG_a$. 
So it suffices to show that the differentials $-dK_\tau=d(2\wt_\tau(\nu_\fp^\un)+\wt(w^\un))$ are linearly independent on the identity component $\CW^0$. 
In fact, with the basis $\{\gamma_i\}_{i=1}^d$ chosen in the previous section, for $\gamma_j\in \CO_{F_\fp}$ and $\chi\in\CW^0$, we have
\[
\log (T_j+1)=\log \chi(\gamma_j)=\sum_{\tau\in I_\fp}\wt_\tau(\nu_\fp^\un)(\chi)\log\tau(\gamma_{j}). 
\]
Note that the coefficients $\log\tau(\gamma_{j})$ appeared above are the elements of the matrix $A:=A^{(0)}/q_0$.
Thus under the basis $dT_j$, after choosing an order for $I$, the differentials $(dK_\tau)_{\tau\in I}$ are the rows of
\[
\begin{pmatrix}
2 & 0 & \cdots & 0 & 1 \\
0 & 2 & \cdots & 0 & 1 \\
\vdots & \vdots & \ddots & \vdots & \vdots \\
0 & 0 & \cdots & 2 & 1
\end{pmatrix}_{(d-1) \times d}
\times A^{-1} \times 
\begin{pmatrix}
\frac{dT_1}{1+T_1} \\
\vdots \\
\frac{dT_d}{1+T_d}
\end{pmatrix}. 
\]
Since the rank of the first matrix is $d-1=g=|I|$, we are done. 
Alternatively, one can take $\theta=((1_\tau)_{\tau\in I}, , -2)\in X$, $a=a(\theta)$ and use the calculation in Lemma~\ref{lem:curve-a}. 
\end{proof}

\begin{proposition}\label{prop:etale-product}
Let $(\lambda_i)_{i\in J}$ be a $\BZ$-basis of
$X=\BZ^I\oplus\BZ$. Write
\[
\lambda_i=\sum_{k\in J}c_{ik}\mathbf e_k,
\qquad
C=(c_{ik})\in\GL_d(\BZ).
\]
Put
\[
\theta_i^\lambda=\chi_{q_0\lambda_i}, \qquad
b_i:=(b_{ij})_{1\leq j\leq d}:=a(\theta_i^\lambda), \qquad
B=(b_{ij})_{i\in J,\ 1\le j\le d}.
\]
Then $B=CA^{(0)}$ and $B$ is invertible. 
Put $L_i=L(b_i)$. Then the multiplication map
\[
\mu_C\colon \prod_{i\in J}L_i\longrightarrow\CW^0, \qquad
(h_i)_i\longmapsto\prod_i h_i,
\]
is \'etale at the identity.
Consequently, there are affinoid connected neighbourhoods $V_i\subset L_i$ of the identity, such that
\[
\mu_C\colon 
V:=\prod_iV_i\longrightarrow\CW^0
\]
is \'etale and $U_C^0:=\mu_C(V)$ is an admissible open neighbourhood of the identity.
\end{proposition}

\begin{proof}
The equality $B=CA^{(0)}$ follows from the way we put $X$ into $\CW(E)$ and additivity of the logarithm.
Thus $B$ is invertible by Proposition~\ref{prop:A0invertible}.
Fix $i$, and choose $j_i$ such that $b_{i,j_i}\ne0$. 
To compute the tangent space of $L_i$ at the identity, we need to solve the following equations
\[
b_{i,j_i}v_j-b_{ij}v_{j_i}=0
\qquad(j\ne j_i)
\]
about $v_i$, which give us
\[
T_1L_i=E\xi_i, \qquad
\xi_i=\sum_{j=1}^db_{ij}\frac{\partial}{\partial T_j}.
\]
The tangent map of multiplication is the summation map
\[
(T\mu_C)_1\colon  \bigoplus_iT_1L_i\longrightarrow T_1\CW^0,
\qquad (v_i)_i\longmapsto\sum_i v_i.
\]
The coordinate vectors of the $\xi_i$ are the rows of $B$, so they form a basis.
Thus $(T\mu_C)_{1}$ is an isomorphism, and $\mu_C$ is \'etale at the identity.
\end{proof}

\subsection{Density of cohomological weights}

We are going to prove Theorem~\ref{thm:density of coh wts, app}. 
Fix $x=((\nu_\sigma^\circ)_\sigma,w^\circ)\in X_{\coh}$. 
The component containing $\chi_x$ is given by the translation of $\CW^0$, i.e. $\CW_x=\chi_x\CW^0$.
We define the coordinates $\widetilde{T_j}$ on $\CW_x$ centered at $x$ by
\[
1+\widetilde T_j(\chi)=\frac{\chi(\gamma_j)}{\chi_x(\gamma_j)}.
\]
Let $\varpi$ be a uniformizer of $E$. 
For $n\ge1$, put
\[
V_n=\left\{\chi\in\CW_x:|\widetilde T_j(\chi)|\le|\varpi|^n\ \text{for every }j\right\}.
\]
Then $V_n\simeq
\Sp E\langle S_1,\ldots,S_d\rangle$, where $S_j=\frac{\widetilde T_j}{\varpi^n}$,
and $V_n$ form a fundamental system of affinoid neighbourhoods of $\chi_x$.

\begin{lemma}\label{lem:binomial-interpolation}
There exists an integer $c_E\ge0$, depending only on $E$, such that for every $n\ge1$ and every $u\in1+\varpi^{n+c_E}\CO_E$, the binomial series
\[
\mathcal B_u(Z):=\sum_{r\ge0}\binom Zr(u-1)^r
\]
converges in $E\langle Z\rangle$, and satisfies
\[
\mathcal B_u(Z)-1\in\varpi^n\CO_E\langle Z\rangle, 
\]
and
\begin{equation}\label{eq:Bu(m)}
\mathcal B_u(m)=u^m \qquad(m\in\BZ_{\ge0}) 
\end{equation}
Moreover,
\begin{equation}
\log\mathcal B_u(Z)=Z\log u
\quad\text{in }E\langle Z\rangle. 
\label{eq:binomial-logarithm}
\end{equation}
\end{lemma}

\begin{proof}
Put $\alpha=v_p(\varpi)>0$. 
By Legendre's estimate, $v_p(r!)\le\frac{r}{p-1}$.
Choose $c_E$ with
\[
c_E\alpha\ge\frac1{p-1}.
\]
If $u-1\in\varpi^{n+c_E}\CO_E$, then the $r$-th nonconstant term
has valuation at least
\[
r(n+c_E)\alpha-v_p(r!)\ge rn\alpha.
\]
Thus the series converges in the Tate algebra and every non-constant term is divisible by $\varpi^n$.
So the logarithm on the left side of \eqref{eq:binomial-logarithm} is defined in the Tate algebra.
For a nonnegative integer $m$, the series terminates and the usual binomial theorem gives $\mathcal B_u(m)=u^m$. 
Both sides of \eqref{eq:binomial-logarithm} agree at the Zariski dense subset $\BZ_{\geq0}\subset \Sp(E\langle Z\rangle)$, so we are done.
\end{proof}

Fix an integer $h>2$. Let
\begin{equation}
\lambda_0=\mathbf e_0, \qquad
\lambda_\sigma=h\mathbf e_0-\mathbf e_\sigma \quad(\sigma\in I).
\label{eq:positive-directions}
\end{equation}
Then they form a basis of $X$.
Let $C_+\in\GL_d(\BZ)$ be the corresponding change-of-basis matrix. Put
\[
\theta_i^+=\chi_{q_0\lambda_i},
\qquad
b_i^{(0)}=a(\theta_i^+),
\]
and let $B^{(0)}$ be the matrix with rows $b_i^{(0)}$. 
By Proposition~\ref{prop:etale-product},
\begin{equation}
B^{(0)}=C_+A^{(0)}\in\GL_d(E).
\label{eq:positive-basis-matrix}
\end{equation}
Put $L_i=L(b_i^{(0)})$ for $i\in J$, and let
\[
\mu_+\colon \prod_{i\in J}L_i\longrightarrow\CW^0
\]
be the multiplication morphism. 
By Proposition~\ref{prop:etale-product}, $\mu_+$ is \'etale at the identity.
Now we parameterize the algebraic weights in $L_i$. 

\begin{corollary}\label{cor:parameterize weights}
For any $n\geq1$, there exists $s_n$, such that if we put
\[
q_n=q_0p^{s_n}, \qquad \eta_{i,n}=\chi_{q_n\lambda_i}, \quad u_{ij,n}=\eta_{i,n}(\gamma_j)=\chi_{\lambda_i}(\gamma_j)^{q_n}, 
\]
then for each $i\in J$, there exists a morphism 
\[
\varphi_{i,n}\colon \CD:=\Sp E\langle Z\rangle\longrightarrow L_i
\]
such that $\varphi_{i, n}(m)=\eta_{i, n}^m$ for every $m\in\BZ_{\geq0}$. 
\end{corollary}

\begin{proof}
Choose $s_n$ sufficiently large such that $u_{ij, n}\in 1+\varpi^{n+c_E}\CO_E$ for any $i, j$. 
We define 
\[
\varphi_{i, n}\colon  \CD\to \CW^0
\]
by $\varphi_{i,n}^*(1+T_j)=\mathcal B_{u_{ij,n}}(Z)$. 
Put $b_{ij,n}=\log u_{ij,n}=p^{s_n}b_{ij}^{(0)}$.
Then \eqref{eq:binomial-logarithm} gives
\begin{equation}
\varphi_{i,n}^*\log(1+T_j)
=b_{ij,n}Z
=p^{s_n}b_{ij}^{(0)}Z.
\label{eq:curve-linearization}
\end{equation}
Hence all defining equations of $Z(b_i^{(0)})$ pull back to zero. 
Note that $\CD$ is connected and $\varphi_{i,n}(0)\in \CW$ is the identity, so $\varphi_{i, n}$ factors through $L_i$. 
Finally, for every $m\in\BZ_{\ge0}$,
\[
1+T_j(\varphi_{i, n}(m))=\CB_{u_{ij, n}}(m)=u_{ij, n}^m=\eta_{i, n}^m(\gamma_j), 
\]
thus $\varphi_{i, n}(m)=\eta_{i, n}^m$. 
\end{proof}

\begin{proof}[Proof of Theorem~\ref{thm:density of coh wts, app}]
Fix $n\geq1$. We will show that $X_{\coh}(M)\cap V_n$ is Zariski dense in $V_n$. 
Let $s_n, \varphi_{i, n}$ be as in Corollary~\ref{cor:parameterize weights}, 
\[
\Phi_n=\prod_{i\in J}\varphi_{i,n}\colon 
\CD^d\longrightarrow\prod_{i\in J}L_i,
\]
and
\begin{equation}
\Psi_n=\tau_{\chi_x}\circ\mu_+\circ\Phi_n\colon  \CD^d\longrightarrow\CW_x,
\label{eq:product-interpolation}
\end{equation}
where $\tau_{\chi_x}(h)=\chi_xh$ is the translation map. 
In the coordinates $\tilde{T_j}$,
\begin{equation}\label{eq:Psi_n}
\Psi_n^*(1+\widetilde T_j)=\prod_{i\in J}\mathcal B_{u_{ij,n}}(Z_i).
\end{equation}
Every factor on the right hand side lies in $1+\varpi^n\CO_E\langle Z_i\rangle$, so $\Psi_n(\CD^d)\subset V_n$.
Taking the differential of \eqref{eq:Psi_n} at $Z=0$, we see that $(d\Psi_n)_0$ is represented by the matrix
\[
B_n=(\log u_{ij,n})=p^{s_n}B^{(0)},
\]
which is invertible by \eqref{eq:positive-basis-matrix}.
Thus $\Psi_n$ is \'etale at the origin. 
Therefore, there exists an admissible open neighbourhood $U_n\subset\CD^d$ of the origin such that $\Psi_n(U_n)\subset V_n$ is a nonempty admissible open subset.
For $m=((m_\sigma)_{\sigma\in I}, m_0)\in \BZ_{\geq0}^d$, Corollary~\ref{cor:parameterize weights} gives
\[
\Psi_n(m)=\chi_x\prod_{i\in J}\eta_{i,n}^{m_i}.
\]
In additive algebraic weight notation,
\begin{equation}
\Psi_n(m)=x+q_n\left(m_0\lambda_0+\sum_{\sigma\in I}m_\sigma\lambda_\sigma\right)\in X.
\end{equation}
Using \eqref{eq:positive-directions}, one can compute that
\begin{equation*}
\begin{aligned}
\nu_\sigma(\Psi_n(m))&=\nu_\sigma^\circ-q_nm_\sigma, \\
2\nu_\sigma(\Psi_n(m))+w(\Psi_n(m))
&={}2\nu_\sigma^\circ+w^\circ+q_n\left(m_0+(h-2)m_\sigma
+h\sum_{\tau\ne\sigma}m_\tau
\right).
\end{aligned}
\end{equation*}
So there exists $N=N(n,M)$, such that for $D_N:=\BZ_{\geq N}^d\subset\CD^d$, we have
\begin{equation}
\Psi_n(D_N)\subset X_{\coh}(M)\cap V_n.
\label{eq:dense-algebraic-family}
\end{equation}
Let $f\in\CO(V_n)$ vanish on $X_{\coh}(M)\cap V_n$.
By \eqref{eq:dense-algebraic-family}, $\Psi_n^*f$ vanishes on $D_N$, which is Zariski dense in $\CD^d$.
Thus $\Psi_n^*f=0$. 
In particular, $f$ vanishes on the nonempty admissible open subset $\Psi_n(U_n)\subset V_n$. 
Since $V_n$ is an irreducible affinoid polydisc, this implies $f=0$ on $V_n$. 
Hence $X_{\coh}(M)\cap V_n$ is Zariski dense in $V_n$, and the theorem follows.
\end{proof} 
\section{Pseudo-characters and families of Galois representations}\label{app:global realization}

In this appendix, we record a result on the global realization of families of two-dimensional odd pseudo-characters as families of representations over smooth curves.

\subsection{Two-dimensional odd pseudo-character}

\begin{definition}
Let $A$ be a commutative $\BQ$-algebra and let $G$ be a group. 
A \emph{two-dimensional $A$-valued pseudo-character} is a function
\[
T\colon G\longrightarrow A
\]
such that $T(1)=2, T(gh)=T(hg)$ and, for all $g_1,g_2,g_3\in G$,
\begin{align*}
0={}&T(g_1)T(g_2)T(g_3)
   -T(g_1g_2)T(g_3)-T(g_1g_3)T(g_2)-T(g_2g_3)T(g_1)\\
 &+T(g_1g_2g_3)+T(g_1g_3g_2).
\end{align*}
Moreover, if $G$ is a topological group, $A$ is a topological ring, and $T$ is continuous, then we say $T$ is a \emph{continuous two-dimensional A-valued pseudo-character}. 
\end{definition}

This is the $d=2$ case of Taylor's trace identities \cite[\S 1.1]{Taylor91}. 
Over an algebraically closed field of characteristic zero, Taylor's theorem \cite[Theorem~1]{Taylor91} says that such a pseudo-character is the trace of a unique semisimple two-dimensional representation, up to isomorphism. 

Fix a rigid analytic space $X$ over a finite extension $E$ of $\BQ_p$ and a topological group $G$. 

\begin{definition}
A \emph{continuous two-dimensional pseudo-character on $X$} is a map
\[
T\colon G\longrightarrow \CO(X)
\]
whose restriction $T_U\colon G\longrightarrow \CO(U)$ is a continuous two-dimensional pseudo-character for every affinoid admissible open $U\subset X$.
\end{definition}

For any closed point $x\in X$, write $T_x$ for the specialization to $k(x)$.  
After extending scalars to an algebraic closure
$\overline{k(x)}$, Taylor's theorem associates to $T_x$ a unique
semisimple representation
\[
\rho^{\ss}_{T,x}\colon G\longrightarrow \GL_2(\overline{k(x)})
\]
with trace $T_x$.

\begin{definition}
Let $T$ be a two-dimensional pseudo-character on $X$. 
An \emph{odd involution of $T$} is an element $c\in G, c^2=1$, such that $T(c)=0$.
If there exists an odd involution of $T$, then we say $T$ is \emph{odd}. 
\end{definition}

At every closed point $x\in X$, the semisimple representation associated to $T_x$ satisfies
\[
\tr(\rho^{\ss}_{T,x}(c))=0,\qquad \rho^{\ss}_{T,x}(c)^2=1.
\]
Since the residue field has characteristic zero, $c$ is diagonalizable with eigenvalues $+1$ and $-1$.

\subsection{Coleman--Mazur's pseudo-representation}

We now pass from an odd pseudo-character $T$ to Coleman--Mazur's pseudo-representation. 
In this subsection, we assume that $X$ is reduced, and let $c$ be an odd involution of $T$.
Define
\begin{equation}\label{eq:alpha-delta}
  \alpha(g)=\frac{T(g)+T(cg)}{2},
  \qquad
  \delta(g)=\frac{T(g)-T(cg)}{2},
\end{equation}
and
\begin{equation}\label{eq:xi}
  \xi(g,h)=\alpha(gh)-\alpha(g)\alpha(h).
\end{equation}

\begin{lemma}\label{lem:CM-pseudorep}
The triple
\[
  r=(\xi,\alpha,\delta)
\]
is a two-dimensional pseudo-representation in the sense of Coleman--Mazur, i.e.
\begin{align}
 \alpha(gh)&=\alpha(g)\alpha(h)+\xi(g,h),\label{eq:CM1a}\\
 \delta(gh)&=\delta(g)\delta(h)+\xi(h,g),\label{eq:CM1b}\\
 \xi(gh,k)&=\alpha(g)\xi(h,k)+\delta(h)\xi(g,k),\label{eq:CM2a}\\
 \xi(g,hk)&=\alpha(k)\xi(g,h)+\delta(h)\xi(g,k),\label{eq:CM2b}\\
 \xi(g,h)\xi(k,\ell)&=\xi(g,\ell)\xi(k,h),\label{eq:CM3}
\end{align}
and
\[
 \alpha(1)=\delta(1)=1,\qquad
 \alpha(c)=1,\qquad \delta(c)=-1,
\]
while $\xi(g,h)=0$ if $g$ or $h$ belongs to $\{1,c\}$.
Moreover
\[
  \alpha(g)+\delta(g)=T(g).
\]
\end{lemma}

\begin{proof}
The first displayed equality in the conclusion is the definition of
$\xi$, and the last equality is immediate from
\eqref{eq:alpha-delta}.  We prove the remaining identities pointwise.

Let $x$ be a closed point of $X$, extend the scalars to
$\overline{k(x)}$, and let
\[
\rho_{x}^{\ss}\colon G\longrightarrow \GL_2(\overline{k(x)})
\]
be the semisimple representation with trace $T_x$.  
Choose a basis in which
\[
\rho_{x}^{\ss}(c)=\begin{pmatrix}1&0\\0&-1\end{pmatrix}.
\]
Write
\[
\rho_{x}^{\ss}(g)=
\begin{pmatrix}
a(g)&b(g)\\ c(g)&d(g)
\end{pmatrix}.
\]
Then
\[
T_x(g)=a(g)+d(g),\qquad T_x(cg)=a(g)-d(g),
\]
so $\alpha_x(g)=a(g), \delta_x(g)=d(g)$.
Taking the upper-left entry in
$\rho_{x}^{\ss}(gh)=\rho_{x}^{\ss}(g)\rho_{x}^{\ss}(h)$ gives $\xi_x(g,h)=b(g)c(h)$.
The identities \eqref{eq:CM1b}--\eqref{eq:CM3} now follow immediately from matrix multiplication.  
The identities involving $1$ and $c$ are
also immediate.
Thus every difference between the two sides of each required identity is an analytic function on $X$ which vanishes at every closed point. 
Since $X$ is reduced, such a function is zero. 
Hence the identities hold globally.
\end{proof}

Define the coherent ideal sheaf $\CI_r=\bigl(\xi(g,h): \;g,h\in G\bigr)\subset \CO_X$. 
Then for a closed point $x\in X$, $x\notin V(\CI_r)$ if and only if the semisimple representation associated to $T_x$ is absolutely irreducible.

\subsection{From odd Pseudo-characters to families of representations}

Let $T$ still be a continuous two-dimensional odd pseudo-character of $G$ on a rigid space $X$ over some finite extension $E$ of $\BQ_p$. 

\begin{theorem}\label{thm:global realization and generic ss}
Suppose $X$ is a separated smooth rigid analytic curve over $E$. Then there exists:
\begin{itemize}
\item a family of two-dimensional $p$-adic representations $V_X$ of $G$ on $X$, 
\item a discrete subset $E_{\mathrm{exc}}\subset X$;
\end{itemize}
such that:
\begin{enumerate}
\item the pseudo-character associated to $V_X$ is $T$;
\item for every closed point $x\in X\setminus E_{\mathrm{exc}}$, the specialization $V_x$ is semisimple. 
\end{enumerate}
More precisely, for every connected component $C$ of $X$, if $\CI_r|_C\neq 0$, then $V(\CI_r|_C)$ is discrete.
We may take $E_{\mathrm{exc}}=\bigcup_{\CI_r|_C\neq0}V(\CI_r|_C)$.
In particular, since we are working over characteristic zero, for any $x\notin E_{\mathrm{exc}}$ and any finite extension $E'$ of $k(x)$, $V_x\otimes_{k(x)}E'$ is still semisimple. 
\end{theorem}

\begin{proof}
This follows from applying \cite[Theorem~5.1.2]{ColemanMazur} and the subsequent remark to $r=(\xi,\alpha,\delta)$ defined by Lemma~\ref{lem:CM-pseudorep}. 
Also, see \cite[Remark~3.1.2]{JohanssonNewton}. 
\end{proof} 

\medskip
\noindent\textbf{Competing interests.}
The author declares none.

\bibliographystyle{plain}
\bibliography{References}

\begin{thebibliography}{10}

\bibitem{Andreatta-Goren}
F.~Andreatta and E.~Z. Goren.
\newblock Geometry of hilbert modular varieties over totally ramified primes.
\newblock {\em Int. Math. Res. Not. IMRN}, 2003(33):1785--1835, 01 2003.

\bibitem{AIP}
Fabrizio Andreatta, Adrian Iovita, and Vincent Pilloni.
\newblock On overconvergent {Hilbert} modular cusp forms.
\newblock {\em Ast{\'e}risque}, 382:163--193, 2016.

\bibitem{Bellaiche}
Jo{\"e}l Bella{\"i}che.
\newblock {\em The Eigenbook: {Eigenvarieties}, Families of {Galois} Representations, {$p$}-Adic {$L$}-Functions}.
\newblock Springer, 2021.

\bibitem{BC}
Jo{\"e}l Bella{\"i}che and Ga{\"e}tan Chenevier.
\newblock {\em Families of {Galois} representations and {Selmer} groups}.
\newblock Number 324 in Ast{\'e}risque. Soci{\'e}t{\'e} Math{\'e}matique de France, 2009.

\bibitem{Bellovin15}
Rebecca Bellovin.
\newblock {$p$}-adic {Hodge} theory in rigid analytic families.
\newblock {\em Algebra Number Theory}, 9(2):371--433, 2015.

\bibitem{Bellovin}
Rebecca Bellovin.
\newblock Cohomology of {$(\varphi,\Gamma)$}-modules over pseudorigid spaces.
\newblock {\em Int. Math. Res. Not. IMRN}, 2024(4):2999--3051, 2024.

\bibitem{Berger}
Laurent Berger.
\newblock Repr{\'e}sentations {$p$}-adiques et {\'e}quations diff{\'e}rentielles.
\newblock {\em Invent. Math.}, 148:219--284, 2002.

\bibitem{BergerColmez}
Laurent Berger and Pierre Colmez.
\newblock Familles de repr{\'e}sentations de de {Rham} et monodromie {$p$}-adique.
\newblock {\em Ast{\'e}risque}, 319:303--337, 2008.

\bibitem{Bij}
St{\'e}phane Bijakowski.
\newblock Classicit{\'e} de formes modulaires de {Hilbert}.
\newblock {\em Ast{\'e}risque}, 382:49--71, 2016.

\bibitem{Birkbeck}
Christopher Birkbeck.
\newblock The {Jacquet--Langlands} correspondence for overconvergent {Hilbert} modular forms.
\newblock {\em Int. J. Number Theory}, 15(3):479--504, 2019.

\bibitem{Blasius-Rogawski}
Don Blasius and Jonathan~D. Rogawski.
\newblock Motives for {Hilbert} modular forms.
\newblock {\em Invent. Math.}, 114(1):55--87, 1993.

\bibitem{BGR}
Siegfried Bosch, Ulrich G{\"u}ntzer, and Reinhold Remmert.
\newblock {\em Non-Archimedean Analysis: A Systematic Approach to Rigid Analytic Geometry}.
\newblock Springer-Verlag, Berlin, 1984.

\bibitem{Buzzard}
Kevin Buzzard.
\newblock Eigenvarieties.
\newblock In David Burns, Kevin Buzzard, and Jan Nekov{\'a}{\v r}, editors, {\em {$L$}-Functions and {Galois} Representations}, volume 320 of {\em London Mathematical Society Lecture Note Series}, pages 59--120. Cambridge University Press, 2007.

\bibitem{BuzzardCalegari}
Kevin Buzzard and Frank Calegari.
\newblock The {$2$}-adic eigencurve is proper.
\newblock {\em Doc. Math.}, Extra Volume:211--232, 2006.

\bibitem{Calegari}
Frank Calegari.
\newblock The {Coleman--Mazur} eigencurve is proper at integral weights.
\newblock {\em Algebra Number Theory}, 2:209--215, 2008.

\bibitem{Carayol}
Henri Carayol.
\newblock Sur les repr\'esentations $l$-adiques associ\'ees aux formes modulaires de {Hilbert}.
\newblock {\em Ann. Sci. {\'E}c. Norm. Sup{\'e}r.}, 4e s{\'e}rie, 19(3):409--468, 1986.

\bibitem{ChenevierGLn}
Ga{\"e}tan Chenevier.
\newblock Familles {$p$}-adiques de formes automorphes pour {${\rm GL}_n$}.
\newblock {\em J. Reine Angew. Math.}, 570:143--217, 2004.

\bibitem{ColemanMazur}
Robert~F. Coleman and Barry Mazur.
\newblock The eigencurve.
\newblock In A.~J. Scholl and R.~L. Taylor, editors, {\em Galois Representations in Arithmetic Algebraic Geometry}, volume 254 of {\em London Mathematical Society Lecture Note Series}, pages 1--113. Cambridge University Press, 1998.

\bibitem{DP}
Pierre Deligne and Georgios Pappas.
\newblock Singularit\'es des espaces de modules de {Hilbert,} en les caract\'eristiques divisant le discriminant.
\newblock {\em Compos. Math.}, 90(1):59--79, 1994.

\bibitem{DiaoLiu}
Hansheng Diao and Ruochuan Liu.
\newblock The eigencurve is proper.
\newblock {\em Duke Math. J.}, 165(7):1381--1395, 2016.

\bibitem{Fargues}
Laurent Fargues.
\newblock La filtration canonique des points de torsion des groupes $p$-divisibles.
\newblock {\em Ann. Sci. {\'E}c. Norm. Sup{\'e}r.}, 4e s{\'e}rie, 44(6):905--961, 2011.

\bibitem{Fontaine}
Jean-Marc Fontaine.
\newblock Sur certains types de repr{\'e}sentations {$p$}-adiques du groupe de {Galois} d'un corps local; construction d'un anneau de {Barsotti--Tate}.
\newblock {\em Ann. of Math. (2)}, 115(3):529--577, 1982.

\bibitem{Hall}
Jack Hall and David Rydh.
\newblock The {Hilbert} stack.
\newblock {\em Adv. Math.}, 253:194--233, 2014.

\bibitem{Hattori2}
Shin Hattori.
\newblock On canonical subgroups of {Hilbert{\textendash}Blumenthal} abelian varieties.
\newblock {\em J. Th\'eor. Nombres Bordeaux}, 30(2):355--391, 2018.

\bibitem{Hattori}
Shin Hattori.
\newblock On a properness of the {Hilbert} eigenvariety at integral weights: the case of quadratic residue fields.
\newblock {\em J. Inst. Math. Jussieu}, 22(6):2645--2716, 2023.

\bibitem{Hida}
Haruzo Hida.
\newblock Nearly ordinary {Hecke} algebras and {Galois} representations of several variables.
\newblock In {\em Algebraic Analysis, Geometry, and Number Theory}, pages 115--134. Johns Hopkins University Press, Baltimore, MD, 1989.

\bibitem{JohanssonNewton}
Christian Johansson and James Newton.
\newblock Parallel weight {$2$} points on {Hilbert} modular eigenvarieties and the parity conjecture.
\newblock {\em Forum Math. Sigma}, 7:e27, 2019.
\newblock 36 pages.

\bibitem{KedlayaLiu}
Kiran~S. Kedlaya and Ruochuan Liu.
\newblock On families of {$(\varphi,\Gamma)$}-modules.
\newblock {\em Algebra Number Theory}, 4(7):943--967, 2010.

\bibitem{KedlayaPottharst}
Kiran~S. Kedlaya and Jonathan Pottharst.
\newblock On categories of {$(\varphi,\Gamma)$}-modules.
\newblock In {\em Algebraic Geometry: Salt Lake City 2015}, volume 97.2 of {\em Proceedings of Symposia in Pure Mathematics}, pages 281--304. American Mathematical Society, Providence, RI, 2018.

\bibitem{KPX}
Kiran~S. Kedlaya, Jonathan Pottharst, and Liang Xiao.
\newblock Cohomology of arithmetic families of {$(\varphi,\Gamma)$}-modules.
\newblock {\em J. Amer. Math. Soc.}, 27(4):1043--1115, 2014.

\bibitem{Liu}
Ruochuan Liu.
\newblock Triangulation of refined families.
\newblock {\em Comment. Math. Helv.}, 90:831--904, 2015.

\bibitem{TLiu}
Tong Liu.
\newblock Lattices in filtered {$(\varphi,N)$}-modules.
\newblock {\em J. Inst. Math. Jussieu}, 11(3):659--693, 2012.

\bibitem{McDonald}
Vaughan McDonald.
\newblock Eigenvarieties over {CM} fields and trianguline representations, 2025.
\newblock arXiv:2504.18319 [math.NT].

\bibitem{P}
Vincent Pilloni.
\newblock Formes modulaires {$p$}-adiques de {Hilbert} de poids {$1$}.
\newblock {\em Invent. Math.}, 208(2):633--676, 2017.

\bibitem{Reduzzi-Xiao}
Davide~A. Reduzzi and Liang Xiao.
\newblock Partial {Hasse} invariants on splitting models of {Hilbert} modular varieties.
\newblock {\em Ann. Sci. {\'E}c. Norm. Sup{\'e}r.}, Ser. 4, 50(3):579--607, 2017.

\bibitem{Saito}
Takeshi Saito.
\newblock {Hilbert} modular forms and {$p$}-adic {Hodge} theory.
\newblock {\em Compos. Math.}, 145(5):1081--1113, 2009.

\bibitem{Skinner}
Christopher Skinner.
\newblock A note on the {$p$}-adic {Galois} representations attached to {Hilbert} modular forms.
\newblock {\em Doc. Math.}, 14:241--258, 2009.

\bibitem{Taylor89}
Richard Taylor.
\newblock On {Galois} representations associated to {Hilbert} modular forms.
\newblock {\em Invent. Math.}, 98:265--280, 1989.

\bibitem{Taylor91}
Richard Taylor.
\newblock {Galois} representations associated to {Siegel} modular forms of low weight.
\newblock {\em Duke Math. J.}, 63(2):281--332, 1991.

\bibitem{TianXiao}
Yichao Tian and Liang Xiao.
\newblock {$p$}-adic cohomology and classicality of overconvergent {Hilbert} modular forms.
\newblock {\em Ast{\'e}risque}, 382:73--162, 2016.

\bibitem{Wiles}
Andrew Wiles.
\newblock On ordinary {$\lambda$}-adic representations associated to modular forms.
\newblock {\em Invent. Math.}, 94(3):529--573, 1988.

\end{thebibliography}
\end{document}